\documentclass{amsart}
\usepackage{a4wide} 
\usepackage{amssymb, amsmath,mathtools}
\usepackage{mathrsfs}
\usepackage{amscd}
\usepackage{verbatim}

\usepackage{enumerate}
\usepackage{enumitem}

\usepackage[colorlinks,linkcolor={blue},citecolor={blue},urlcolor={red},]{hyperref}

\allowdisplaybreaks

\theoremstyle{plain}
\newtheorem{theorem}{Theorem}[section]
\theoremstyle{remark}
\newtheorem{remark}[theorem]{Remark}
\newtheorem{example}[theorem]{Example}
\theoremstyle{plain}
\newtheorem{corollary}[theorem]{Corollary}
\newtheorem{lemma}[theorem]{Lemma}
\newtheorem{proposition}[theorem]{Proposition}
\newtheorem{definition}[theorem]{Definition}

\newtheorem{assumption}[theorem]{Assumption}

\numberwithin{equation}{section}

\def\N{{\mathbb N}}
\def\Z{{\mathbb Z}}

\def\R{{\mathbb R}}
\def\C{{\mathbb C}}

\newcommand{\E}{{\mathbb E}}
\renewcommand{\P}{{\mathbb P}}
\newcommand{\F}{{\mathscr F}}

\newcommand{\CY}{C_Y}

\def\XXint#1#2#3{{\setbox0=\hbox{$#1{#2#3}{\int}$}
	\vcenter{\hbox{$#2#3$}}\kern-.5\wd0}}

	\newcommand{\rmd}{\mathrm{d}}

	\renewcommand{\Re}{\hbox{\rm Re}\,}
	\renewcommand{\Im}{\hbox{\rm Im}\,}

	\newcommand{\calL}{\mathcal{L}}
    
    \newcommand{\calO}{\mathcal{O}}
    \newcommand{\calP}{\mathcal{P}}

	\newcommand{\one}{{{\bf 1}}}

	\newcommand{\T}{\mathbb{T}}
    \newcommand{\dWQ}{\,\mathrm{d} W_Q}
	
	\newcommand{\ds}{\,\mathrm{d}s}
    \newcommand{\dt}{\,\mathrm{d}t}
	\newcommand{\dx}{\,\mathrm{d}x}
	\newcommand{\dy}{\,\mathrm{d}y}
    \newcommand{\dWH}{\,\mathrm{d}W_H}
	\newcommand{\dWHs}{\,\mathrm{d}W_H(s)}
	\newcommand{\la}{\langle}
	\newcommand{\ra}{\rangle}
    \newcommand{\nn}{|\!|\!|}
    \newcommand{\bignn}{\big|\!\big|\!\big|}
    \newcommand{\Nk}{N}

    \ExplSyntaxOn
\NewDocumentCommand{\maybeParen}{m}{
    \str_if_in:nnTF {#1} {+} {(#1)} {
        \str_if_in:nnTF {#1} {-} {(#1)} {#1}
    }
}
\ExplSyntaxOff

	\newcommand{\LHX}{{\calL_2(H,X)}}
	\newcommand{\LHY}{{\calL_2(H,Y)}}
    \newcommand{\LHZ}{{\calL_2(H,Z)}}

	\newcommand{\seq}{\subseteq}
	\newcommand{\ce}{\coloneqq}
	\newcommand{\ec}{\eqqcolon}
	\newcommand{\hra}{\hookrightarrow}
    \newcommand{\curve}{\curvearrowleft}
	\newcommand{\1}{\mathbf{1}}
	
	\renewcommand{\emptyset}{\varnothing}

    \newcommand{\grad}{{\nabla}}

    \newcommand{\Vlxr}{V_{t_\ell}^x(r)}
    \newcommand{\Vlxrtaun}{V_{t_\ell}^x(r\land \tau_n)}
    \newcommand{\Vlxt}{V_{t_\ell}^x(t)}
    \newcommand{\Vlxs}{V_{t_\ell}^x(s-t_\ell)}
    \newcommand{\Vlxstaun}{V_{t_\ell}^x((s-t_\ell)\land\tau_n)}
    \newcommand{\floors}{{\lfloor s \rfloor}}
    \newcommand{\floort}{{\lfloor t \rfloor}}

    \DeclareMathOperator{\PV}{P.V.}

	\allowdisplaybreaks
	
\begin{document}

\author{Katharina Klioba}
\address{Delft Institute of Applied Mathematics\\ TU Delft\\ Mekelweg 4\\
2628 CD Delft\\ the Netherlands.} \email{k.klioba-1@tudelft.nl}

\thanks{The author gratefully acknowledges support by the Alexander von Humboldt foundation through a Feodor
Lynen Research Fellowship. The author is also grateful for support from the VICI subsidy VI.C.212.027 of the Netherlands Organisation for Scientific Research (NWO)}

\date{September 8, 2026}
\title[Tamed exponential Euler schemes for superlinear hyperbolic SPDEs]{Strong convergence rates of tamed exponential Euler schemes for superlinear hyperbolic SPDEs}

\keywords{stochastic evolution equations, non-parabolic SPDEs, time discretization, tamed exponential Euler schemes, pathwise uniform strong convergence rate, superlinear nonlinearities}

\subjclass[2020]{Primary: 60H35, Secondary: 65C30, 60H15, 65J08, 47D06}

\begin{abstract}
In this paper, we prove pathwise uniform convergence at rates up to $1/2$ for tamed exponential Euler schemes for semilinear hyperbolic stochastic evolution equations with superlinearly growing nonlinearities and multiplicative noise. We take the
term hyperbolic to mean that the leading operator generates a contractive
$C_0$-semigroup but no parabolic smoothing occurs. Under local Lipschitz, polynomial growth, coercivity, and monotonicity conditions on the nonlinearities, we establish pathwise uniform strong error estimates of the form
\begin{equation*}
    \Big(\E\max_{0\le j \le N} \|U(t_j)-U^j\|_X^p\Big)^{1/p} \lesssim \sqrt{k}
\end{equation*}
on a Hilbert space $X$ for $p\in [2,\infty)$. Here, $U$ is the mild solution and $U^j$ is the tamed exponential Euler approximation at time $t_j=jk$ with step size $k>0$. This extends previous convergence results for non-parabolic SPDEs from globally to locally Lipschitz nonlinearities, allowing both drift and diffusion to grow polynomially. In a stochastic Kato framework, we further establish local and global well-posedness as well as uniform a priori estimates for the mild solution and its approximation. Applications to nonlinear stochastic transport, Airy, wave-type, and dissipatively damped nonlinear Schrödinger equations are included, covering different nonlinearities-stopped and fractionally tamed schemes. For the Klein--Gordon equation with cubic velocity damping, this complements previous results obtained for additive noise. 
\end{abstract}

\maketitle
\setcounter{tocdepth}{1}

\tableofcontents

\section{Introduction}
\label{sec:introduction}

In this article, we study the temporal approximation of semilinear stochastic evolution equations of the form
\begin{align}
\label{eq:introSEE}
    \rmd U + AU\dt  =F(U)\dt+G(U)\dWH\quad\text{ on }[0,T],     \qquad U(0)=u_0\in L^p(\Omega;X)
\end{align}
on a Hilbert space $X$ with $p\in[2,\infty)$ and a cylindrical Brownian motion $W_H$. Our focus lies on non-parabolic problems, i.e.\ problems for which the leading operator $-A$ generates a contractive but non-analytic $C_0$-semigroup
$(S(t))_{t\ge0}$, meaning that the semigroup does not regularise at positive times. This well-studied class of hyperbolic SPDEs includes wave-like, dispersive, and transport-type equations (see e.g.\ \cite{deBouardDebussche98KdV,deBouardDebussche03NLS,MasieroPriola17Wave}).

Our aim is to obtain strong convergence rates for the temporal approximation of \eqref{eq:introSEE}. Time discretisation schemes for non-parabolic SPDEs have gained significant interest in the last years; recent developments include structure-preserving schemes \cite{BLM21,BC23,OstermannSchratzEtAl24NLSsymplectic} or the influence of different geometries \cite{CohenDiGiovacchinoLang26waveSphere}. Many results in the literature are concerned with globally Lipschitz $F$ and $G$ in \eqref{eq:introSEE}. Convergence at rate $\frac12$ has been obtained for exponential Euler schemes for linear Schrödinger and Maxwell equations \cite{AC18,CCHS20} with multiplicative noise. Higher convergence rates up to $1$ can be attained for wave-type equations \cite{Wang15} or higher-order Milstein schemes \cite{KastnerKlioba26Milstein}. Up to a logarithmic correction factor, the same rates apply to rational schemes such as implicit Euler or Crank--Nicolson \cite{KliobaVeraar24Rate}. 
A useful analytical framework describing these equations is given by the stochastic counterpart of Kato's framework for hyperbolic evolution equations \cite{Kato75}, which uses a pair of  Hilbert spaces $X$ and $Y\hra X$ as a substitute for parabolic smoothing. Spatial regularity of $Y$ is encoded through the embedding $Y\hra D(A^\alpha)$ for some $\alpha\in (0,1]$ and yields decay via  semigroup increment estimates. This setting was used by Veraar and the author in \cite{KliobaVeraar24Rate} to obtain pathwise uniform convergence rates for a large class of time discretisation schemes for semilinear problems with $F$ and $G$ that are globally Lipschitz on $X$ and of linear growth on $Y$.

However, the situation becomes considerably more challenging if $F$ or $G$ are not globally Lipschitz and grow superlinearly. It is known for both parabolic \cite{Beccari2019strongweakdivergenceexponential} and hyperbolic \cite{Cui25JDE} SPDEs that explicit schemes may fail to satisfy uniform moment bounds, thereby preventing stability and leading to divergence of the numerical scheme. Nonetheless, such equations can be approximated if one modifies the nonlinearities within the scheme to limit their growth. Various approaches to achieve this exist in the parabolic literature 
\cite{BeckerJentzen19,Brehier22tamedExpEuler,GyongySabanisSiska16,JentzenPusnik19,LiuShen26tamed,
Wang20taming}, including stopping the nonlinearities at a certain threshold or
taming them continuously. Also for hyperbolic SPDEs, modified or implicit schemes have recently been investigated for specific equations such as the nonlinear Schrödinger equation \cite{BC23,Cui25JDE,CuiHong,OstermannSchratzEtAl24NLSsymplectic} and wave-type equations \cite{CaiCohenWang25,CuiHongSun25KleinGordon}.

\subsection{Setting and contributions}
The present work develops a convergence analysis as well as global well-posedness and stability at the intersection of
these two challenges, superlinear nonlinearities and absence of parabolic smoothing. We extend the stochastic Kato framework from \cite{KliobaVeraar24Rate} to locally Lipschitz nonlinearities of polynomial growth. More precisely, let $F:Y \to Y$ and $G:Y \to  \LHY$ be locally Lipschitz continuous w.r.t.\ the $X$-norm and of polynomial growth on $Y$, i.e.\ there are $\nu,\rho\ge 1$ such that for all $x,y\in Y$
    \begin{align}\tag{locLipX}
    \label{eq:introLocLip}
        \|F(x)-F(y)\|_X+\|G(x)-G(y)\|_\LHX &\lesssim \|x-y\|_X(1+\|x\|_Y^\nu+\|y\|_Y^\nu),\\
        \tag{polGrY}
        \label{eq:introPolGro}
        \|F(x)\|_Y +\|G(x)\|_\LHY &\lesssim 1+\|x\|_Y^\rho.
    \end{align} 
For temporal discretisation, we employ tamed exponential Euler schemes, which include the \emph{nonlinearities-stopped exponential Euler scheme} \cite{JentzenPusnik19} defined below. The error considered here is the {\em pathwise uniform strong error}
\begin{align} \label{eq:uniformerrorestintro}
\Big(\E \max_{j\in \{0, \ldots, N\}} \|U(t_j) - U^j\|_X^p\Big)^{1/p},
\end{align}
where the maximum over $j$ is inside the expectation. Compared to the {\em pointwise strong error}, where the maximum is outside of the expectation, this error provides information on convergence of the path. Moreover, if error estimates for \eqref{eq:uniformerrorestintro} hold for \textit{all} $p\in [2,\infty)$, a Borel--Cantelli argument yields almost sure convergence at rates smaller but arbitrarily close to the pathwise uniform convergence rate.

We establish pathwise uniform convergence at rates up to $\frac12$ for a class of tamed exponential Euler schemes under a coercivity condition on $Y$ relating the dissipation of $F$ with the growth of $G$ as well as a monotonicity assumption on $X$, also known as a one-sided Lipschitz condition. While coercivity and polynomial growth in the stronger $Y$-norm are needed for stability, the monotonicity required for convergence suffices in the space $X$, where local Lipschitz continuity holds. Our main contributions are as follows.
\begin{itemize}
    \item pathwise uniform strong convergence of the nonlinearities-stopped
    exponential Euler scheme with rate
    $\alpha\in(0,\frac12]$ for non-parabolic SPDEs with superlinear nonlinearities, with
    $L^p(\Omega)$-maximal estimates for all $p\in[2,\infty)$ rather than just mean-square estimates,
    \item a general treatment of different tamings, including nonlinearities-stopped exponential Euler, drift-only taming, optimal growth-dependent stopping, and fractional tamings,
    \item extension of Kato's framework to local and global well-posedness of a class of dissipative, non-parabolic stochastic evolution equations, where neither the drift nor the diffusion is assumed to be globally Lipschitz,
    \item novel pointwise strong and pathwise uniform stability estimates,
    \item first optimal pathwise uniform convergence rates for a nonlinear stochastic transport equation with superlinear multiplicative noise and a nonlinearly damped Klein--Gordon equation with multiplicative noise.
\end{itemize}

To make the above results applicable to implementable numerical schemes for SPDEs, an additional space discretisation is needed. Since the main novelty of our work lies in the treatment of temporal discretisations, we restrict our considerations to the latter. Space discretisation is usually performed by means of spectral Galerkin methods \cite{CohenDiGiovacchinoLang26waveSphere, JentzenRoeckner15_Milstein}, finite differences \cite{ACQ20, GyMi09} or finite elements \cite{aclw16,KLP20, Kruse}, e.g.\ in \cite{QiWang25galerkinWaveCubic} for a cubic wave equation. 

\subsection{Main results}

Up to the final time $T>0$, our aim is to approximate the mild solution
\begin{equation}
\label{eq:mildSolIntro}
    U(t)=S(t)u_0+\int_0^t S(t-s)F(U(s))\ds+\int_0^t S(t-s)G(U(s))\dWHs,\quad t \in [0,T],
\end{equation}
at the grid points $t_j=jk$, $0 \le j \le \Nk $, $\Nk \in \N$, where $k \ce \frac{T}{\Nk }>0$ is the time step.

The \emph{nonlinearities-stopped exponential Euler scheme} $(U^j)_{j=0,\ldots,\Nk}$ is defined by
    \begin{equation}
    \label{eq:introDefSchemeUj}
        U^{j+1} \ce S(k) U^j + kS(k) F_k(U^j) + S(k)G_k(U^j)\Delta W_{j+1},~0\le j \le \Nk -1,\quad U^0 \ce u_0,
    \end{equation}
    where $\Delta W_{j+1} \ce W_H(t_{j+1})-W_H(t_j)$, $F_k\ce a_k F$, $G_k\ce b_k G$ with $a_k,b_k:Y \to [0,1]$ given by 
    \begin{equation}
    \label{eq:akbkIndicatorIntro}
         a_k(U^j(\omega))\ce b_k(U^j(\omega)) \ce \1_{\{\|F(U^j(\omega))\|_Y +\|G(U^j(\omega))\|_{\calL_2(H,Y)} \le k^{-1/4}\}},\quad \omega\in\Omega,\,0 \le j \le \Nk.
    \end{equation}
This fully explicit scheme is a slight modification of the scheme investigated in \cite{JentzenPusnik19} for parabolic problems.  Our first main result establishes its pathwise uniform convergence at rates up to $\frac12$. 

\begin{theorem}
\label{thm:introMain}
    Let $H,X,Y$ be real separable Hilbert spaces such that $Y \hra X$ continuously and $Y\hra (X,D(A))_{\alpha,\infty}$ for some $\alpha \in (0,\frac{1}{2}]$. Suppose that $-A$ generates a $C_0$-contraction semigroup $(S(t))_{t\ge 0}$ on both $X$ and $Y$. Let $F:Y \to Y$ and $G:Y \to    \LHY$ be locally Lipschitz continuous on $X$ as in \eqref{eq:introLocLip} and of polynomial growth on $Y$ as in \eqref{eq:introPolGro}.
    Let $p\in [2,\infty)$, $\lambda \in (0,1)$, $q=\frac{p}{\lambda}$, and set $r=\rho q(1+\nu+4\alpha)$. Suppose that the coercivity condition
    \begin{align}\tag{coY}
    \label{eq:introCoYr}
        \la x,F(x)\ra_Y + \frac{r-1}{2}\|G(x)\|_\LHY^2 \lesssim 1+\|x\|_Y^2\qquad (x\in Y).
    \end{align}
    is satisfied as well as, for some $\eta>0$, the monotonicity condition
    \begin{equation}
    \tag{monX}
    \label{eq:introMonq}
	     \la x-y,F(x)-F(y)\ra_X + (1+\eta) \frac{q-1}{2} \|G(x)-G(y)\|_\LHX^2 \lesssim \|x-y\|_X^2\quad (x,y\in Y).
	\end{equation} 
    Then for all $u_0\in L^r(\Omega;Y)$ the nonlinearities-stopped exponential Euler scheme $(U^j)_{0\le j\le N}$ converges at rate $\alpha\in (0,\frac{1}{2}]$ to the solution $(U(t))_{t\in [0,T]}$ of \eqref{eq:introSEE} in a pathwise uniform sense. Moreover, there is a constant $C\ge 0$ independent of $k$, $N$, and $u_0$ such that 
    \begin{align}
    \label{eq:mainErrorEstIntro}
        \Big\|\max_{0\le j \le N} \|U^j-U(t_j)\|_X\Big\|_{L^p(\Omega)}
        &\le C\big(1+\|u_0\|_{L^r(\Omega;Y)}^{r/q} \big)k^\alpha.
    \end{align}    
\end{theorem}
This theorem is obtained as a special case of Theorem \ref{thm:convergenceMain}, which also applies to complex Hilbert spaces and includes an additional stopping parameter.
In the above, one usually takes $Y$ to be a suitable intermediate space between $X$ and $D(A)$, in particular, if $Y = D(A^{1/2})$ one can take $\alpha=\frac12$. For suboptimal rates $\alpha<\frac12$, \eqref{eq:mainErrorEstIntro} extends to a uniform estimate on the full interval $[0,T]$ for the piecewise constant interpolation $(U^k(t))_{t\in[0,T]}$ of the scheme. At the endpoint $\alpha=\frac12$, a square-root logarithmic correction factor occurs by Theorem \ref{thm:convergenceExtension}, that is,
\begin{align*}
    \bigg\|\sup_{t\in[0,T]}
    \|U^k(t)-U(t)\|_X\bigg\|_{L^p(\Omega)}
    \lesssim
    k^{1/2}\sqrt{\max\Big\{\log\Big(\frac{T}{k}\Big),p\Big\}}.
\end{align*}
This agrees with Lévy's modulus of continuity, the optimal rate achievable given only the Wiener increments as information from the underlying Brownian motion by \cite{Muller-Gronbach}. 

Clearly, convergence of the scheme in particular implies its stability. In Section \ref{sec:stability}, we show that pointwise strong stability and pathwise uniform stability already hold under weaker assumptions: \eqref{eq:introLocLip} and \eqref{eq:introMonq} are not required and \eqref{eq:introCoYr} suffices with parameter $p$ or $q$, respectively, for stability of $p$-th moments. Global well-posedness in the sense of mild solutions is ensured by Theorem \ref{thm:aprioriY}, which also provides pathwise uniform a priori estimates for $U$.

Our second main convergence result concerns different ways to tame the nonlinearities in the scheme than stopping them at the threshold $k^{-1/4}$. The class of \textit{tamed exponential Euler schemes} is given by \eqref{eq:introDefSchemeUj} with general scalar tamings $a_k,b_k:Y\to [0,1]$, which need not take the form of \eqref{eq:akbkIndicatorIntro}. Choosing the indicator functions \eqref{eq:akbkIndicatorIntro} as $a_k,b_k$, the nonlinearities-stopped exponential Euler scheme is included as a special case. This more general class converges, too, at rates up to $\frac12$ in a pathwise uniform sense. 

\begin{theorem}[Convergence rate for tamed exponential Euler schemes]
\label{thm:introMainTamed}
    Let the assumptions of Theorem \ref{thm:introMain} hold for $r=(\beta+\rho\nu)q$, where $\beta\ge \rho$ is such that the consistency condition
     \begin{equation*}
        (1-a_k(x))\|F(x)\|_Y+(1-b_k(x))\|G(x)\|_\LHY \lesssim k^\alpha\big(1+\|x\|_Y^\beta\big),\quad x\in Y.
    \end{equation*}
    holds for all $k\in (0,T]$ and $x\in Y$. Further, suppose that $0\le b_k^2 \le a_k\le 1$ and $\theta_F,\theta_G\in [0,\frac12]$ satisfy $\theta_F+\theta_G\le \frac12$ as well as
    \begin{align*}
        a_k(x)\|F(x)\|_Y \lesssim k^{-\theta_F}(1+\|x\|_Y),\quad
        b_k(x)\|G(x)\|_\LHY \lesssim k^{-\theta_G}(1+\|x\|_Y),\qquad x\in Y.
    \end{align*}
    Then for all $u_0\in L_{\F_0}^{r}(\Omega;Y)$, the tamed exponential Euler scheme $(U^j)_{0\le j\le N}$ converges to the solution $(U(t))_{t\in [0,T]}$ of \eqref{eq:SEEsemilinear} at pathwise uniform rate $\alpha$ and there is a constant $C\ge 0$ independent of $k$, $N$, and $u_0$ such that
     \begin{align*}
        \Big\|\max_{0\le j \le N} \|U^j-U(t_j)\|_X\Big\|_{L^p(\Omega)} 
        &\le C\big(1+\|u_0\|_{L^r(\Omega;Y)}^{r/q} \big)k^\alpha.
    \end{align*}
\end{theorem}
Theorem \ref{thm:introMainTamed} follows from Theorem \ref{thm:convergenceMainTamed} and Remark \ref{rem:tamingInTermsOfakbk} as a special case. In particular, Theorem \ref{thm:introMainTamed} applies to the \emph{fractional taming} given by
\begin{align*}
    a_k(x)\ce b_k(x)\ce \frac{1}{1+k^{1/2} R(x)^2},\quad R(x)\ce\|F(x)\|_Y+\|G(x)\|_\LHY,
\end{align*}
which tames smoothly rather than stopping discontinuously.
The parameter $\beta$ in the consistency condition directly influences the moment and coercivity parameter via $r=q(\beta+\rho\nu)$ as opposed to $r=\rho q(1+\nu+4\alpha)$ before. This enables a comparison of different taming strategies and an optimisation of the stopping parameters for nonlinearities stopped differently via
\begin{align*}
    a_k(x) \ce \1_{\{\|F(x)\|_Y\le k^{-\theta_F}\}},\qquad
    b_k(x)\ce \1_{\{\|F(x)\|_Y\le k^{-\theta_F},\ \|G(x)\|_\LHY\le k^{-\theta_G}\}}.
\end{align*}
Optimal choices for $\theta_F,\theta_G$ are given in Corollary \ref{cor:optimalSeparateTaming} for the case of drift $F$ and noise $G$ with different growth parameters. When the noise is globally Lipschitz and satisfies a linear growth condition, it may be left untamed, allowing to take larger $\theta_F$ resulting in smaller $r$ for drift-only tamings. In all cases, lower values of $r$ translate to lower moment requirements on the initial value as well as, if present, a less restrictive smallness condition on the noise.

Applications to various nonlinear hyperbolic SPDEs are included in Section \ref{sec:applications}. Convergence at the optimal rate $\frac12$ is achieved for a transport equation with cubic Allen--Cahn nonlinearity and small quadratic noise, a nonlinear Schrödinger equation with dissipative damping, and a Klein--Gordon equation with velocity-dependent damping. The latter extends the mean-square convergence rates for implicit schemes in the additive noise case for  damped wave equations \cite{CuiHongSun25KleinGordon,CaiCohenWang25} to pathwise uniform convergence rates for explicit schemes with multiplicative noise. To the best of the author's knowledge, pathwise uniform strong convergence rates for explicit time discretisation schemes in this superlinear setting with multiplicative noise are novel for the specific equations listed above, in particular for non-globally Lipschitz noise. They align well with well-established results and recent developments for closely related equations such as the standard nonlinear Schrödinger equation \cite{deBouardDebussche06NLSconvOrder,ChenDangHong24,Cui25JDE} and the KdV equation \cite{CuiDAmbrosioDiGiovacchinoSun26}. 

Lastly, an application to the Airy equation is included as a third-order equation obtained as a linearisation of KdV. The resulting convergence rate is $\frac{1}{3}$, reflecting the coercivity available in $Y=H^1$ paired with the equation's higher order. Future work might address whether a weakened coercivity condition allows for optimal convergence rates. Other possible generalisations concern the analysis of rational tamed schemes like a tamed implicit Euler scheme or tamed higher-order schemes like a tamed Milstein scheme in analogy to the globally Lipschitz case considered in \cite{KliobaVeraar24Rate} and \cite{KastnerKlioba26Milstein}, respectively.

\subsection{Method of proof}
The error analysis is performed for a suitable continuous-time extension $(\hat{U}(t))_{t\in[0,T]}$ of the scheme $(U^j)_{0\le j\le N}$. Its difference from the mild solution is then split with the help of the auxiliary process $\bar{U}$, which differs from the mild solution in that in the variation-of-constants formula, it evaluates $F$ and $G$ at the value $U^j$ of the scheme at the last grid point rather than the solution $U(s)$ at the temporal integration parameter $s$.

The term $\hat{U}-\bar{U}$ thus consists of the error caused by stopping the nonlinearities $F,G$ via $F_k,G_k$ as well as the local discretisation error caused by the difference of the semigroup at grid points and at intermediate times. While Markov's inequality can be used to control the stopping error, semigroup estimates combined with polynomial growth as in \eqref{eq:introPolGro} reduce the local discretisation error to higher moment estimates of the scheme. This is advantageous over using the growth threshold from the stopping directly, as the latter would deteriorate the achievable rate by the stopping parameter $\frac{1}{4}$ to $\frac{1}{2}-\frac{1}{4}=\frac{1}{4}$.

Monotonicity as in \eqref{eq:introMonq} is essential to control the remaining error term $\bar{U}-U$ after applying an Itô-type inequality for mild solutions that combines Itô's formula with the dissipativity of the generator of the semigroup. For the remaining terms, local Lipschitz continuity \eqref{eq:introLocLip} produces products of higher moments of the scheme and the auxiliary process as well as the difference of the scheme with $\bar{U}$. This difference can be split into the already controlled error $\hat{U}-\bar{U}$ and increments over small intervals. A stochastic Gronwall argument finally closes the estimate.

An extension to the full time interval $[0,T]$ follows by combining the error estimate with path regularity of the mild solution in suitable generalised Hölder spaces, which is obtained as a consequence of a logarithmic square function estimate.

To generalise convergence rates to other scalar tamings $a_k,b_k$, three properties of the tamed nonlinearities have to be verified: Preservation of coercivity, suitable growth bounds depending on the step size, and consistency of order $\alpha$ with the untamed nonlinearities. These replace arguments specific to the nonlinearities-stopped scheme, such as the use of Markov's inequality.

For the required moment bounds of the mild solution, pathwise uniform a priori estimates in $Y$ are established in Theorem \ref{thm:aprioriY} via the coercivity assumption \eqref{eq:introCoYr} and the stochastic Gronwall inequality, which is the key ingredient to avoid imposing a linear growth assumption on the noise. This builds on the local well-posedness result from Theorem \ref{thm:localWellPosed}, which relies on a reduction to the globally Lipschitz case via suitable projections, a compatibility argument as well as explosion times. Establishing pointwise and pathwise uniform moment bounds for the scheme of the form
\begin{align*}
    \sup_{N\in\N}\max_{0\le j\le N} \E\|U^j\|_Y^q \le C \qquad\text{ and } \qquad\sup_{N\in\N} \E \Big[\max_{0\le j\le N}\|U^j\|_Y^q\Big] \le C
\end{align*}
in Propositions \ref{prop:stabY} and \ref{prop:pathwiseUniformStabY}, respectively, essentially relies on the stopping or taming: Together with coercivity, it ensures a one-step stability estimate for an auxiliary process that does not contain the action of the semigroup. Combining it with contractivity of the semigroup and iterating results in pointwise stability. For the pathwise uniform stability estimates, a stronger coercivity condition allows controlling the contribution of the martingale via a Burkholder--Davis--Gundy argument.

\subsection*{Acknowledgements} The author wishes to thank Sonja Cox, Sarah Geiss, and Mark Veraar for inspiring and helpful discussions.

\section{Preliminaries}
\label{sec:preliminaries}

\subsection*{Notation}
Throughout the paper, we fix a final time $T>0$ and a filtered probability space $(\Omega, \F, (\F_t)_{t\in[0,T]},\P)$ satisfying the usual conditions. Denote the progressive $\sigma$-algebra on $\Omega \times[0,T]$ associated with $(\F_t)_{t\in[0,T]}$ by $\calP$ and indicate progressively measurable subspaces by the subscript $\calP$. Moreover, $H$, $X$, and $Y$ denote separable Hilbert spaces, where $H$ is used to define the $(\F_t)_{t \in [0,T]}$-cylindrical Brownian motion $W_H$. Whenever a stochastic process defined on a random interval is multiplied by an indicator of the form $\1_{\Gamma\times[0,\tau)}$ for $\Gamma\in\F_0$ and $\tau$ a stopping time, we understand it as its extension by zero outside $\Gamma\times[0,\tau)$. The space of Hilbert--Schmidt operators from $H$ to $X$ is denoted by $\calL_2(H,X)$. Subsequently, we consider a uniform time grid with $t_j = jk$, where $k>0$ is the time step and $j=0, \ldots, N$ with $N = T/k \in \N$, and define $\lfloor t \rfloor \ce \max\{t_j:\,t_j\le t\}$ for $t \in [0,T]$. By $(S(t))_{t\geq 0}$, we denote a $C_0$-semigroup. For a given evolution equation, $(U(t))_{t\in [0,T]}$ is the exact solution and $U^j$ the numerical solution approximating $U$ at time $t_j$ for $j=0,\ldots,N$. For $f$ and $g$ in the respective spaces and $p\in [2,\infty)$, let $\|f\|_{p,Z} \ce \|f\|_{L^p(\Omega;Z)}$ and $\nn g\nn_{p,Z} \ce \|g\|_{L^p(\Omega;\calL_2(H,Z))}$. Throughout, we use generic constants $C\ge 0$, whose value may change from line to line and which are indexed by parameters they depend on where suitable. The notation $a \curve b$ indicates that a result depending on the parameter $a$ is applied with $a$ replaced by $b$.  

\subsection{Semigroups and interpolation}

Common choices for the spaces $Y$ on which the scheme converges at a certain rate are domains of fractional powers of $A$. An important property of these spaces is that they embed into the real interpolation spaces with parameter $\infty$, i.e., for $\alpha\in (0,1)$, $D(A^{\alpha}) \hookrightarrow D_A(\alpha, \infty)$.
Here, $D_A(\alpha,\infty)$ denotes the real interpolation space $(X,D(A))_{\alpha,\infty}$. See \cite{Lun,Tr1} for details on interpolation spaces.
Such embeddings and properties of $D_A(\alpha,\infty)$ allow us to obtain decay rates for differences of contraction semigroups as in \cite[Formula~(2.4)]{KliobaVeraar24Rate}. 

\begin{definition}
\label{def:quasiContractive}
    A $C_0$-semigroup $(S(t))_{t \ge 0}$ on $X$ is called \emph{quasi-contractive with parameter $\lambda\ge 0$} if $\|S(t)\|_{\calL(X)} \le e^{\lambda t}$ for all $t \ge 0$ and \emph{contractive} if this holds with $\lambda=0$.
\end{definition}

We say that $-A$ generates a $C_0$-contraction semigroup on some $Y \subseteq X$, if the part of $A$ in $Y$ generates a $C_0$-contraction semigroup on $Y$.

\begin{lemma}
    \label{lem:sgInterpolation}
    Let $X$ and $Y$ be separable Hilbert spaces such that $Y \hra X$ continuously.  
    Let $-A: D(A)\seq X\to X$ be the generator of a $C_0$-contraction semigroup $(S(t))_{t \ge 0}$ on both $X$ and $Y$. 
    Let $\alpha \in (0,1]$ and suppose that $Y \hookrightarrow D_A(\alpha,\infty)$ continuously if $\alpha \in (0,1)$ or $Y \hookrightarrow D(A)$ continuously if $\alpha=1$, with embedding constant $\CY\ge 0$. Then, we have
    \begin{equation*}
        \|S(t)-S(s)\|_{\calL(Y,X)} \leq 2\CY (t-s)^\alpha,\qquad 0 \le s \le t.
    \end{equation*}
\end{lemma}
\begin{proof}
    The triangle inequality and contractivity of the semigroup imply
    $\|S(t)-S(s)\|_{\calL(X)}\le 2$ for $0\le s \le t \le T$. Since
    \begin{equation*}
        \|(S(t)-S(s))x\|_X = \Big\|-\int_s^t AS(r)x\,\rmd r\Big\|_X\le \int_s^t \|S(r)Ax\|_X\,\rmd r \le (t-s)\|x\|_{D(A)}
    \end{equation*}
    for $x\in D(A)$, it holds that $\|S(t)-S(s)\|_{\calL(D(A),X)}\le (t-s)$, i.e.\ the claim for $\alpha=1$.
    Real interpolation and the embedding $Y \hra D_A(\alpha,\infty)$ imply the claim for $\alpha \in (0,1)$ via
    \begin{equation*}
        \|S(t)-S(s)\|_{\calL(Y,X)}\le \CY\|S(t)-S(s)\|_{\calL(D_A(\alpha,\infty),X)}\le 2^{1-\alpha}\CY(t-s)^\alpha. \qedhere
    \end{equation*}
\end{proof}

\subsection{Stochastic integration}\label{sec:stochInt}

We denote by $\LHX$ the space of all Hilbert--Schmidt operators from $H$ to $X$, that is, all bounded operators $R:H \to X$ such that
\begin{equation*}
    \|R\|_{\LHX}^2 \ce \sum_{n \in \N}\|Re_n\|_X^2 <\infty,
\end{equation*}
where $(e_n)_{n \in \N}$ is an orthonormal basis of the separable Hilbert space $H$. 
For $R\in \LHX$, $(e_n)_{n \in \N}$ as before, and $\gamma = (\gamma_n)_{n\geq 1}$ centered i.i.d.\ normally distributed random variables, we define
\begin{equation}\label{eq:convradonW}
R \gamma = \sum_{n\geq 1} \gamma_n R e_n,
\end{equation}
where the convergence is in $L^p(\Omega;X)$ for $p < \infty$ and almost surely by \cite[Corollary 6.4.12]{AnalysisBanachSpacesII}.

The stochastic integrals appearing in expressions such as \eqref{eq:mildSolIntro} are taken with respect to an $H$-cylindrical Brownian motion to allow for $\calL_2(H,X)$-valued integrands. An \emph{$H$-cylindrical Brownian motion} is a mapping $W_H:L^2(0,T;H) \to L^2(\Omega)$ such that
\begin{enumerate}[label=(\roman*)]
    \item $W_H b$ is Gaussian for all $b \in L^2(0,T;H)$,
    \item $\E(W_H b_1 \cdot W_H b_2) = \langle b_1,b_2 \rangle_{L^2(0,T;H)}$ for all $b_1, b_2 \in L^2(0,T;H)$,
    \item $W_H b$ is $\F_t$-measurable for all $b\in L^2(0,T;H)$ with support in $[0,t]$,
    \item $W_H b$ is independent of $\F_s$ for all $b\in L^2(0,T;H)$ with support in $[s,T]$,
\end{enumerate}
where we include a complex conjugate on $W_H b_2$ in case we want to use a complex $H$-cylindrical Brownian motion.
For $h \in H$ and $t \in [0,T]$, we use the shorthand notation $W_H(t)h \ce W_H(\1_{(0,t)} \otimes h)$. Hence, $(W_H(t)h)_{t \in[0,T]}$ is a real-valued Brownian motion for each fixed $h \in H$, which is standard if and only if $\|h\|_H=1$, and $(W_H(t))_{t\ge 0}$ itself is a real-valued Brownian motion in the special case $H=\R$.

We refer to an $H$-valued stochastic process $(W(t))_{t \geq 0}$ as a \emph{$Q$-Wiener process} if $W(0)=0$, $W$ has continuous trajectories and independent increments, and $W(t)-W(s)$ is normally distributed with parameters $0$ and $(t-s)Q$ for $t\geq s \geq 0$. The operator $Q$ is in $\calL(H)$, positive self-adjoint, and of trace class. One can show that $W$ is a $Q$-Wiener process if and only if there exists an $H$-cylindrical Brownian motion $W_H$ such that $Q^{1/2}W_H(t)\ce\sum_{n\geq 1} Q^{1/2} h_n W_H(t) h_n = W(t)$ for an orthonormal basis $(h_n)_{n \ge 1}$ of $H$, where we used the notation \eqref{eq:convradonW}. To consider an equation such as \eqref{eq:introSEE} with a $Q$-Wiener process $W$ instead of a cylindrical Brownian motion, one can reduce to the cylindrical case by replacing $G$ with $G Q^{1/2}$. For further properties of $H$-cylindrical Brownian motions, $Q$-Wiener processes and the Itô integral, we refer to \cite{DaPratoZabczyk14}.

A central estimate for stochastic integrals is the following maximal inequality for stochastic convolutions with a quasi-contractive semigroup.

\begin{theorem}
    \label{thm:maximal-inequality}
    Let $X$ be a Hilbert space and let $(S(t))_{t\ge 0}$ be a quasi-contractive semigroup with parameter $\lambda\ge 0$. Then for all $p\in [2,\infty)$ and progressively measurable $g\in L^2(0,T;L^p(\Omega;\LHX))$
    \begin{align*}
        \bigg\| \sup_{t\in[0,T]} \Big\| \int_0^t S(t-s)g(s) \dWHs \Big\|_X \bigg\|_{L^p(\Omega)}
        &\le e^{\lambda T} C_p \|g\|_{L^p(\Omega;L^2(0,T;\LHX))}\\
        &\le e^{\lambda T} C_p \|g\|_{L^2(0,T;L^p(\Omega;\LHX))}.
    \end{align*}
\end{theorem}

If the semigroup considered is the identity, the classical \emph{Burkholder--Davis--Gundy inequalities}, or BDG inequality for short, are recovered. In this case, the inequality remains valid in the absence of a supremum, and is referred to as \emph{Itô's isomorphism}.
\begin{proof}
    The contractive case with the norm of $g$ in $L^p(\Omega;L^2(0,T;\LHX))$ follows from \cite{HausSei}. Via a scaling argument, this can be extended to the quasi-contractive case. Lastly, noting that $\frac{p}{2}\ge 1$, Minkowski's integral inequality proves the second inequality.
\end{proof}

We also use the following localised version of the BDG inequality (see \cite[Lem.~5.4, Thm.~5.5]{vNVW07stochasticIntegrationUMD}).
\begin{proposition}[Localised BDG inequality]
\label{prop:localizedBDG}
For $n\in \N$, let $\phi_n,\phi\in L^0(\Omega;L^2(0,T;H))$ be progressively measurable. If $\phi_n\to\phi$ in $L^0(\Omega;L^2(0,T;H))$, then
\begin{align*}
    \int_0^\cdot\phi_n(s)\dWHs \xrightarrow[n\to\infty]{} \int_0^\cdot\phi(s)\dWHs 
    \quad\text{in }L^0(\Omega;C([0,T])).
\end{align*}
\end{proposition}

\subsection{A stochastic Gronwall lemma}

Besides the classical deterministic Gronwall inequalities, we use the following stochastic Gronwall inequality from \cite[Cor.~5.4]{Geiss24}, which allows for an additional martingale term on the right-hand side and gives maximal estimates with an arbitrarily small moment loss for $\lambda \approx 1$.

\begin{lemma}
\label{lem:stochasticGronwall}
    Suppose that $\varphi:\Omega\times[0,\infty)\to[0,\infty)$ is a.s.\ continuous and adapted, $A:\Omega\times[0,\infty)\to[0,\infty)$ is a.s.\ continuous, increasing, adapted, and $A_0=0$, $H:\Omega\times[0,\infty)\to[0,\infty)$ is a.s.\ continuous, increasing, and adapted, and $M:\Omega\times[0,\infty)\to \R$ is a continuous local martingale with $M_0=0$. Further, suppose that for all $t\ge 0$ a.s.\ 
    \begin{equation*}
        \varphi_t \le \int_0^t \varphi_s\,\rmd A_s + M_t+H_t.
    \end{equation*}
    Then for all $\lambda\in (0,1)$ and $T> 0$,
    \begin{equation*}
        \bigg\|e^{-A_T}\sup_{t\in [0,T]} \varphi_t \bigg\|_{L^\lambda(\Omega)} \le \frac{(1-\lambda)^{-1/\lambda}}{\lambda} \|H_T\|_{L^\lambda(\Omega)}.
    \end{equation*}
\end{lemma}

\section{Local well-posedness}
\label{sec:localWP}
	
	We consider the semilinear stochastic evolution equation
	\begin{equation}
		\label{eq:SEEsemilinear}
		\rmd u + A u\dt = F(u) \dt + G(u) \dWH\text{ on }[0,T], \quad u(0)=u_0
	\end{equation}
	on a Hilbert space $X$. Here, $T >0$ denotes the final time and $-A$ generates a contractive $C_0$-semigroup $(S(t))_{t\ge 0}$. We consider initial values $u_0 \in L^p(\Omega;Y)$ for some $p \in [2,\infty)$. 

    To show local well-posedness up to a stopping time, the following conditions on the nonlinearities, the semigroup, and the underlying spaces are imposed. The two-space setup is a stochastic Kato setting, which is an adaptation of the setting introduced by Kato to study hyperbolic PDEs in \cite{Kato75} and generalises the setting introduced in \cite{KliobaVeraar24Rate} for globally Lipschitz nonlinearities.
	\begin{assumption}
		\label{ass:localWP}
		Let $H$ be a real separable Hilbert space and $X,Y$ be real or complex separable Hilbert spaces such that $Y \hra X$ continuously. Suppose that $-A$ generates a $C_0$-contraction semigroup on both $X$ and $Y$. Let $F:Y \to Y$ and $G:Y \to \LHY$ be Borel-measurable. Denote by 
    	\begin{equation*}
    	    B_n \ce B_n^Y(0) \ce \{ y \in Y: \|y\|_Y \le n\}
    	\end{equation*}
        the closed balls of radius $n\in\N$ in the $Y$-norm around the origin. Suppose that for all $n\in\N$ the following holds.
		\begin{enumerate}[label=(\alph*)]
			\item Local Lipschitz continuity of $F$ and $G$: There is a constant $C_n\ge 0$ such that 
            \begin{equation*}
                \|F(x)-F(y)\|_X+\|G(x)-G(y)\|_\LHX \le C_n\|x-y\|_X\text{ for all }x,y \in B_n.
            \end{equation*}
            \item Local $Y$-boundedness of $F$ and $G$: 
            \begin{equation*}
                L_n \ce \sup_{x \in B_n} \big(\|F(x)\|_Y+\|G(x)\|_\LHY\big) < \infty.
            \end{equation*} 
		\end{enumerate}
	\end{assumption}

    Local well-posedness is shown in four steps: We construct approximations of the local solution via a projection argument reducing to the globally Lipschitz setting where global well-posedness is known in Subsection \ref{subsec:localReduction}. Based on this reduction, local uniqueness is established in Subsection \ref{subsec:localUniqueness}. Then, blow-up criteria are established with the help of explosion times in Subsection \ref{subsec:blowupLocWP}. A limiting procedure finally gives existence and uniqueness of the local solution in Theorem \ref{thm:localWellPosed}.
	
	\subsection{Reduction to the globally Lipschitz setting}
	\label{subsec:localReduction}

    Projections onto the balls $B_n$ allow us to reduce to the globally Lipschitz case, where global well-posedness is known. Before proving this reduction step, we define the notion of maximal local solutions.
	\begin{definition}
		\label{def:localSolution}
		For a stopping time $\sigma$ taking values in $[0,T]$, a \emph{localizing sequence} for $\sigma$ is an increasing sequence of stopping times $(\sigma_n)_{n\in\N}$ such that $\sigma_n\nearrow\sigma$ as $n\to \infty$ and $\sigma_n <\sigma$ on $\{\sigma<T\}$. Define 
        \begin{equation*}
            \Gamma \times [0,\sigma) \ce \{(\omega,t)\in\Gamma \times[0,T]:0 \le t <\sigma(\omega)\}
        \end{equation*}
        for $\Gamma \in \F$. Further, write $(u(t))_{t\in [0,\sigma)}$ for an $Y$-valued process $u: \Omega \times [0,\sigma) \to Y$ such that $\{\omega \in \Omega:t <\sigma(\omega)\}\ni \omega \mapsto u(\omega,t)$ is $\F_t$-measurable for all $t\in [0,T]$ and $[0,\sigma(\omega))\ni t \mapsto u(\omega,t)\in Y$ is continuous for almost every $\omega \in \Omega$. Such an $Y$-valued process $(u(t))_{t\in [0,\sigma)}$, abbreviated by $(u,\sigma)$, is called a \emph{local solution} of \eqref{eq:SEEsemilinear} if $\sigma \in (0,T]$ almost surely and there exists a localizing sequence $(\sigma_n)_{n\in\N}$ for $\sigma$ such that 
		\begin{enumerate}[label=(\roman*)]
			\item for all $t\in [0,T]$, $s\mapsto S(t-s)F(u(s))\1_{[0,\sigma_n]}(s) \in L^0(\Omega;L^1(0,t;Y))$, 
			\item for all $t\in [0,T]$, $s\mapsto S(t-s)G(u(s))\1_{[0,\sigma_n]}(s) \in L^0(\Omega;L^2(0,t;\LHY))$,
			\item almost surely for all $t\in [0,\sigma_n]$,
			\begin{equation*}
			u(t) = S(t)u_0 + \int_0^t S(t-s)F(u(s))\ds + \int_0^t S(t-s)G(u(s))\dWHs.
			\end{equation*}
		\end{enumerate}
		A local solution $(u(t))_{t\in [0,\sigma)}$ is called \emph{locally unique} if for any other local solution $(v(t))_{t\in [0,\nu)}$, almost surely $v|_{[0,\sigma \land \nu)} \equiv u|_{[0,\sigma \land \nu)}$.
		A local solution $(u(t))_{t\in [0,\sigma)}$ is called \emph{maximal} if for any other local solution $(v(t))_{t\in [0,\nu)}$, almost surely we have $\nu \le \sigma$ and $v \equiv u|_{[0,\nu)}$. 
	\end{definition}
	Clearly, a maximal local solution is always unique. Moreover, it follows from the definition that $u(\cdot \land \sigma_n) \in L_\calP^0(\Omega;C([0,T];Y))$ for all $n \in \N$. 

\begin{lemma}
\label{lem:locSolun}
    Suppose that Assumption \ref{ass:localWP} holds. Let $p \in [2,\infty)$, $u_0\in L_{\F_0}^p(\Omega;Y)$, and $\Gamma_n \ce \{\|u_0\|_Y \le \frac{n}{2}\}\seq \Omega$. Then, for every $n\in\N$, there exists an adapted process $u_n\in L^p(\Omega;C([0,T];Y))$ such that $u_n$ solves \eqref{eq:SEEsemilinear} on $\Gamma_n\times[0,\sigma_n]$ with the stopping time defined by 
    \begin{equation}
    \label{eq:defSigmaN}
        \sigma_n \ce \inf\{t\in [0,T]:\|u_n(t)\|_Y\geq n\},\quad \inf \emptyset \ce T.
    \end{equation}
    That is, for almost all $\omega \in \Gamma_n$ and all $t \in [0,\sigma_n(\omega)]$,
    \begin{equation*}
            u_n(t) = S(t)u_0+\int_0^t S(t-s)F(u_n(s))\ds
            + \int_0^t S(t-s)G(u_n(s))\dWHs.
    \end{equation*}
\end{lemma}
Here and in the following, unless stated otherwise, we use the convention $\inf \emptyset \ce T$.
\begin{proof}
    Since $Y$ is reflexive as a Hilbert space and $Y \hra X$ continuously, the balls $B_n$ are not only bounded in $Y$ but also closed in $X$ by \cite[Lemma~7.3]{Kato75}. Moreover, they are convex. Thus, we can find uniformly $1$-Lipschitz continuous projections $P_n:X \to X$ onto $B_n$ for $n \in \N$ by \cite{Phelps58convexSets} (see also \cite[Thm.~2.46]{FletcherMoors15unifLipProj}). Note, this requires $X$ to be a Hilbert space as well as reflexivity of $Y$ and $P_n$ are in general neither orthogonal projections nor linear mappings. 
	For $n \in \N$, define 
    \begin{equation*}
        F_n:X \to X,\quad F_n(x) \ce F(P_n(x)),\qquad G_n:X \to \LHX,\quad G_n(x) \ce G(P_n(x)).
    \end{equation*}
    Since $P_n(x)\in Y$ for all $x\in X$, these mappings are well-defined. Separability of $X$ and $Y$ and the continuity of the embedding $Y\hra X$ ensure Borel measurability of $P_n:X\to Y$ as a $Y$-valued mapping. Local Lipschitz continuity of $F$ together with Lipschitz continuity of $P_n$ implies global Lipschitz continuity of $F_n:X \to X$ for each $n \in \N$ via
	\begin{equation*}
	    \|F_n(x)-F_n(y)\|_X = \|F(P_n(x))-F(P_n(y))\|_X \le C_n\|P_n(x)-P_n(y)\|_X \le C_n\|x-y\|_X
	\end{equation*}
	for all $x,y \in X$. Similarly, local $Y$-boundedness of $F$ implies that $F_n$ is uniformly bounded on $Y$ and, in particular, of linear growth on $Y$ due to
	\begin{equation*}
	    \|F_n(x)\|_Y \le L_n \le L_n(1+\|x\|_Y)
	\end{equation*}
	for all $x\in Y$. Likewise, $G_n$ is globally Lipschitz on $X$ and of linear growth on $Y$.
    For $n \in \N$, we truncate the initial values by setting 
    \begin{equation*}
        u_{0,n} \ce u_0 \1_{\Gamma_n},\quad \Gamma_n \ce \Big\{\omega \in \Omega:\|u_0(\omega)\|_Y \le \frac{n}{2}\Big\}
    \end{equation*}
    such that $u_{0,n} \in L^\infty(\Omega;Y)$.

    Hence, for $n \in \N$ and $p \in [2,\infty)$, \cite[Thm.~4.4]{KliobaVeraar24Rate} provides the existence of unique global solutions $u_n\in L^p(\Omega;C([0,T];Y))$ to the equation 
	\begin{equation}
    \label{eq:SEEprojected} 
		\rmd u_n + A u_n\dt = F_n(u_n) \dt + G_n(u_n) \dWH \quad u_n(0)=u_{0,n}.
	\end{equation}
	By construction, $F$ and $F_n$ agree on $B_n$, as do $G$ and $G_n$. 
    Furthermore, $u_0=u_{0,n}$ on $\Gamma_n$. 
    By definition of $\sigma_n$, for almost all $\omega \in \Gamma_n$, we have $u_n(\omega,t)\in B_n$ for all $t\in [0,\sigma_n(\omega)]$, and therefore the mild solution $u_n$ to \eqref{eq:SEEprojected} satisfies the claimed mild formulation with $F$ and $G$ instead of $F_n$ and $G_n$, respectively, for all $(\omega,t)\in \Gamma_n\times[0,\sigma_n]$.
\end{proof}

	\subsection{Local uniqueness}
    \label{subsec:localUniqueness}
    
    Localising to bounded regions on which the nonlinearities are globally Lipschitz, the stopped local solutions become fixed points of the truncated fixed point functional. Iterating this procedure and taking the limit of the localisation parameter gives local uniqueness.
    
	\begin{lemma}[Local uniqueness]
		\label{lem:localUnique}
		  Suppose that Assumption \ref{ass:localWP} holds and let $p \in [2,\infty)$. Let $(u_1,\sigma^1)$ and $(u_2,\sigma^2)$ be local solutions of \eqref{eq:SEEsemilinear} with initial values $u_0^1,u_0^2 \in L_{\F_0}^p(\Omega;Y)$, respectively. 
		Then $u_1|_{[0,\sigma^1 \land \sigma^2)} = u_2|_{[0,\sigma^1 \land \sigma^2)}$ almost surely on $\Gamma\ce \{u_0^1=u_0^2\}\seq\Omega$.
	\end{lemma}
	\begin{proof}
		Define $\sigma \ce \sigma^1 \land \sigma^2$ and for $j=1,2$ let $(\sigma_n^j)_{n\in\N}$ be localizing sequences for $\sigma^j$. Further, define $\rho_n^j \ce \inf\{t \in [0,\sigma_n^j]: \|u_j(t)\|_Y \ge n\}$ for $n \in \N$, $j=1,2$, where we set $\inf \emptyset \ce \sigma_n^j$, and let $\rho_n \ce \rho_n^1\land\rho_n^2$.
        Let $\Gamma_n \ce \Gamma \cap \{\|u_0^1\|_Y \le \frac{n}{2}\}$ for $n \in \N$, which by definition of $\Gamma$ also ensures $\|u_0^2\|_Y \le \frac{n}{2}$ a.s. By construction, $(\rho_n)_{n \in \N}$ is increasing, 
        $\rho_n \nearrow \sigma$,
        and $u_j \1_{\Gamma_n \times [0,\rho_n]} \in L^p(\Omega;L^\infty(0,T;Y))$ for $j=1,2$, $n \in \N$. 
		
		We want to determine a stopped mild solution up to time $\rho_n$. Denote the local solution by $(u,\sigma)$ and the initial values by $u_0$. Later we will apply this for $u_1$ and $u_2$ with initial values $u_0^1$ and $u_0^2$.
		Because $(u,\sigma)$ is a local solution, it solves the fixed point problem
		\begin{equation*}
		    u(t) = S(t)u_0 + \int_0^t S(t-s)F(u(s))\ds+ \int_0^t S(t-s)G(u(s))\dWHs\quad \text{a.s.\ for all }t \in [0,\sigma).
		\end{equation*}
		Let $T_0>0$ and for every progressively measurable, $Y$-valued $v$ for which the following convolutions are well-defined, consider the fixed point functional $L_{T_0}(v): [0,T_0] \to L^p(\Omega;X)$,
		\begin{align*}
			L_{T_0}(v) &\ce S(\cdot)u_0 + \int_0^\cdot S(\cdot-s)F(v(s))\ds+ \int_0^\cdot S(\cdot-s)G(v(s))\dWHs.
		\end{align*}
		Since as a local solution, $u$ is a fixed point only before time $\sigma$, $L_{T_0}(u) \neq u$ in general. However, the stopped solution $u^{\rho_n}\ce u(\cdot \land \rho_n)$ is a fixed point of the stopped functional, so almost surely
		\begin{equation*}
		    u^{\rho_n}(t)=(L_{T_0}(u^{\rho_n}))^{\rho_n}(t)\qquad (t \in [0,T_0]).
		\end{equation*}
		We claim that almost surely on $[0,T_0]$
		\begin{equation}
			\label{eq:stoppedFixedPt}
			u^{\rho_n} \1_{\Gamma_n \times [0,\rho_n)}=(L_{T_0}(u^{\rho_n}))^{\rho_n} \1_{\Gamma_n \times [0,\rho_n)}= L_{T_0}(u^{\rho_n}\1_{\Gamma_n \times [0,\rho_n)}) \1_{\Gamma_n \times [0,\rho_n)}.
		\end{equation}
        The first equality follows after multiplying both sides with the indicator function. We verify the second equality for all three terms of the fixed point functional separately. For the initial values, the indicator function ensures that $t \land \rho_n=t$ whenever the expression does not vanish. Thus, a.s.\ for $t\in [0,T_0]$
		\begin{equation*}
        \big((S(\cdot)u_0)^{\rho_n}\1_{\Gamma_n \times [0,\rho_n)}\big)(t) = S(t \land \rho_n)u_0 \1_{\Gamma_n \times [0,\rho_n)}(t) = S(t)u_0 \1_{\Gamma_n \times [0,\rho_n)}(t).
		\end{equation*}
		For the deterministic convolution, we observe
		\begin{align*}
			\Big( \int_0^\cdot S(\cdot-s)F(u^{\rho_n}(s))\ds\Big)^{\rho_n}\1_{\Gamma_n \times [0,\rho_n)}
			&= \int_0^{\cdot \land \rho_n} S(\cdot \land \rho_n-s)F(u^{\rho_n}(s))\ds\1_{\Gamma_n \times [0,\rho_n)}\\
			&= \int_0^{\cdot} S(\cdot-s)F((u^{\rho_n}\1_{\Gamma_n \times [0,\rho_n)})(s))\ds\1_{\Gamma_n \times [0,\rho_n)},
		\end{align*}
		where in the last equality we have used that $u^{\rho_n}(s)=(u^{\rho_n}\1_{\Gamma_n \times [0,\rho_n)})(s)$ for all $s \in [0,\rho_n)$ and either $\cdot\land \rho_n=\cdot$ or the whole expression vanishes.
		
		More caution is needed when stopping the stochastic convolution. We introduce the notation $I\Phi \ce \int_0^\cdot S(\cdot-s)\Phi(s)\dWHs$ for $\Phi$ such that the stochastic integral exists. For a stopping time $\tau$, one may not evaluate $(I\Phi)(t \land \tau)$ by naively replacing $t$ by $t\land \tau$ in the definition of $(I\Phi)(t)$, since the resulting integrand need not be adapted. Instead, the stopped stochastic convolution satisfies $(I\Phi)^\tau \1_{[0,\tau)}=I(\1_{[0,\tau)}\Phi^\tau)\1_{[0,\tau)}$ (see \cite[Lemma~A.1]{Brzezniak05}).  Hence, for $\Phi=G(u^{\rho_n})$ and $\tau=\rho_n$,
		\begin{equation*}
		    (I(G(u^{\rho_n})))^{\rho_n} \1_{\Gamma_n \times [0,\rho_n)} = I(\1_{\Gamma_n \times [0,\rho_n)}G( u^{\rho_n}))\1_{\Gamma_n \times [0,\rho_n)}
            =  I(G( u^{\rho_n}\1_{\Gamma_n \times [0,\rho_n)}))\1_{\Gamma_n \times [0,\rho_n)},
		\end{equation*}
        where we have used that $\Gamma_n\in\F_0$.
		Combining the equalities for the three terms the fixed point functional is composed of, we recover the second equality in \eqref{eq:stoppedFixedPt}.

    For $j=1,2$, set $z_j\ce u_j^{\rho_n}\1_{\Gamma_n\times[0,\rho_n)}$.
	Since $(u_1,\sigma^1)$ and $(u_2,\sigma^2)$ are local solutions, the above considerations imply
		\begin{align}
        \label{eq:z1z2diffProof}
			\|&z_1-z_2\|_{L^p(\Omega;L^\infty(0,T_0;X))}= \big\|\big(L_{T_0}^{(1)}(z_1)-L_{T_0}^{(2)}(z_2)\big)\1_{\Gamma_n \times [0,\rho_n)}\big\|_{L^p(\Omega;L^\infty(0,T_0;X))}\nonumber\\
            &= \big\|\big(L_{T_0,n}^{(1)}(z_1)-L_{T_0,n}^{(2)}(z_2)\big)\1_{\Gamma_n \times [0,\rho_n)}\big\|_{L^p(\Omega;L^\infty(0,T_0;X))}\nonumber\\
            &\le \Big\|\1_{\Gamma_n\times[0,\rho_n)} \int_0^\cdot S(\cdot-s)\big[F_n(z_1(s))-F_n(z_2(s))\big]\ds \Big\|_{L^p(\Omega;L^\infty(0,T_0;X))}\nonumber\\
		      &\phantom{\le }+ \Big\|\1_{\Gamma_n\times[0,\rho_n)} \int_0^\cdot S(\cdot-s)\big[G_n(z_1(s))-G_n(z_2(s))\big]\dWHs \Big\|_{L^p(\Omega;L^\infty(0,T_0;X))},
		\end{align}
        where $L_{T_0}^{(j)}$, $j=1,2$, denotes the fixed point functional with initial value $u_0^j$. By construction, $u_j^{\rho_n}\in B_n$ on $\Gamma_n \times [0,\rho_n)$, which allows us to replace $L_{T_0}^{(j)}$ by $L_{T_0,n}^{(j)}$ on $\Gamma_n \times [0,\rho_n)$ with $L_{T_0,n}^{(j)}$ defined as $L_{T_0}^{(j)}$ with the globally Lipschitz $F_n,G_n$ in place of $F, G$, respectively. Since $u_0^1=u_0^2$ on $\Gamma_n$, the initial value terms cancel above, so that only the two convolutions remain in the last inequality.  Then, we can trivially estimate the two remaining indicators in the last inequality by $1$. Denote by $L_{T_0,n}^{(0)}$ the fixed point functional with $F_n,G_n$ and initial values $u_0=0$.
        As an inspection of the proof of \cite[Thm.~4.3]{KliobaVeraar24Rate} shows, $L_{T_0,n}^{(0)}$ is a strict contraction for $T_0\le T$ sufficiently small, since there is a constant $C_n(T_0)\ge 0$ with $\lim_{T_0\searrow 0} C_n(T_0)=0$ such that for all progressively measurable $u,v\in L^p(\Omega;L^\infty(0,T_0;X))$
		\begin{equation*}
		    \|L_{T_0,n}^{(0)}(u)-L_{T_0,n}^{(0)}(v)\|_{L^p(\Omega;L^\infty(0,T_0;X))} \le C_n(T_0) \|u-v\|_{L^p(\Omega;L^\infty(0,T_0;X))}.
		\end{equation*}
	   More precisely, the proof gives that this estimate also applies to the sum of the norms of the two convolutions in \eqref{eq:z1z2diffProof} after removing the indicators. Thus,
		\begin{align*}
			&\|z_1-z_2\|_{L^p(\Omega;L^\infty(0,T_0;X))}
            \le C_n(T_0)\|z_1-z_2\|_{L^p(\Omega;L^\infty(0,T_0;X))}.
		\end{align*}
		In conclusion, for $T_0\le T$ sufficiently small such that $C_n(T_0)<1$, we have $u_1^{\rho_n}\1_{\Gamma_n \times [0,\rho_n)}=z_1=z_2= u_2^{\rho_n}\1_{\Gamma_n \times [0,\rho_n)}$ in $L^p(\Omega;L^\infty(0,T_0;X))$. 
        
        This argument can be repeated on successive intervals of length at most $T_0$ until $T$ is reached using the semigroup property and the path continuity of the local solutions before $\rho_n$ to restart at each interval endpoint. It yields
        \begin{equation*}
            u_1^{\rho_n}\1_{\Gamma_n \times [0,\rho_n)}= u_2^{\rho_n}\1_{\Gamma_n \times [0,\rho_n)}\text{ in }L^p(\Omega;L^\infty(0,T;X)).
        \end{equation*}
        As local solutions, $(u_1,\sigma^1)$ and $(u_2,\sigma^2)$ almost surely have continuous $Y$-, and thus $X$-valued paths. Hence, the above equality implies that a.s.\ on $\Gamma_n$, $u_1(t)=u_2(t)$ in $X$ for every $t\in [0,\rho_n)$. By injectivity of the embedding $Y \hra X$, this equality also holds with values in $Y$.  Finally, letting $n \to \infty$, we conclude $u_1=u_2$ almost surely on $\Gamma\times [0,\sigma)$. 
	\end{proof}
	
	\subsection{Blow-up criterion and local well-posedness}
    \label{subsec:blowupLocWP}

    The final step towards local well-posedness consists of a suitable blow-up criterion related to the supremum in time of the $Y$-norm of the solution.
	\begin{definition}
		Let $(u,\sigma)$ be a local solution of \eqref{eq:SEEsemilinear}. We call $\sigma$ an \emph{explosion time} for $u$ if for almost every $\omega \in \Omega$ with $\sigma(\omega)<T$, we have
		\begin{equation}
        \label{eq:explosionTimeBlowUp}
            \limsup_{t \nearrow \sigma(\omega)} \|u(t,\omega)\|_Y = \infty.
		\end{equation}
	\end{definition}
	
	\begin{lemma}[Blow-up criterion]
		\label{lem:blowUpExplosion}
		Suppose that Assumption \ref{ass:localWP} holds and let $p\in [2,\infty)$. Let $(u_1,\sigma^1)$ and $(u_2,\sigma^2)$ be local solutions of \eqref{eq:SEEsemilinear} with initial values $u_0^1,u_0^2 \in L_{\F_0}^p(\Omega;Y)$, respectively. Let $\sigma^1$ be an explosion time for $u_1$.
		
		Then $\sigma^1 \ge \sigma^2$ almost surely on $\Gamma\ce \{u_0^1=u_0^2\} \seq \Omega$. If, moreover, $\sigma^2$ is an explosion time for $u_2$, then $\sigma^1=\sigma^2$ and $u_1 = u_2$ almost surely on $\Gamma$ and $\Gamma\times[0,\sigma^1)$, respectively.
	\end{lemma}
	\begin{proof}
		We follow the proof by contradiction of \cite[Lemma~8.2]{NVW3} based on \cite[Lemma~5.3]{Seidler93ExplosionTime}. Adopt the notation from the proof of Lemma \ref{lem:localUnique} and for $n\in\N$ let
        \begin{equation*}
            \nu_n^j \ce \inf\{t\in [0,\sigma^j): \|u_j(t)\|_Y\ge n\},\quad \inf \emptyset \ce \sigma^j,\,j=1,2.
        \end{equation*}
		Suppose that $\Gamma_0 \ce \Gamma \cap \{\sigma^1<\sigma^2\}$ has positive probability. Then as a local solution, the path $u_2(\omega,\cdot)$ is continuous and thus bounded on $[0,\sigma^1(\omega)]\seq [0,\sigma^2(\omega))$ for a.e.\ $\omega \in \Gamma_0$.       
        Thus, by definition of $\nu_n^2$, almost surely $ \Gamma_0=\bigcup_{n\in\N} (\Gamma_0 \cap \{\sigma^1 <\nu_n^2\} \cap \{\|u_0^1\|_Y\le n\})$. Hence, we can find an $n \in \N$ such
		that 
        \begin{equation*}
            \P(\Gamma_{0,n})>0,\quad \Gamma_{0,n}\ce \Gamma_0 \cap \{\sigma^1 <\nu_n^2\} \cap \{\|u_0^1\|_Y\le n\}.
        \end{equation*}
        On $\Gamma_{0,n}$, since $\sigma^1$ is an explosion time for $u_1$, we conclude recalling the definition of $\nu_{n+1}^1$ that
        \begin{equation}
        \label{eq:smallerHittingTimes}
            \nu_{n+1}^1(\omega)<\sigma^1(\omega)<\nu_n^2(\omega)\qquad (\text{a.e.}\,\omega \in \Gamma_{0,n}).
        \end{equation}
        On the one hand, \eqref{eq:smallerHittingTimes} implies $\|u_2(\omega,\nu_{n+1}^1(\omega))\|_Y<n$ for a.e.\ $\omega \in \Gamma_{0,n}$ by definition of $\nu_n^2$. On the other hand, local uniqueness from Lemma \ref{lem:localUnique} gives $u_1(\omega,\nu_{n+1}^1(\omega))=u_2(\omega,\nu_{n+1}^1(\omega))$ for a.e.\ $\omega \in \Gamma_{0,n}$ by the first inequality in \eqref{eq:smallerHittingTimes}. Because $\sigma^1$ is an explosion time and $u_1$ almost surely has continuous paths, we have $\|u_1(\omega,\nu_{n+1}^1(\omega))\|_Y=n+1$ for a.e.\ $\omega \in \Gamma_{0,n}$. In summary,
        \begin{align*}
            n+1 = \|u_1(\omega,\nu_{n+1}^1(\omega))\|_Y = \|u_2(\omega,\nu_{n+1}^1(\omega))\|_Y <n,
        \end{align*}
        which is a contradiction. Hence $\sigma^1 \ge \sigma^2$ a.s.\ on $\Gamma$.
        By symmetry, $\sigma^1=\sigma^2$ a.s.\ on $\Gamma$ if $\sigma^2$ is also an explosion time. The second assertion is now immediate from Lemma \ref{lem:localUnique}.
	\end{proof}
	From the previous  subsections and the last lemma, we can conclude local well-posedness in the sense of existence and uniqueness of a maximal local solution to \eqref{eq:SEEsemilinear}. 
	
	\begin{theorem}
		\label{thm:localWellPosed}
		Suppose that Assumption \ref{ass:localWP} is satisfied. Let $p \in [2,\infty)$ and $u_0 \in L_{\F_0}^p(\Omega;Y)$. 
		Then there exists a unique maximal local solution $(u,\sigma)$ of \eqref{eq:SEEsemilinear} and $\sigma$  is an explosion time for $u$.
	\end{theorem}
	
	\begin{proof}
		We follow the proof of \cite[Theorem~8.1(1)]{NVW3}. By Lemma \ref{lem:locSolun}, for every $n\in\N$ there exists $u_n\in L^p(\Omega;C([0,T];Y))$ such that $u_n$ solves the original equation \eqref{eq:SEEsemilinear} on $\Gamma_n\times[0,\sigma_n]$ with $\Gamma_n, \sigma_n$ as in Lemma \ref{lem:locSolun}. These solutions satisfy the following compatibility condition: For $m \le n$, we have $u_m=u_n$ on $\Gamma_m \times [0,\sigma_m]$ a.s.\ and $\sigma_m \le \sigma_n$ a.s.\ on $\Gamma_m$. Indeed, the first claim on $\Gamma_m \times [0,\sigma_m \land \sigma_n]$ follows by local uniqueness arguing as in Lemma \ref{lem:localUnique}. It follows on $\Gamma_m \times [0,\sigma_m]$ via the second claim. To show the latter, suppose that $\sigma_n<\sigma_m$ on a subset of $\Gamma_m$ of positive probability. Then $\sigma_n<T$ and thus $\|u_n(\sigma_n)\|_Y=n$ a.s.\ on this set. However, since $u_m(\sigma_n)=u_n(\sigma_n)$ and $\sigma_n<\sigma_m$ a.s.\ on this set, 
        \begin{equation*}
            n=\|u_n(\sigma_n)\|_Y=\|u_m(\sigma_n)\|_Y<m \le n,
        \end{equation*}
        which is a contradiction. 
        
        As a consequence of $\Gamma_m\seq\Gamma_n$ and $\sigma_m\le\sigma_n$ on $\Gamma_m$ a.s., the sequence $(\sigma_n\1_{\Gamma_n})_{n\in\N}$ is increasing.
        Due to the compatibility condition, $\Gamma_n \nearrow \Omega$ as $n \to \infty$, and the increasing sequence,  setting
        \begin{equation*}
            u(\omega,t) \ce u_n(\omega,t) \text{ for }\omega\in \Gamma_n,\,t<\sigma_n(\omega),\quad \sigma \ce \lim_{n \to \infty} \sigma_n\1_{\Gamma_n}
        \end{equation*}
        produces a well-defined local solution $(u,\sigma)$ to \eqref{eq:SEEsemilinear}. The adaptedness, a.s.\ continuous paths in $Y$, and the local mild formulation are inherited from $u_n$, cf.\ Lemma \ref{lem:locSolun}.

        Next, we show that $\sigma$ is an explosion time for $u$. Suppose that $\sigma<T$ and $\limsup_{t\nearrow \sigma} \|u(t)\|_Y<\infty$ on a set $\Omega_0$ of positive probability. By a similar construction as in the proof of Lemma \ref{lem:blowUpExplosion}, there are $n\in \N$ and a set $\Omega_{0,n}\subseteq \Omega_0\cap\Gamma_n$ of positive probability such that $\sup_{t<\sigma(\omega)} \|u(\omega,t)\|_Y<n$ for a.e.\ $\omega \in \Omega_{0,n}$. By construction of the set, $\sigma_n\ge \sigma$ a.s.\ on $\Omega_{0,n}$ and thus $\sigma_n=\sigma<T$ a.s.\ on $\Omega_{0,n}$. But by definition of $\sigma_n$ and continuity of paths, we have $\|u_n(\omega,\sigma_n(\omega))\|_Y=n$ for a.e.\ $\omega \in \Omega_{0,n}$. This is a contradiction to $\|u_n(\omega,\sigma_n(\omega))\|_Y \le \sup_{t<\sigma(\omega)} \|u(\omega,t)\|_Y<n$, which follows from $u_n(\omega,\sigma_n(\omega))=\lim_{t \nearrow \sigma(\omega)} u_n(\omega,t)=\lim_{t \nearrow \sigma(\omega)} u(\omega,t)$.

        Finally, let $(v,\tau)$ also be a local solution of \eqref{eq:SEEsemilinear} with the same initial value $u_0$. Because $\sigma$ is an explosion time for $u$, Lemma \ref{lem:blowUpExplosion} implies $\sigma \ge \tau$ a.s., so that by local uniqueness (cf.\ Lemma \ref{lem:localUnique}), $u=v$ on $\Omega \times [0,\tau)$ almost surely. In conclusion, $(u,\sigma)$ is the unique maximal solution.
	\end{proof}

\section{Global well-posedness}
\label{sec:globalWP}	

Global well-posedness is established under an additional coercivity assumption via a stochastic Gronwall argument in Theorem~\ref{thm:aprioriY}. Moreover, pathwise uniform a priori estimates are obtained in $Y$. We use the following notion of global solutions.
	\begin{definition}
		A local solution $(u(t))_{t\in[0,\sigma)}$ of \eqref{eq:SEEsemilinear} is called a \emph{global (mild) solution} of \eqref{eq:SEEsemilinear} if $\sigma=T$ almost surely and $u$ admits an extension $\hat{u}: \Omega \times [0, T] \to Y$ that almost surely has continuous $Y$-valued paths and satisfies a.s.\ for all $t\in[0,T]$
        \begin{equation}
        \label{eq:mildSolutionFormula}
            u(t)= S(t)u_0+\int_0^t S(t-s)F(u(s))\ds+\int_0^t S(t-s)G(u(s))\dWHs.
        \end{equation}
	\end{definition}

    Not only to prove global well-posedness but also for stability and convergence rates, an Itô formula for the $p$-th power of the norm of $u(t)$ is needed. However, as a mild solution, $u$ is not necessarily a semimartingale, prohibiting a straightforward application of Itô's formula. Instead, we prove an Itô-type inequality for mild solutions of hyperbolic SPDEs in Lemma \ref{lem:ItoFormulaApprox}. A similar inequality has, under different integrability assumptions, been shown in \cite[Theorem~6]{SalavatiZangeneh16}. We prove an $L^0$-version below based on Yosida approximations. Contractivity of the underlying semigroup is used to obtain an upper bound independent of its generator. First, we show a pointwise a priori estimate via a standard localisation procedure.  
	
	\begin{lemma}
		\label{lem:supuBddLocalization}
		Let $Z$ be a separable Hilbert space, $T>0$, $u_0\in L_{\F_0}^0(\Omega;Z)$, and let 
        \begin{equation*}
            f \in L^0(\Omega;L^1(0,T;Z)),\quad g \in L^0(\Omega;L^2(0,T;\LHZ))
        \end{equation*}
        be progressively measurable. Suppose that $-A$ generates a $C_0$-contraction semigroup $(S(t))_{t \ge 0}$ on $Z$ and suppose that a.s.\
		\begin{equation*}
		    u(t) = S(t)u_0 + \int_0^t S(t-s)f(s)\ds + \int_0^t S(t-s)g(s)\dWHs,\quad t \in [0,T].
		\end{equation*}
		Then $u$ has an almost surely continuous $Z$-valued modification. In particular, almost surely
		\begin{align}
			\label{eq:supuBddAs}
			\sup_{t\in [0,T]} \|u(t)\|_Z <\infty. 
		\end{align} 
	\end{lemma}
	\begin{proof}
		We show the claim via a standard localization procedure. For $n\in \N$, define $\Gamma_n \ce \{\omega \in \Omega:\|u_0(\omega)\|_Z \le \frac{n}{2}\}$ and the stopping time
		\begin{equation*}
		    \tau_n \ce \inf\{t \in [0,T]: \|f\|_{L^1(0,t;Z)}+\|g\|_{L^2(0,t;\LHZ)}\ge n\}.
		\end{equation*}
        Let $u_{0,n} \ce \1_{\Gamma_n}u_0$, $f_n\ce  \1_{\Gamma_n \times [0,\tau_n]}f$, $g_n \ce \1_{\Gamma_n \times [0,\tau_n]}g$, and 
        \begin{equation*}
            u_n(t) \ce S(t)u_{0,n} + \int_0^t S(t-s)f_n(s)\ds + \int_0^t S(t-s)g_n(s)\dWHs.
        \end{equation*}
		We have $u_{0,n} \in L^2(\Omega;Z)$, $f_n \in L^2(\Omega; L^1(0,T;Z))$, and $g_n \in L^2(\Omega; L^2(0,T;\LHZ))$ with the respective norms bounded by $n$ due to the localization. Hence, the maximal inequality of Theorem \ref{thm:maximal-inequality} yields
		\begin{align*}
			\bigg\|\sup_{t\in [0,T]}\|u_n(t)\|_Z\bigg\|_{L^2(\Omega)} &\le C\big(\|u_{0,n}\|_{L^2(\Omega;Z)} + \|f_n\|_{L^2(\Omega;L^1(0,T;Z))}\\
            &\phantom{\le C \big( }+ \|g_n\|_{L^2(\Omega;L^2(0,T;\LHZ))}\big)\le 3C n.
		\end{align*}
        In particular, $u_n$ admits an a.s.\ continuous $Z$-valued version.
        
        For all $\omega \in \bigcup_{n\in\N} (\Gamma_n\cap \{\tau_n=T\})$, define
        $n_0(\omega)\ce\inf\{n\in\N: \omega \in \Gamma_n \land \tau_n(\omega)=T\}$  and for $t\in [0,T]$ set $\tilde{u}(\omega,t)\ce u_{n_0(\omega)}(\omega,t)$ and $\tilde{u}(\omega,t)=0$ for $\omega$ in the remaining null set. By construction, $u_{0,n},f_n,g_n$ agree with the original data on $(\Gamma_n \cap \{\tau_n=T\})\times[0,T]$. Hence, $\tilde{u}(t)=u(t)$ for all $t \in [0,T]$ a.s., so that $\tilde{u}$ is a modification of $u$ with a.s.\ continuous paths in $Z$. Moreover, for a.e.\ $\omega \in \Omega$, we have
        \begin{align*}
            \sup_{t\in[0,T]}\|\tilde{u}(\omega,t)\|_Z &= \sup_{t\in[0,\tau_{n_0(\omega)}(\omega)]} \|u_{n_0(\omega)}(\omega,t)\|_Z 
            = \sup_{t\in[0,T]}\|u_{n_0(\omega)}(\omega,t)\|_Z
            <\infty,
        \end{align*}
        where we have used that the $L^2(\Omega)$-estimate above implies almost sure finiteness. The claim follows after replacing $u$ by its modification $\tilde{u}$. 
	\end{proof} 
	
	\begin{lemma}
		\label{lem:ItoFormulaApprox}
		Let $Z$ be a real or complex separable Hilbert space, $H$ a real separable Hilbert space, $T>0$, 
        \begin{equation*}
            f \in L^0(\Omega;L^1(0,T;Z))\text{ and }g \in L^0(\Omega;L^2(0,T;\LHZ))
        \end{equation*} 
        be progressively measurable, and $u_0\in L_{\F_0}^0(\Omega;Z)$. Suppose that $-A$ generates a $C_0$-contraction semigroup $(S(t))_{t \ge 0}$ on $Z$ and let a.s.\
		\begin{equation*}
		u(t) = S(t)u_0 + \int_0^t S(t-s)f(s)\ds + \int_0^t S(t-s)g(s)\dWHs,\quad t \in [0,T].
		\end{equation*}
		Then for all $p\in[2,\infty)$ almost surely, for all $t \in [0,T]$ 
		\begin{align}
			\label{eq:lemmaItoApprox}
			\|u(t)\|_Z^p &\le \|u_0\|_Z^p + p \int_0^t \|u(s)\|_Z^{p-2} \Re\la u(s),f(s)\ra_Z \ds+ p \int_0^t \|u(s)\|_Z^{p-2} g^*(s)u(s)\dWHs \nonumber\\
			&\phantom{\le }+ \frac{p(p-2)}{2} \int_0^t \1_{\{u(s)\neq 0\}}\|u(s)\|_Z^{p-4} \|g^*(s)u(s)\|_H^2\ds + \frac{p}{2} \int_0^t \|u(s)\|_Z^{p-2} \|g(s)\|_\LHZ^2\ds. 
		\end{align}
		In particular, for all $p\in[2,\infty)$ almost surely, for all $t \in [0,T]$ 
		\begin{align}
			\label{eq:lemmaItoApproxSimplified}
			\|u(t)\|_Z^p &\le \|u_0\|_Z^p + p \int_0^t \|u(s)\|_Z^{p-2} \Re \la u(s),f(s)\ra_Z \ds+ p \int_0^t \|u(s)\|_Z^{p-2} g^*(s)u(s)\dWHs \nonumber\\
			&\phantom{\le }+ \frac{p(p-1)}{2}  \int_0^t \|u(s)\|_Z^{p-2} \|g(s)\|_\LHZ^2\ds. 
		\end{align}
        Here, $g(s)^*z$ for $z\in Z$ denotes the real adjoint given by $\la g(s)^*z,h\ra_H= \Re \la z,g(s)h\ra_Z$ for all $h\in H$.
	\end{lemma}
	\begin{proof}
		For $n \in \N$, consider $T_n:Z \to Z, T_n \ce nR(n,-A)$. Then $\|T_n\|_{\calL(Z)}\le 1$ by contractivity of $(S(t))_{t\ge 0}$ and $T_n \to I$ in the strong operator topology. Define $u_{0,n},f_n,g_n$ as $T_n$ applied to $u_0,f,g$, respectively, and $u_n:\Omega \times [0,T] \to Z$ by
		\begin{equation*}
		      u_n(t) \ce T_nu(t) = S(t)u_{0,n} + \int_0^t S(t-s)f_n(s)\ds + \int_0^t S(t-s)g_n(s)\dWHs,\quad t \in [0,T].
		\end{equation*}
		\textit{Step 1: Strong solution $u_n$.} We show that $u_n$ is a strong solution of 
        \begin{equation*}
            \rmd u_n = (-Au_n+f_n)\dt + g_n \dWH,\quad u_n(0)= u_{0,n}.
        \end{equation*} Fix $t\in[0,T]$. Then
        \begin{equation}
        \label{eq:strongSolproof}
            u_n(t) = u_{0,n} + \int_0^t (-Au_n(s) + f_n(s)) \ds + \int_0^t g_n(s) \dWHs.
        \end{equation} 
        Indeed, by definition, $u_{0,n},f_n(t)\in D(A)$, the range of $g_n(t)$ lies in $D(A)$, and $Ag_n(t)\in \LHZ$ a.s.\ for a.e.\ $t \in [0,T]$. Hence, inserting the definition of $u_n$ and using the deterministic and the stochastic Fubini theorem as well as $-\int_a^b AS(s)x\ds = S(b)x-S(a)x$, we obtain
		\begin{align*}
			-\int_0^t &Au_n(s) \ds 
			\\
			&= -\int_0^t AS(s)u_{0,n} \ds - \int_0^t \int_r^t AS(s-r)f_n(r)\ds\,\rmd r - \int_0^t \int_r^t AS(s-r)g_n(r)\ds\,\rmd W_H(r)\\
			&= S(t)u_{0,n}-u_{0,n} + \int_0^t [S(t-r)-I]f_n(r)\,\rmd r + \int_0^t [S(t-r)-I]g_n(r)\,\rmd W_H(r).
		\end{align*}
        Rearranging, we deduce \eqref{eq:strongSolproof}.\\
		\textit{Step 2: Itô's formula for $u_n$.} As a strong solution, $(u_n(t))_{t \ge 0}$ is an Itô process, so that we can apply Itô's formula \cite[Theorem~4.32]{DaPratoZabczyk14}, which gives a.s.\ 
		\begin{align}
			\label{eq:unNormItoLemma}
			\|u_n(&t)\|_Z^p = \|u_{0,n}\|_Z^p + p \int_0^t \|u_n(s)\|_Z^{p-2} \Re \la u_n(s),-Au_n(s)\ra_Z \ds\nonumber\\
            &+ p \int_0^t \|u_n(s)\|_Z^{p-2} \Re \la u_n(s),f_n(s)\ra_Z \ds+ p \int_0^t \|u_n(s)\|_Z^{p-2} g_n^*(s)u_n(s)\dWHs\nonumber\\
			& + \frac{p(p-2)}{2} \int_0^t \1_{\{u_n(s)\neq 0\}}\|u_n(s)\|_Z^{p-4} \|g_n^*(s)u_n(s)\|_H^2\ds+ \frac{p}{2} \int_0^t \|u_n(s)\|_Z^{p-2} \|g_n(s)\|_\LHZ^2\ds\nonumber\\
            & \ec \sum_{j=1}^6 I_j \le I_1 + \sum_{j=3}^6 I_j.
		\end{align}
        In the last inequality, we have used that contractivity of $(S(t))_{t \ge 0}$ implies dissipativity of $-A$, i.e.\ $\Re \langle -Au_n(s),u_n(s)\ra\le 0$ for all $s \in [0,T]$. Hence, $I_2\le 0$ for all $t\in [0,T]$ almost surely.\\
		\textit{Step 3: Taking limits $n \to \infty$.} For the convergence analysis of $I_j$, $j=1,4,5,6$, fix $\omega \in \Omega$. Since $T_n\to I$ strongly, $u_{0,n}(\omega) = T_nu_0(\omega) \to u_0(\omega)$ in $Z$ and thus $\|u_{0,n}(\omega)\|_Z^p \to \|u_0(\omega)\|_Z^p$, i.e. $I_1 \to \|u_0\|_Z^p$ a.s.\ as $n \to \infty$.
		
		To determine the limit of $I_3$, we split
		\begin{align*}
			&\Big|I_3-p \int_0^t \|u(s)\|_Z^{p-2} \Re \la u(s),f(s)\ra_Z \ds\Big| \le p \int_0^t \big|\|u_n(s)\|_Z^{p-2}-\|u(s)\|_Z^{p-2}\big| |\Re\la u_n(s),f_n(s)\ra_Z| \ds\\
			&+ p \int_0^t \|u(s)\|_Z^{p-2} |\Re\la u_n(s)-u(s),f_n(s)\ra_Z| \ds + p \int_0^t \|u(s)\|_Z^{p-2} |\la u(s),f_n(s)-f(s)\ra_Z| \ds.
		\end{align*}
		Since $T_n\to I$ strongly, the integrands converge pointwise to $0$. Since $\|T_n\|_{\calL(Z)}\le 1$, we further have $\sup_{n \in \N}\|u_n(s)\|_Z \le \|u(s)\|_Z$ and $\sup_{n \in \N}\|f_n(s)\|_Z \le \|f(s)\|_Z$. An analogous bound holds for $g_n$. Thus, via the Cauchy--Schwarz and the triangle inequality, all integrands can be bounded by $2p (\sup_{r\in [0,t]}\|u(r)\|_Z^{p-1})\|f(s)\|_Z$, which is integrable, since a.s.\ by Lemma \ref{lem:supuBddLocalization}
		\begin{align*}
			&\int_0^t 2p \Big(\sup_{r\in [0,t]}\|u(r)\|_Z^{p-1}\Big)\|f(s)\|_Z \ds = 2p\|f\|_{L^1(0,t;Z)} \Big(\sup_{r\in [0,t]}\|u(r)\|_Z^{p-1}\Big)<\infty.
		\end{align*}
		 Hence, dominated convergence on $[0,t]$ gives $I_3\to p \int_0^t \|u(s)\|_Z^{p-2} \Re\la u(s),f(s)\ra_Z \ds$ a.s.\ as $n \to \infty$. 

		Note that $I_5$ vanishes for $p=2$. Let $p>2$. Further, observe that $u(s)=0$ implies $u_n(s)=0$ by linearity of $T_n$. If $u(s)\neq 0$, also $u_n(s)\neq 0$ for sufficiently large $n$ due to $u_n(s)\to u(s)$. By reflexivity of $Z$, $-A^*$ generates a $C_0$-semigroup, too, and thus $T_n^*=nR(n,-A^*)\to I$ strongly on $Z$. Together with $T_n\to I$ strongly and uniform boundedness of $\|T_n^*\|_{\calL(Z)}$, this implies
        $g_n^*(s)u_n(s)=g^*(s)T_n^*T_nu(s) \to g^*(s)u(s)$ in $H$ as $n\to \infty$ for a.e.\ $s\in [0,t]$ a.s. Consequently,
        \begin{equation*}
            \1_{\{u_n(s)\neq 0\}}\|u_n(s)\|_Z^{p-4}\|g_n^*(s)u_n(s)\|_H^2 \to 
            \1_{\{u(s)\neq 0\}}\|u(s)\|_Z^{p-4}\|g^*(s)u(s)\|_H^2
        \end{equation*}
        for a.e.\ $s\in [0,t]$ a.s. The upper bound
        \begin{align*}
            \1_{\{u_n(s)\neq 0\}}\|u_n(s)\|_Z^{p-4}\|g_n^*(s)u_n(s)\|_H^2 
            \le \|u_n(s)\|_Z^{p-2}\|g_n(s)\|_\LHZ^2
            \le \|u(s)\|_Z^{p-2}\|g(s)\|_\LHZ^2
        \end{align*}
        is integrable on $[0,t]$ a.s.\ by Lemma \ref{lem:supuBddLocalization} and the assumption on $g$. Dominated convergence then yields $I_5 \to \frac{1}{2}p(p-2)\int_0^t\1_{\{u(s)\neq 0\}}\|u(s)\|_Z^{p-4}\|g^*(s)u(s)\|_H^2\ds$ a.s.\ as $n \to \infty$.
		
		For the sixth term, we split
		\begin{align*}
			&\Big|\int_0^t \|u_n\|_Z^{p-2} \|g_n\|_\LHZ^2\ds-\int_0^t \|u\|_Z^{p-2} \|g\|_\LHZ^2\ds\Big|\\
			&\le \Big|\int_0^t \big(\|u_n\|_Z^{p-2}-\|u\|_Z^{p-2}\big) \|g_n\|_\LHZ^2\ds\Big|+\Big|\int_0^t \|u\|_Z^{p-2} \big(\|g_n\|_\LHZ^2-\|g\|_\LHZ^2\big)\ds\Big|,
		\end{align*}
        where we have omitted the integration variable. Note that $T_n\to I$ strongly and uniform boundedness of $\|T_n\|_{\calL(Z)}$ implies $g_n(s)\to g(s)$ in $\LHZ$ as $n\to \infty$ for a.e. $s\in [0,t]$ a.s. Repeating earlier arguments results in     $I_6 \to \frac{p}{2}\int_0^t \|u(s)\|_Z^{p-2} \|g(s)\|_\LHZ^2\ds$ a.s.\ as $n \to \infty$.
		
		Lastly, we show convergence of the stochastic integral term $I_4$. After splitting
		\begin{align*}
			\|u_n(s)\|_Z^{p-2}& g_n^*(s)u_n(s) - \|u(s)\|_Z^{p-2} g^*(s)u(s)= \|u_n(s)\|_Z^{p-2} \big(g_n^*(s)u_n(s)-g^*(s)u(s)\big)\\
            &+ \big(\|u_n(s)\|_Z^{p-2}-\|u(s)\|_Z^{p-2}\big) g^*(s)u(s),
		\end{align*}
        previous arguments yield convergence of each term to $0$  in $H$ for a.e.\ $s\in [0,t]$ a.s. Moreover, 
        \begin{equation*}
            \big\|\|u_n(s)\|_Z^{p-2} g_n^*(s)u_n(s) - \|u(s)\|_Z^{p-2} g^*(s)u(s)\big\|_H^2 \le C_p \bigg(\sup_{r\in[0,t]} \|u(r)\|_Z^{2(p-1)}\bigg)\|g(s)\|_{\LHZ}^2,
        \end{equation*}
        which is an integrable upper bound on $[0,t]$ a.s.\ by Lemma \ref{lem:supuBddLocalization} and the assumption on $g$. Therefore, dominated convergence yields $\|u_n\|_Z^{p-2} g_n^*u_n \to \|u\|_Z^{p-2} g^*u$ in $L^2(0,t;H)$ a.s. The localized Burkholder--Davis--Gundy inequality of Proposition \ref{prop:localizedBDG} implies convergence of the corresponding stochastic integrals in $L^0(\Omega;C([0,t]))$. 

        It remains to prove that the inequality obtained in the limit holds almost surely simultaneously for all $t\in[0,T]$ and not just for a fixed $t\in [0,T]$ almost surely. By the previous dominated convergence arguments for $t=T$ and an application of the triangle inequality, we see that the deterministic integrals $I_3$, $I_5$, and $I_6$ converge uniformly on $[0,T]$. For the stochastic integral $I_4$, setting $t=T$ gives convergence of the stochastic integrals in $L^0(\Omega;C([0,T]))$ as $n\to \infty$. Passing to a subsequence $(n_\ell)_{\ell\in\N}$, this convergence holds almost surely in $C([0,T])$. Since $u$ has continuous paths almost surely, $u([0,T])\seq Z$ is compact, so  $\sup_{t\in [0,T]}\|u_n(t)-u(t)\|_Z \to 0$. Taking the limit $n \to \infty$ along the subsequence in \eqref{eq:unNormItoLemma} uniformly in $t$ thus finishes the proof.
	\end{proof}
	
With the Itô-type inequality at hand, global well-posedness and pathwise uniform a priori estimates can be established under the following parameter-dependent coercivity assumption. 
    \begin{assumption}[$p$) ($p$-Coercivity in $Y$]
\label{ass:coYp}
    For a given parameter $p\in [2,\infty)$, there is a constant $C_p\ge 0$ such that for all $x\in Y$,
    \begin{align*}
        \Re \la x,F(x)\ra_Y + \frac{p-1}{2}\|G(x)\|_\LHY^2 \le C_p(1+\|x\|_Y^2).
    \end{align*}
\end{assumption}

	\begin{theorem}[Global existence and a priori estimates in $Y$]
		\label{thm:aprioriY}
		Let $q \in[2,\infty)$ and $u_0\in L_{\F_0}^q(\Omega;Y)$. Suppose that Assumption \ref{ass:localWP} holds as well as $\frac{q}{\lambda}$-coercivity as in Assumption \ref{ass:coYp}($\frac{q}{\lambda}$) for some $\lambda \in (0,1)$. Let $(u,\sigma)$ be the unique local maximal solution to \eqref{eq:SEEsemilinear}. 
        
        Then the unique maximal local solution $(u,\sigma)$ to \eqref{eq:SEEsemilinear} is a global solution. That is, $\sigma=T$ almost surely and $u\in L^q(\Omega;C([0,T];Y))$. Moreover, there is a constant $C_{T,q,\lambda}\ge 0$ such that
		\begin{align*}
			\bigg\|\sup_{t\in [0,\sigma \land T)}\|u(t)\|_Y\bigg\|_{L^q(\Omega)} &\le C_{T,q,\lambda}
			\big(1+\|u_0\|_{L^q(\Omega;Y)}\big).
		\end{align*}
	\end{theorem}
Note that the a priori estimate extends to all $\bar{q}\in [q,\frac{q}{\lambda})$ provided that $u_0\in L_{\F_0}^{\bar{q}}(\Omega;Y)$, as an inspection of the following proof shows.
    
\begin{proof}
	\textit{Step 1: A priori estimate.} Let $(\sigma_n)_{n\in \N}$ be a localizing sequence for the unique maximal local solution $(u,\sigma)$, which exists by Theorem \ref{thm:localWellPosed}. For $n\in \N$, define
		\begin{equation*}
			\tau_n \ce \inf\{t \in [0,\sigma_n]: \|u(t)\|_Y \ge n\},\quad \inf \emptyset \ce \sigma_n,\qquad \Gamma_n \ce \Big\{\omega\in \Omega:\|u_0(\omega)\|_Y\le \frac{n}{2}\Big\}.
		\end{equation*}
        Then $(\tau_n)_{n\in\N}$ is also a localizing sequence for $(u,\sigma)$ and by local well-posedness, cf.\ Theorem \ref{thm:localWellPosed}, it holds that for all $t\in[0,T]$, a.s.\ $\1_{\Gamma_n\times[0,\tau_n]}S(t-\cdot)F(u(\cdot))\in L^1(0,t;Y)$ and $\1_{\Gamma_n\times[0,\tau_n]}S(t-\cdot)G(u(\cdot))\in L^2(0,t;\LHY)$. Hence,  
		\begin{align*}
			v_n &\ce S(\cdot)\1_{\Gamma_n}u_0+\int_0^\cdot \1_{\Gamma_n\times[0,\tau_n]}(s)S(\cdot-s)F(u(s))\ds+\int_0^\cdot \1_{\Gamma_n\times[0,\tau_n]}(s)S(\cdot-s)G(u(s))\dWHs
		\end{align*}
        is well-defined.

     Furthermore, $v_n=u$ on $\Gamma_n\times [0,\tau_n]$. Indeed, by Definition \ref{def:localSolution}, a.s.\
		\begin{align*}
			u(t)\1_{\Gamma_n\times[0,\tau_n]}(t) &= \1_{\Gamma_n\times[0,\tau_n]}(t)\Big(S(t)u_0 + \int_0^t S(t-s)F(u(s))\ds+ \int_0^t S(t-s)G(u(s))\dWHs\Big)\\
			&= \1_{\Gamma_n\times[0,\tau_n]}(t)\Big(S(t)u_0 + \int_0^t S(t-s)F(u(s))\1_{\Gamma_n\times[0,\tau_n]}(s)\ds\\
			&\phantom{= }+ \int_0^t S(t-s)G(u(s))\1_{\Gamma_n\times[0,\tau_n]}(s)\dWHs\Big)\\
			&= \1_{\Gamma_n\times[0,\tau_n]}(t)v_n(t)
		\end{align*}
        for all $t\in [0,T]$.
        Moreover, by the definition of $\tau_n$ and $\Gamma_n\in \F_0$, local $Y$-boundedness of $F$ and $G$ imply that a.s.\
        \begin{equation*}
            \int_0^T \|\1_{\Gamma_n\times[0,\tau_n]}(s)F(u(s))\|_Y\ds+\int_0^T \|\1_{\Gamma_n\times[0,\tau_n]}(s)G(u(s))\|_\LHY^2\ds \le L_n T + L_n^2 T<\infty,
        \end{equation*}
        i.e.\ $\1_{\Gamma_n\times[0,\tau_n]}F(u(\cdot))\in L^1(0,t;Y)$ and $\1_{\Gamma_n\times[0,\tau_n]}G(u(\cdot))\in L^2(0,T;\LHY)$. Hence, the Itô-type formula \eqref{eq:lemmaItoApproxSimplified} from Lemma \ref{lem:ItoFormulaApprox} is applicable to $v_n$ on $Z=Y$ with $p=\tilde{q}\ce \frac{q}{\lambda}$. It implies that for all $t\in [0,T]$ a.s.\
		\begin{align*}
			\|&v_n(t)\|_Y^{\tilde{q}} \le \|\1_{\Gamma_n}u_0\|_Y^{\tilde{q}} + \tilde{q} \int_0^t \|v_n\|_Y^{\tilde{q}-2} \Re\la v_n,\1_{\Gamma_n\times[0,\tau_n]}F(u)\ra_Y \ds\\
			&\phantom{\le }+ M_{n,t} + \frac{\tilde{q}(\tilde{q}-1)}{2}  \int_0^t \|v_n\|_Y^{\tilde{q}-2} \|\1_{\Gamma_n\times[0,\tau_n]}G(u)\|_\LHY^2\ds\\
			&=\|\1_{\Gamma_n}u_0\|_Y^{\tilde{q}} + \tilde{q} \int_0^t \1_{\Gamma_n\times[0,\tau_n]} \|v_n\|_Y^{\tilde{q}-2}\Big(\Re\la v_n,F(v_n)\ra_Y + \frac{\tilde{q}-1}{2}\|G(v_n)\|_\LHY^2\Big)\ds+M_{n,t}\\
			&\le \|u_0\|_Y^{\tilde{q}} + \tilde{q} C_{\tilde{q}} \int_0^t \1_{\Gamma_n\times[0,\tau_n]} \|v_n\|_Y^{\tilde{q}-2} (1+\|v_n\|_Y^2)\ds+M_{n,t}
		\end{align*}
        with
        \begin{equation}
        \label{eq:defMtglobalWPproof}
             M_{n,t}\ce \tilde{q} \int_0^t \|v_n\|_Y^{\tilde{q}-2} \1_{\Gamma_n\times[0,\tau_n]}G(v_n)^*v_n\dWH.
        \end{equation}
		Here, we have omitted the dependence on the integration variable and used that $v_n=u$ on $\Gamma_n\times[0,\tau_n]$ as well as the $\tilde{q}$-coercivity Assumption \ref{ass:coYp}($\tilde{q})$.

        Young's inequality with parameters $\frac{\tilde{q}}{\tilde{q}-2}$ and $\frac{\tilde{q}}{2}$ for the integrand of the second term yields
		\begin{align*}
			\tilde{q}(\|v_n(s)\|_Y^{\tilde{q}-2} +\|v_n(s)\|_Y^{\tilde{q}}) &\le  \tilde{q}  \Big(\frac{\tilde{q}-2}{\tilde{q}}\|v_n(s)\|_Y^{\tilde{q}} + \frac{2}{\tilde{q}} +\|v_n(s)\|_Y^{\tilde{q}}\Big) 
			= 2(\tilde{q}-1)\|v_n(s)\|_Y^{\tilde{q}} + 2.
		\end{align*}
		Altogether, for all $t \in [0,T]$ a.s.\ with generic constants
        \begin{align*}
			\|v_n(t)\|_Y^{\tilde{q}}
			&\le \|u_0\|_Y^{\tilde{q}} + C_{\tilde{q}} t + C_{\tilde{q}} \int_0^t \|v_n(s)\|_Y^{\tilde{q}} \ds+M_{n,t}.
		\end{align*}
        Because $G$ is locally bounded on $B_n$ and by choice of $\tau_n$, \eqref{eq:defMtglobalWPproof} defines a square-integrable martingale due to
         \begin{align*}
            \langle M_n\rangle_t 
            &\le \tilde{q}^2  \bigg(\sup_{s \in [0,\tau_n]}\|v_n(s)\|_Y^{2\tilde{q}-2}\bigg)\int_0^t\1_{\Gamma_n\times[0,\tau_n]} \|G(v_n)\|_\LHY^2\ds \le \tilde{q}^2  n^{2\tilde{q}-2} L_n^2t<\infty.
        \end{align*}
		Hence, we can apply the stochastic Gronwall inequality from Lemma \ref{lem:stochasticGronwall} with $\varphi_t= \|v_n(t)\|_Y^{\tilde{q}}$, $A_t= C_{\tilde{q}} t$, and $H_t=\|u_0\|_Y^{\tilde{q}}+C_{\tilde{q}} t$. This results in
		\begin{align*}
        \bigg\|e^{-C_{\tilde{q}}T}\sup_{t\in [0,T]}\|v_n(t)\|_Y^{\tilde{q}}\bigg\|_{L^\lambda(\Omega)}
        &\le C_\lambda \big\|\|u_0\|_Y^{\tilde{q}}+C_{\tilde{q}}T\big\|_{L^\lambda(\Omega)},
        \end{align*}
        which after elementary calculations yields
        \begin{align*}
		\bigg\|\sup_{t\in [0,T]}\|v_n(t)\|_Y\bigg\|_{L^q(\Omega)} &\le C_{q,\lambda}e^{C_{\tilde{q}}T}	\big(\|u_0\|_{L^q(\Omega;Y)} + C_{\tilde{q}} T^{1/\tilde{q}}\big).
		\end{align*}
        Since $\1_{\Gamma_n}\sup_{t\in [0,\tau_n]}\|u(t)\|_Y^q\le \sup_{t\in [0,T]}\|v_n(t)\|_Y^q$ a.s.\ and $\1_{\Gamma_n}\to 1$ and $\tau_n \to \sigma$, Fatou's lemma implies that 
        \begin{equation}
        \label{eq:aPrioriYproof}
            \bigg\|\sup_{t\in [0,\sigma)\cap [0,T]} \|u(t)\|_Y\bigg\|_{L^q(\Omega)}\le C_{T,q,\lambda}\big(1+\|u_0\|_{L^q(\Omega;Y)}\big).
        \end{equation}
		
		\textit{Step 2: Global well-posedness.} Step $1$ in particular implies that $\sup_{t\in [0,\sigma)}\|u(t)\|_Y<\infty$ almost surely. By definition of an explosion time, cf.\ \eqref{eq:explosionTimeBlowUp}, this contradicts $\sigma$ being an explosion time unless $\sigma \ge T$ a.s. Since $\sigma$ is indeed an explosion time by Theorem \ref{thm:localWellPosed} and $\sigma \le T$, we conclude $\sigma=T$ almost surely. This also yields the a priori estimate on $[0,T)$ by \eqref{eq:aPrioriYproof}. 

        We can define a continuous extension $\hat{u}$ of $u$ to $[0,T]$ by defining $\hat{u}(T)$ via the mild solution formula \eqref{eq:mildSolutionFormula}. Indeed, the almost sure boundedness of paths of $u$ on $[0,T)$ implies the required integrability of $F$ and $G$ pathwise so that Lemma \ref{lem:supuBddLocalization} on $Z=Y$ implies a.s.\ continuity of paths of $u$ in $Y$. After replacing $u$ by this extension, $u \in L^q(\Omega;C([0,T];Y))$ and by continuity, \eqref{eq:aPrioriYproof} holds with the supremum taken over the closed interval $[0,T]$. Uniqueness of the global solution follows from local uniqueness (see Lemma \ref{lem:localUnique}) and continuity at $T$.
	\end{proof}
	
\section{The nonlinearities-stopped exponential Euler scheme}
\label{sec:defScheme}

Up to the final time $T>0$, our aim is to approximate the mild solution
\begin{equation}
\label{eq:mildSolDefScheme}
    U(t)=S(t)u_0+\int_0^t S(t-s)F(U(s))\ds+\int_0^t S(t-s)G(U(s))\dWHs,\quad t \in [0,T],
\end{equation}
constructed in the previous sections. 
For temporal discretisation, we take a uniform grid $t_j=jk$, $0 \le j \le \Nk $, $\Nk \in \N$, where $k \ce \frac{T}{\Nk }>0$ is the time step and for $t\in [0,T]$, we define 
\begin{equation*}
    \lfloor t \rfloor \ce \max\{t_j:\ t_j \le t,\ 0 \le j \le \Nk \}    
\end{equation*}
as the closest grid point not exceeding $t$. Note that this convention implies $t-\lfloor t \rfloor \in [0,k)$ for $t\in [0,T]$ and $\lfloor t_j \rfloor = t_j$ for $0 \le j \le \Nk$. Clearly, the definition of $\floort$ depends on $k$ and shall not be confused with the floor function rounding non-integers to the closest smaller integer.

It is well-known that explicit schemes like the exponential Euler scheme diverge in the case of superlinear nonlinearities for several classes of SPDEs, such as parabolic reaction-diffusion type SPDEs \cite[Thm.~1.2]{Beccari2019strongweakdivergenceexponential} or stochastic nonlinear Schrödinger equations \cite[Lemma~3.1]{Cui25JDE}. More precisely, stability of the scheme cannot be obtained since the $p$-th moment of the scheme, $p\in [1,\infty)$, diverges as $k\to 0$. Hence, for a numerical scheme, the nonlinearities need to be modified by stopping or taming them. Various approaches exist to achieve this for parabolic SPDEs (see \cite{JentzenPusnik19,GyongySabanisSiska16, Wang20taming, Brehier22tamedExpEuler}). Here, we consider an exponential Euler scheme which stops the nonlinearities if their growth exceeds a certain order. It is an adaptation of the scheme used in \cite{JentzenPusnik19} to the hyperbolic setting. Generalisations to different tamings are discussed in Section \ref{sec:generalTamings}.

\begin{definition}[Scheme]
\label{def:Scheme}
    The \emph{nonlinearities-stopped exponential Euler scheme} $(U^j)_{j=0,\ldots,\Nk}$ with $U^j:\Omega\to Y$ is defined by
    \begin{equation}
    \label{eq:defSchemeUj}
        U^{j+1} \ce S(k) U^j + kS(k) F_\theta(U^j) + S(k)G_\theta(U^j)\Delta W_{j+1},~0\le j \le \Nk -1,\quad U^0 \ce u_0,
    \end{equation}
    where $\Delta W_{j+1} \ce W_H(t_{j+1})-W_H(t_j)$, the last term is understood in the sense of \eqref{eq:convradonW}, $F_\theta(U^j)\ce \1_{\theta,j} F(U^j)$, and $G_\theta(U^j)\ce \1_{\theta,j} G(U^j)$ with 
    \begin{equation*}
         \1_{\theta,j}:\Omega \to \{0,1\},\quad  \1_{\theta,j}(\omega) \ce \1_{\{\|F(U^j(\omega))\|_Y +\|G(U^j(\omega))\|_{\calL_2(H,Y)} \le k^{-\theta}\}},\quad 0 \le j \le \Nk
    \end{equation*}
    for some parameter $\theta \in (0,\frac{1}{4}]$.
\end{definition} 
We omit the dependence on $\omega$ in the following and simply write $\1_{\theta,j}$ and $U^j$. Note that $\1_{\theta,N}$ is not required to define the scheme, just to ensure consistency of the following definitions at the end point. The parameter $\theta$ controls at which growth order the nonlinearities are stopped. Its range $\theta\in(0,\frac{1}{4}]$ ensures stability of the scheme, cf.\ Proposition \ref{prop:stabY}. 

Further, we introduce some auxiliary processes useful for the error analysis. 
These include two extensions $U^k$ and $\hat{U}$ of the scheme to $[0,T]$. The first of these is piecewise constant as opposed to the second one, whose definition is in analogy to \cite[(34)]{JentzenPusnik19}.

\begin{definition}[Piecewise constant extension of the scheme]
\label{def:UkExtension}
    Let
    \begin{equation*}
        U^k(s)\ce U^j\text{ for }s\in[t_j,t_{j+1}),~ 0 \le j \le \Nk -1,\quad U^k(T)\ce U^\Nk 
    \end{equation*}
    denote the right-continuous piecewise constant extension $U^k:\Omega \times[0,T]\to Y$ of the scheme $(U^j)_{j=0}^N$ to $[0,T]$. Further, define the piecewise constant processes 
    \begin{equation*}
        F^k\ce F\circ U^k:\Omega \times [0,T]\to Y\text{ and }G^k\ce G\circ U^k:\Omega \times [0,T]\to \LHY.
    \end{equation*} 
\end{definition}

\begin{definition}[Continuous extension of the scheme]
\label{def:UhatExtension}
    For $t\in [t_j,t_{j+1})$, $0\le j \le \Nk-1$, define the \emph{continuous extension $\hat{U}$} of the scheme $(U^j)_{j=0}^N$ to $[0,T]$ by
    \begin{align}
	\label{eq:defUhattj}
    	\hat{U}(t) &\ce S(t-t_j)\Big[U^j + (t-t_j) \1_{\theta,j} F(U^j) +  \int_{t_j}^t \1_{\theta,j} G(U^j) \dWHs\Big],\quad \hat{U}(T)\ce U^N.
    \end{align}
\end{definition}

Indeed, both $U^k$ and $\hat{U}$ are extensions of the scheme $(U^j)_j$ in the sense of $U^k(t_j)=\hat{U}(t_j)=U^j$. Moreover, $\hat{U}$ is continuous on $[0,T]$, since by strong continuity of $(S(t))_{t\ge 0}$ and Definition \ref{def:Scheme} of the scheme, for $0\le j \le \Nk-1$,
\begin{equation}
\label{eq:UhatCont}
    \lim_{t\nearrow t_{j+1}} \hat{U}(t)= S(k)\Big[U^j + k F_\theta(U^j) +  \int_{t_j}^{t_{j+1}} G_\theta(U^j) \dWHs\Big] = U^{j+1}=\hat{U}(t_{j+1}).
\end{equation}
However, in general, $U^k(t) \neq \hat{U}(t)$ for $t \in (t_j,t_{j+1})$ between the grid points.

\begin{definition}[Generalised stopping functions]
\label{def:generalisedStop}
    For $x\in L^0(\Omega;Y)$, define the stopping functions
    \begin{equation*}
        \1_\theta(x):\Omega\to\{0,1\},\quad \1_\theta(x)(\omega)\ce \1_{\{\|F(x(\omega))\|_Y+\|G(x(\omega))\|_\LHY\le k^{-\theta}\}}
    \end{equation*}
    using the convention that $\|z\|_Z=+\infty$ if $z \notin Z$ for Banach spaces $Z$.
    Abbreviate by $F_\theta(x) \ce (\1_\theta F)(x)$ and $G_\theta(x)\ce (\1_\theta G)(x)$ the corresponding maps from $\Omega$ to $Y$ and $\LHY$, respectively. Define  $F_\theta^k(s)\ce F_\theta(U^k(s))$ and $G_\theta^k(s)\ce G_\theta(U^k(s))$ for $s \in [0,T]$.
\end{definition}

In particular, Definition \ref{def:generalisedStop} implies that $\1_{\theta,j} = \1_\theta(U^j)$ for $0 \le j \le \Nk$. Repeatedly using \eqref{eq:UhatCont}, we can rewrite $\hat{U}(t)$ as
\begin{align}
	\label{eq:UhatfromUk}
	\hat{U}(t) &= S(t-\floort)\hat{U}(\floort) + \int_{\floort}^t S(t-\floors)F_\theta^k(s) \ds+ \int_{\floort}^t S(t-\floors) G_\theta^k(s) \dWHs\\
    \label{eq:UhatMildFormula}
	&=S(t)u_0 + \int_0^tS(t-\floors) F_\theta^k(s) \ds + \int_0^tS(t-\floors)G_\theta^k(s) \dWHs,\quad t\in [0,T].
\end{align}

The values of $F_\theta^k(T)$ and $G_\theta^k(T)$ are irrelevant for the value of the integrals above, so that $\1_{\theta,N}$ is not required to compute $\hat{U}(T)$.

\begin{remark}
\label{rem:coercivityFGtheta}
    We observe that coercivity as in Assumption \ref{ass:coYp}($p$) for some $p\in [2,\infty)$ remains true with $F$ and $G$ replaced by $F_\theta$ and $G_\theta$ for some $\theta\in (0,\frac{1}{4}]$, respectively. Indeed, either both the indicators and thus the left-hand side vanish or $F=F_\theta$ and $G=G_\theta$. Hence, under Assumption \ref{ass:coYp}($p$) also
    \begin{align*}
        \Re\la x,F_\theta(x)\ra_Y + \frac{p-1}{2}\|G_\theta(x)\|_\LHY^2 \le C_p(1+\|x\|_Y^2),\qquad x\in Y.
    \end{align*}
\end{remark}

Lastly, we introduce the semilinear integrated counterparts of the scheme inspired by \cite[(35)]{JentzenPusnik19}. The difference to the mild solution is that inside the convolutions, the nonlinearities are evaluated at the scheme at the last grid point rather than at the solution at the current time.

\begin{definition}[Semilinear integrated counterparts]
\label{def:UbarSemilinearIntegratedCounterparts}
    The \emph{semilinear integrated counterparts $\bar{U}$} of the scheme $(U^j)_{j=0}^N$ are given by
    \begin{equation*}
    	\bar{U}(t) \ce S(t)u_0 +\int_0^t S(t-s)F^k(s)\ds+\int_0^t S(t-s)G^k(s)\dWHs,\quad t\in[0,T].
    \end{equation*}
\end{definition}

\section{Stability of the scheme}
\label{sec:stability}

In order to establish both pointwise and pathwise uniform stability, we leverage the stopping of the nonlinearities in the scheme. Roughly speaking, the stochastic integrals over one time interval arising from Itô's formula decay like $k^{1/2}$, meaning that the integrand may not grow faster than $k^{-1/2}$. Restricting the stopping parameter to $\theta\le \frac14$ ensures the latter and thus pointwise stability of the scheme, as illustrated in Subsection \ref{subsec:pointwiseStabScheme}. Under a polynomial growth assumption on the nonlinearities, the same stability result is obtained for the semilinear integrated counterparts in Subsection \ref{subsec:stabilityUbar}.

Pathwise uniform stability, that is, uniform moment bounds of the supremum over the time grid of the scheme, is analytically more challenging. Due to the supremum inside the expectation, straightforward BDG-type arguments require a linear growth assumption on $G$. However, using a marginally stronger coercivity assumption as in Section \ref{sec:globalWP} allows us to establish pathwise uniform stability also for superlinearly growing noise via a two-step BDG argument in Subsection \ref{subsec:pathwiseStabBDG}.

\subsection{Pointwise stability of the scheme}
\label{subsec:pointwiseStabScheme}

Our main stability result is as follows. 
\begin{proposition}[Pointwise stability of scheme in $Y$]
\label{prop:stabY}
    Suppose that $-A$ generates a $C_0$-contraction semigroup $(S(t))_{t\ge 0}$ on $Y$ and that $q$-coercivity as in Assumption \ref{ass:coYp}($q$) holds for some $q\in [2,\infty)$. Let $\theta\in (0,\frac{1}{4}]$, $u_0\in L_{\F_0}^q(\Omega;Y)$, and assume that $F:Y\to Y$ as well as $G:Y \to \LHY$ are Borel-measurable. Then the scheme is pointwise stable in $Y$. That is, for all $T>0$ there exists a constant $C_{T,q}\ge 0$ independent of $k$, $N$, and $\theta$ such that 
    \begin{equation*}
        \sup_{t\in [0,T]} \|\hat{U}(t)\|_{L^q(\Omega;Y)}\le  C_{T,q} \big(1+\|u_0\|_{L^q(\Omega;Y)} 
        \big)<\infty.
    \end{equation*}
    In particular, the same bound holds for $\max_{j=0,\ldots,\Nk } \|U^j\|_{L^q(\Omega;Y)}$ and $\sup_{t\in [0,T]} \|U^k(t)\|_{L^q(\Omega;Y)}$.
\end{proposition}

 Pointwise stability is obtained by iterating the one-step stability estimate below. We first illustrate how to iterate the one-step estimate before presenting its proof. Note that stability merely requires the stopped nonlinearities in the scheme and contractivity of the semigroup on $Y$. Local Lipschitz continuity or local $Y$-boundedness are not required in the stability proof. 

\begin{lemma}
\label{lem:stabOneTimeStep}
    Suppose that $q$-coercivity as in Assumption \ref{ass:coYp}($q$) holds for some $q\in [2,\infty)$ and that $F:Y\to Y$ and $G:Y \to \LHY$ are Borel-measurable. Let $T>0$, $k\in (0,T]$, $\theta\in (0,\frac{1}{4}]$, and $x\in L^q(\Omega;Y)$ be $\F_{t_\ell}$-measurable. For $r\in [0,k]$ and $0 \le \ell \le \Nk -1$, define
    \begin{equation}
    \label{eq:defVlxr}
        \Vlxr\ce x + r F_\theta(x) + \int_{t_\ell}^{t_\ell+r} G_\theta(x) \dWHs.
    \end{equation}
    Then there exists $C_{T,q}\ge 0$ independent of $k$ and $\theta$ such that for all $x,r,\ell$ as above, 
    \begin{align*}
       \E\|\Vlxr\|_Y^q
        &
        \le e^{C_{T,q} r}\big(\E\|x\|_Y^q+C_{T,q} r\big).
    \end{align*}
\end{lemma}

\begin{proof}[Proof of Proposition \ref{prop:stabY}]
    We show that
    \begin{equation*}
        \E\|\hat{U}(t)\|_Y^q \le e^{C_{T,q}t}\big(\E\|u_0\|_Y^q+C_{T,q}t\big),\quad t\in [0,T],
    \end{equation*}
    from which the claim follows after taking the supremum over $t$ and the $q$-th root.
    Clearly, the claim holds for $t=0$. Consider the case where $t=t_j$ for some $1\le j\le\Nk$. By definition of $\hat{U}$, the scheme, and $\Vlxr$, we can rewrite
    \begin{equation}
\label{eq:UhatRecursivelyGridtj} 
        \hat{U}(t_j) = U^j=S(k)V_{t_{j-1}}^{U^{j-1}}(k).
    \end{equation}
    By induction, $U^j\in L^q(\Omega;Y)$ and it is $\F_{t_j}$-measurable for every $0\le j \le N$ by Lemma \ref{lem:stabOneTimeStep}.
    Contractivity of the semigroup and
    Lemma \ref{lem:stabOneTimeStep} with $r=k$, $\ell=j-1$, and $x=U^{j-1}$ thus imply
    \begin{align*}
        \E\|\hat{U}(t_j)\|_Y^q=\E\|U^j\|_Y^q \le \E\big\|V_{t_{j-1}}^{U^{j-1}}(k)\big\|_Y^q\le e^{C_{T,q} k}\big(\E\|U^{j-1}\|_Y^q+C_{T,q} k\big).
    \end{align*}
    Iterating this estimate (in case $j\ge 2$), we deduce
    \begin{align*}
        \E\|\hat{U}(t_j)\|_Y^q &
        \le e^{C_{T,q} k}\big(e^{C_{T,q} k}\big(\E\|U^{j-2}\|_Y^q+C_{T,q} k\big)+C_{T,q}k\big) \le e^{2C_{T,q} k}\big(\E\|U^{j-2}\|_Y^q+2C_{T,q} k\big)\\
        &\le \ldots \le e^{C_{T,q} t_j}\big(\E\|u_0\|_Y^q+C_{T,q} t_j\big).
    \end{align*}
    Now let $t\in (t_j,t_{j+1})$ for some $0\le j \le \Nk-1$. Then by \eqref{eq:UhatfromUk}, 
    $\hat{U}(t) = S(t-t_j)V_{t_j}^{U^j}(t-t_j)$.
    Lemma \ref{lem:stabOneTimeStep} with $r=t-t_j$, $\ell=j$, and $x=U^j$ combined with the previous estimate thus imply
    \begin{align*}
        \E\|\hat{U}(t)\|_Y^q &\le e^{C_{T,q} (t-t_j)}\big(\E\|\hat{U}(t_j)\|_Y^q+C_{T,q} (t-t_j)\big)\\
        &\le  e^{C_{T,q} (t-t_j)}\big(e^{C_{T,q} t_j}\big(\E\|u_0\|_Y^q + C_{T,q} t_j\big)+C_{T,q} (t-t_j)\big)\le e^{C_{T,q} t}\big(\E\|u_0\|_Y^q +C_{T,q} t\big). \qedhere
    \end{align*}
\end{proof}

\begin{proof}[Proof of Lemma \ref{lem:stabOneTimeStep}]
    The estimate trivially holds for $r=0$. We present the proof for $q>2$. The case $q=2$ follows by the same argument but omitting the use of Hölder's and Young's inequalities with exponents $\frac{q}{q-2}$ and $\frac{q}{2}$ in the following. Let $r\in (0,k]$, $q\in (2,\infty)$, and define the stopping times
    \begin{equation*}
        \tau_n \ce \inf \{t \in [0,k]:\|\Vlxt\|_Y\ge n\},\quad \inf \emptyset \ce k
    \end{equation*}
    for the shifted filtration $(\F_{t_\ell+t})_{t\in [0,k]}$.
    Itô's formula yields that
    \begin{align}
    \label{eq:stabProofIto}
        \|\Vlxrtaun\|_Y^q &\le \|x\|_Y^q+q\int_{t_\ell}^{t_\ell+(r\land \tau_n)} \|\Vlxs\|_Y^{q-2}\Re\langle \Vlxs,F_\theta(x)\rangle_Y\ds \nonumber\\
        &\phantom{\le }+q \int_{t_\ell}^{t_\ell+(r\land \tau_n)} \|\Vlxs\|_Y^{q-2} G_\theta(x)^*\Vlxs\dWHs\nonumber\\
        &\phantom{\le }+ \frac{q(q-1)}{2} \int_{t_\ell}^{t_\ell+(r\land \tau_n)} \|\Vlxs\|_Y^{q-2} \|G_\theta(x)\|_\LHY^2\ds
    \end{align}
    almost surely. Since 
    \begin{align*}
        \E\int_{t_\ell}^{t_\ell+(r\land \tau_n)}q^2\|\Vlxs\|_Y^{2q-4}\|G_\theta(x)^*\Vlxs\|_H^2 \ds 
        \le r q^2n^{2q-2} k^{-2\theta}<\infty,
    \end{align*}
    the stochastic integral term is a square-integrable martingale evaluated at $r$, so its conditional expectation w.r.t.\ $\F_{t_\ell}$ vanishes. 
    We abbreviate $\phi_{n,p}(s-t_\ell)\ce \1_{[0,\tau_n]}(s-t_\ell)\|\Vlxs\|_Y^{p}$ for $p\in (0,\infty)$. After taking the conditional expectation of \eqref{eq:stabProofIto}, Fubini's theorem, the definition of $\Vlxr$, and $\F_{t_\ell}$-measurability of $x$, $F_\theta(x)$, and $G_\theta(x)$ thus imply a.s.\
    \begin{align}
    \label{eq:condExpStabProof}
        \E\big[&\|\Vlxrtaun\|_Y^q\big|\F_{t_\ell}\big] \le \E\big[\|x\|_Y^q|\F_{t_\ell}\big]+q\int_{t_\ell}^{t_\ell+r} \E\big[\phi_{n,q-2}(s-t_\ell)\Re\langle \Vlxs,F_\theta(x)\rangle_Y\big|\F_{t_\ell}\big]\ds \nonumber\\
        &\qquad\qquad\qquad\quad\qquad\phantom{\le }+ \frac{q(q-1)}{2} \int_{t_\ell}^{t_\ell+r} \E\big[\phi_{n,q-2}(s-t_\ell) \|G_\theta(x)\|_\LHY^2\big|\F_{t_\ell}\big]\ds\nonumber\\
        &= \E\big[\|x\|_Y^q\big|\F_{t_\ell}\big]+q\int_{t_\ell}^{t_\ell+r} \E\big[\phi_{n,q-2}(s-t_\ell)\big|\F_{t_\ell}\big] \Big(\Re\langle x ,F_\theta(x)\rangle_Y + \frac{q-1}{2}\|G_\theta(x)\|_\LHY^2\Big)\ds \nonumber\\
        &\phantom{\le }+q\int_{t_\ell}^{t_\ell+r} \E\Big[\phi_{n,q-2}(s-t_\ell)\Re \Big\langle (s-t_\ell)F_\theta(x)+\int_{t_\ell}^s G_\theta(x)\,\rmd W_H(\tau),F_\theta(x)\Big\rangle_Y\Big|\F_{t_\ell}\Big]\ds.
    \end{align}
    By Remark \ref{rem:coercivityFGtheta}, $F_\theta$ and $G_\theta$ satisfy the coercivity assumption. Together with Hölder's inequality for conditional expectations and  Young's inequality both with parameters $\frac{q}{q-2}$ and $\frac{q}{2}$ as well as linearity of the conditional expectation, we can thus use Assumption \ref{ass:coYp}($q$) to bound the second term in \eqref{eq:condExpStabProof} a.s.\ by 
    \begin{align*}
        &C_q \int_{t_\ell}^{t_\ell+r} \E\big[\phi_{n,q-2}(s-t_\ell)\big(1+\|x\|_Y^2\big)\big|\F_{t_\ell}\big]\ds\\
        \le &C_q \int_{t_\ell}^{t_\ell+r} \E\big[\phi_{n,q}(s-t_\ell)\big|\F_{t_\ell}\big]\ds +C_q  r\big(1+ \E\big[\|x\|_Y^q\big|\F_{t_\ell}\big]\big).
    \end{align*}
    Proceeding likewise for the last term in \eqref{eq:condExpStabProof} results in
    \begin{align*}
        \E\big[\|\Vlxrtaun\|_Y^q&\big|\F_{t_\ell}\big] 
        \le  C_q  r+(C_q  r+1)\E\big[\|x\|_Y^q\big|\F_{t_\ell}\big]+ C_q\int_{t_\ell}^{t_\ell+r} \E\big[\phi_{n,q}(s-t_\ell)\big|\F_{t_\ell}\big]\ds  \\
        &+ 2\int_{t_\ell}^{t_\ell+r}
        \E\Big[\Big|\Re\Big\langle (s-t_\ell)F_\theta(x)+\int_{t_\ell}^s G_\theta(x)\,\rmd W_H(\tau),F_\theta(x)\Big\rangle_Y\Big|^{q/2}\Big|\F_{t_\ell}\Big]
        \ds.
    \end{align*}
    After taking the expectation and making use of Fubini's theorem as well as the tower property, we get
    \begin{align}
    \label{eq:stabProofExp}
        &\E\|\Vlxrtaun\|_Y^q \le   C_q  r+(C_q  r+1)\E\|x\|_Y^q+ C_q\int_{t_\ell}^{t_\ell+r} \E\|\Vlxstaun\|_Y^{q}\ds  \nonumber\\
        &\phantom{\le }+ 2\int_{t_\ell}^{t_\ell+r}
        \Big\|\Re\Big\langle (s-t_\ell)F_\theta(x)+\int_{t_\ell}^s G_\theta(x)\,\rmd W_H(\tau),F_\theta(x)\Big\rangle_Y\Big\|_{L^{q/2}(\Omega)}^{q/2}
        \ds.
    \end{align}
    We estimate the $q/2$-th root of the last integrand. The Cauchy--Schwarz inequality, Hölder's inequality, the triangle inequality, and the Burkholder--Davis--Gundy inequalities imply
    \begin{align}
    \label{eq:stabProofTaming}
        \Big\|&\Big\langle (s-t_\ell)F_\theta(x)+\int_{t_\ell}^s G_\theta(x)\,\rmd W_H(\tau),F_\theta(x)\Big\rangle_Y\Big\|_{L^{q/2}(\Omega)}\nonumber\\
        &\le  \Big\|\Big\|(s-t_\ell)F_\theta(x)+\int_{t_\ell}^s G_\theta(x)\,\rmd W_H(\tau)\Big\|_Y\|F_\theta(x)\|_Y\Big\|_{L^{q/2}(\Omega)}\nonumber\\
        &\le  (s-t_\ell)\|F_\theta(x)\|_{L^q(\Omega;Y)}^2+\|F_\theta(x)\|_{L^q(\Omega;Y)}\Big\|\int_{t_\ell}^s G_\theta(x)\,\rmd W_H(\tau)\Big\|_{L^q(\Omega;Y)}\nonumber\\
        &\le (s-t_\ell)\|F_\theta(x)\|_{L^q(\Omega;Y)}^2+C_q\sqrt{s-t_\ell}\|F_\theta(x)\|_{L^q(\Omega;Y)}\|G_\theta(x)\|_{L^q(\Omega;\LHY)}.
    \end{align}
    To estimate this further, we leverage the stopped nonlinearities in the scheme. By definition of $F_\theta(x)$, we have $\|F_\theta(x)\|_Y\le k^{-\theta}$ a.s.\ and thus $\|F_\theta(x)\|_{L^q(\Omega;Y)}\le k^{-\theta}$, and likewise for $G_\theta(x)$. Further recall that $r\le k \le T$ and note that $\theta\in (0,\frac{1}{4}]$ implies $\min\{1-2\theta,\frac{1}{2}-2\theta\}\ge 0$. Hence, we can bound the last integral in \eqref{eq:stabProofExp} by
    \begin{align*}
        2&\int_{t_\ell}^{t_\ell+r}\Big( (s-t_\ell)\|F_\theta(x)\|_{L^q(\Omega;Y)}^2+C_q\sqrt{s-t_\ell}\|F_\theta(x)\|_{L^q(\Omega;Y)}\|G_\theta(x)\|_{L^q(\Omega;\LHY)}\Big)^{q/2}\ds\\
        &\le C_q\int_{t_\ell}^{t_\ell+r}\Big((s-t_\ell)k^{-2\theta}+\sqrt{s-t_\ell}k^{-2\theta}\Big)^{q/2}\ds \le C_q r(k^{1-2\theta}+k^{\frac{1}{2}-2\theta})^{q/2} \le C_{T,q}r,
    \end{align*}
    where we have used the non-negativity of the exponents to estimate $k^{\frac{1}{2}-2\theta} \le T^{\frac{1}{2}-2\theta}\le \max\{1,T\}$ and likewise for $k^{1-2\theta}$ in the last step.
    Inserting this in \eqref{eq:stabProofExp} and a change of variables give
    \begin{align*}
        \E\|\Vlxrtaun\|_Y^q &\le  C_{T,q} r+(C_q  r+1)\E\|x\|_Y^q  + C_q\int_0^r \E\|V_{t_\ell}^x(s\land \tau_n)\|_Y^{q}\ds
    \end{align*}
    for all $r\in [0,k]$ and with constants independent of $n$. Gronwall's inequality therefore gives \begin{align*}
        \E\|\Vlxrtaun\|_Y^q &\le  e^{C_qr}\big(C_{T,q} r+(C_q  r+1)\E\|x\|_Y^q   \big) \le e^{C_{T,q}r}\big(\E\|x\|_Y^q +C_{T,q} r  \big).
    \end{align*}
    By continuity of paths of $V_{t_\ell}^x$, we have $\tau_n \nearrow k$ and, for every $r\in [0,k]$, $V_{t_\ell}^x(r\land\tau_n)\to V_{t_\ell}^x(r)$ a.s. The assertion thus follows by Fatou's lemma.
\end{proof}

\subsection{Pointwise stability of the semilinear integrated counterparts}
	\label{subsec:stabilityUbar}
	
	Stability of the semilinear integrated counterparts $\bar{U}$ from Definition \ref{def:UbarSemilinearIntegratedCounterparts} can be obtained from stability of the scheme under the following quantified growth assumption on $Y$. 
	
\begin{assumption}[$\rho$) (Polynomial growth of nonlinearities on $Y$]
    \label{ass:polGrRho}
    For given $\rho \ge 1$, there is a constant $C_\rho\ge 0$ such that for all $x\in Y$
    \begin{equation*}
        \|F(x)\|_Y +\|G(x)\|_\LHY \le C_\rho(1+\|x\|_Y^\rho).
    \end{equation*}
\end{assumption}
	In particular, this stronger assumption implies the local $Y$-boundedness from Assumption \ref{ass:localWP}.
	
\begin{proposition}[Pointwise stability of $\bar{U}$]
\label{prop:stabUbar}
    Suppose that $-A$ generates a $C_0$-contraction semigroup on $Y$ and Assumption \ref{ass:polGrRho}($\rho$) holds for some $\rho\ge 1$. Further, suppose that Assumption \ref{ass:coYp}($\rho q$) holds for some $q\in [2,\infty)$. Let $\theta\in (0,\frac{1}{4}]$, $u_0 \in L_{\F_0}^{\rho q}(\Omega;Y)$, and assume that $F:Y\to Y$ as well as $G:Y \to \LHY$ are Borel-measurable. Then for all $T>0$ there is a constant $C_{T,q,\rho}\ge 0$ independent of $k$ and $\theta$ such that 
    \begin{align*}
        \sup_{t \in [0,T]} \|\bar{U}(t)\|_{L^q(\Omega;Y)} \le C_{T,q,\rho}\big(1+\|u_0\|_{L^{\rho q}(\Omega;Y)}^\rho\big).
    \end{align*} 
\end{proposition}

\begin{proof}   
    Again, we show the claim for $q>2$; in the proof for $q=2$ the inequalities for the terms with exponent $q-2$ are omitted. Due to the $(\rho q)$-coercivity and higher integrability of the initial values, Proposition \ref{prop:stabY} is applicable with parameter $\rho q$. It yields that
    \begin{equation}
    \label{eq:stabSchemeRhoqProof}
        \sup_{t\in [0,T]} \E\|U^k(t)\|_Y^{\rho q} = \max_{0\le j \le N} \E\|U^j\|_Y^{\rho q} \le C_{T,q,\rho}\big(1+\E\|u_0\|_Y^{\rho q}\big)
    \end{equation}
    after taking $\rho q$-th powers. In particular, $U^j \in L^{\rho q}(\Omega;Y)$ for every $0\le j \le N$ and thus $\|U^j\|_Y^{\rho q}<\infty$ almost surely. Hence, polynomial growth and $2\le q$ imply 
     \begin{align*}
        \int_0^T \|G^k(s)\|_\LHY^2 \ds \le 2C_\rho^2\int_0^T\big(1+\|U^k(s)\|_Y^{2\rho}\big)\ds = 2C_\rho^2 k\sum_{j=0}^{N-1} \big(1+\|U^j\|_Y^{2\rho}\big)<\infty.
    \end{align*}
    Similarly, we show that $F^k\in L^1(0,T;Y)$ almost surely. Hence, the Itô-type inequality from Lemma \ref{lem:ItoFormulaApprox} is applicable to $\bar{U}$. For $n\in \N$, define the stopping times
    \begin{equation*}
        \tau_n \ce \inf \{t \in [0,T]:\|\bar{U}(t)\|_Y\ge n\},\quad \inf \emptyset \ce T.
    \end{equation*}
    The Itô-type inequality \eqref{eq:lemmaItoApproxSimplified} evaluated at $t\land \tau_n$ yields
    \begin{align*}
        \|\bar{U}(t\land\tau_n)\|_Y^q &\le \|u_0\|_Y^q +  \int_0^{t\land\tau_n} \|\bar{U}\|_Y^{q-2} 
        \big(q\Re\la \bar{U},F^k\ra_Y
        + C_q  \|G^k\|_\LHY^2\big) \ds+M_{n,t},
    \end{align*}
    where $M_{n,t} \ce q \int_0^{t\land\tau_n} \|\bar{U}(s)\|_Y^{q-2} G^k(s)^*\bar{U}(s)\dWHs$ and we omitted the integration variables. The Cauchy--Schwarz inequality, Young's inequality with parameters $\frac{q}{q-1}$ and $q$ as well as $\frac{q}{q-2}$ and $\frac{q}{2}$, and polynomial growth then give 
    \begin{align*}
        \|\bar{U}(t\land\tau_n)\|_Y^q 
        &\le \|u_0\|_Y^q + C_q\Big( \int_0^{t\land\tau_n} \|\bar{U}\|_Y^q\ds +\int_0^{t\land\tau_n} \big(\|F^k\|_Y^q 
        +  \|G^k\|_\LHY^q\big)\ds\Big)+M_{n,t}\\
        &\le \|u_0\|_Y^q + C_q \int_0^{t\land\tau_n} \|\bar{U}(s)\|_Y^q\ds +C_{q,\rho}\Big(t+ \int_0^{t\land\tau_n} \|U^k(s)\|_Y^{\rho q} \ds\Big)+M_{n,t}.
    \end{align*}
    Moreover, $M_n \ce (M_{n,t})_{t\in[0,T]}$ is a square-integrable martingale, since by \eqref{eq:stabSchemeRhoqProof} and $q\ge 2$,
    \begin{align*}
        \E[M_n]_t \le q^2\E\int_0^{t\land\tau_n}\|\bar{U}(s)\|_Y^{2(q-1)}\|G^k(s)\|_\LHY^2 \ds \le C_{q,\rho} n^{2(q-1)}t \bigg(1+\sup_{s\in [0,t]}\E\|U^k(s)\|_Y^{2\rho}\bigg)<\infty.
    \end{align*}
    Hence, the expectation of $M_{n,t}$ vanishes. Taking the expectation thus results in the estimate
    \begin{align*}
        \E\|\bar{U}(t\land\tau_n)\|_Y^q 
        &\le \E\|u_0\|_Y^q + C_q \int_0^{t} \E\big[\1_{[0,\tau_n]}(s)\|\bar{U}(s)\|_Y^q\big]\ds +C_{q,\rho}\Big(t+ \E\int_0^{t\land\tau_n} \|U^k(s)\|_Y^{\rho q} \ds\Big)\\
        &\le \E\|u_0\|_Y^q + C_q \E\int_0^{t} \|\bar{U}(s\land\tau_n)\|_Y^q\ds +C_{q,\rho,T}t\big(1+\E\|u_0\|_Y^{\rho q}\big),
    \end{align*}
    where we have used \eqref{eq:stabSchemeRhoqProof} in the last inequality. Gronwall's inequality yields
    \begin{equation*}
        \E\|\bar{U}(t\land\tau_n)\|_Y^q \le e^{C_qt}\big(\E\|u_0\|_Y^q  +C_{q,\rho,T}t\big(1+\E\|u_0\|_Y^{\rho q}\big)\big)\le C_{q,\rho,T} \big(1+\E\|u_0\|_Y^q+\E\|u_0\|_Y^{\rho q}\big)
    \end{equation*}
    for all $t\in [0,T]$ with the constant $C_{q,\rho,T}$ independent of $n$, $k$, and $\theta$. Continuity of paths of $\bar{U}$ on $[0,T]$ implies $\tau_n \nearrow T$ and $\bar{U}(t\land\tau_n) \to \bar{U}(t)$ for all $t\in [0,T]$ a.s. Thus, by Fatou's lemma
    \begin{equation*}
        \E\|\bar{U}(t)\|_Y^q \le \liminf_{n\to \infty}\E\|\bar{U}(t\land\tau_n)\|_Y^q \le C_{q,\rho,T} \big(1+\E\|u_0\|_Y^q+\E\|u_0\|_Y^{\rho q}\big).
    \end{equation*}
    Noting that the right-hand side is independent of $t$ and $\rho\ge 1$, the claim follows after taking the supremum in time and the $q$-th root.
\end{proof}

\subsection{Pathwise uniform stability}
    \label{subsec:pathwiseStabBDG}

    Under a slightly stronger coercivity assumption, pointwise stability with the supremum outside of the expectation can be upgraded to pathwise uniform stability with the supremum inside. This is done in a two-step procedure inspired by \cite[Prop.~4.3]{AgrestiVeraar24CriticalVariationalSetting}. First, the additional coercivity is used to derive a bound for the expectation of one of the critical terms. Second, the pathwise supremum is estimated via a BDG-type maximal inequality and the first step.
    
	\begin{proposition}[Pathwise uniform stability of the scheme in $Y$]
		\label{prop:pathwiseUniformStabY}
		Suppose that $-A$ generates a $C_0$-contraction semigroup $(S(t))_{t\ge 0}$ on $Y$ and that $(q+\varepsilon)$-coercivity as in Assumption \ref{ass:coYp}($q+\varepsilon$) holds for some $q\in [2,\infty)$ and $\varepsilon>0$. Let $\theta\in (0,\frac{1}{4}]$, $u_0\in L_{\F_0}^q(\Omega;Y)$, and assume that $F:Y\to Y$ as well as $G:Y \to \LHY$ are Borel-measurable. 
		Then the scheme is pathwise uniformly stable in $Y$. That is, for every $T>0$, there exists a constant
		$C_{T,q,\varepsilon}\ge0$, independent of $k$, $N$, and $\theta$ such that
		\begin{align*}
			\bigg\|
			\sup_{t\in[0,T]}\|\hat U(t)\|_Y
			\bigg\|_{L^q(\Omega)}
			\le
			C_{T,q,\varepsilon}
			\big(
			1+\|u_0\|_{L^q(\Omega;Y)}
			\big).
		\end{align*}
		In particular, the same bound holds for $\|\max_{j=0,\ldots,N}\|U^j\|_Y\|_{L^q(\Omega)}$ and $\|\sup_{t\in [0,T]}\|U^k(t)\|_Y\|_{L^q(\Omega)}$.
	\end{proposition}
	
	\begin{proof}
    \textit{Step 1: Iterated Itô's formula via $(q+\varepsilon)$-coercivity.} 
    For $0\le j\le\Nk-1$ and $t\in[t_j,t_{j+1})$, define
		\begin{align}
			\label{eq:defVthetak}
			V_\theta^k(t) \ce U^j +(t-t_j)F_\theta(U^j) +\int_{t_j}^tG_\theta(U^j)\dWHs
		\end{align}
		and let $V_\theta^k(T)\ce U^N$. Then $V_\theta^k(t)=V_{t_j}^{U^j}(t-t_j)$ in the notation \eqref{eq:defVlxr} of Lemma \ref{lem:stabOneTimeStep}. Since  
        $V_\theta^k(t_j)=U^j= S(k)V_\theta^k(t_j-)$ for $1\le j\le\Nk$, contractivity of the semigroup on $Y$ implies
		\begin{align}
			\label{eq:VthetakJumpNorm}
			\|U^j\|_Y=\|V_\theta^k(t_j)\|_Y \le \|V_\theta^k(t_j-)\|_Y.
		\end{align}
        For $n\in\N$, define the stopping times
    \begin{equation*}
        \tau_n \ce \inf \{t \in [0,T]:\|V_\theta^k(t)\|_Y\ge n\},\quad \inf \emptyset \ce T.
    \end{equation*}
    Since $V_\theta^k$ is an Itô process on $[t_j,t_{j+1})$, we can apply Itô's formula on every interval
	$[t_j,t_{j+1})$, $0\le j \le N-1$, as in \eqref{eq:stabProofIto} with $x \curve U^j$, $t_\ell \curve t_j$, and $r\curve t-t_j$. After rearranging, this gives a.s.\ on the event $\{t_j\le\tau_n\}$
    \begin{align*}
			&\|V_\theta^k(t\land\tau_n)\|_Y^q 
            \le \|U^j\|_Y^q + q\int_{t_j}^{t\land\tau_n} \|V_\theta^k(s)\|_Y^{q-2} \Big(\Re\big\langle U^j,F_\theta(U^j)\big\rangle_Y+\frac{(q-1)}{2}\|G_\theta(U^j)\|_\LHY^2\Big)\ds\\
			&+ q\int_{t_j}^{t\land\tau_n} \|V_\theta^k(s)\|_Y^{q-2} \Re \big\langle V_\theta^k(s)-U^j,F_\theta(U^j)\big\rangle_Y\ds  
			+q\int_{t_j}^{t\land\tau_n}\|V_\theta^k(s)\|_Y^{q-2} (G_\theta^k(s))^*V_\theta^k(s)\dWHs\\
			&\ec \|U^j\|_Y^q+O_j(t\land\tau_n)+P_j(t\land\tau_n)+M_j(t\land\tau_n),\qquad\qquad\qquad\qquad\qquad\qquad\qquad t\in[t_j,t_{j+1}).
		\end{align*}
        By Remark \ref{rem:coercivityFGtheta}, we can estimate $O_j(t\land\tau_n)$ using $(q+\varepsilon)$-coercivity from Assumption \ref{ass:coYp}($q+\varepsilon$). Since the coercivity parameter is strictly larger than $q$, we obtain a bound not only for the norm of $V_\theta^k(t\land\tau_n)$ but also for a mixed integral containing the noise term. More precisely, a.s.\
		\begin{align*}
			O_j&(t\land\tau_n)\le q\int_{t_j}^{t\land\tau_n} \|V_\theta^k(s)\|_Y^{q-2} \Big(C_{q,\varepsilon} \big(1+\|U^j\|_Y^2\big)-\frac{\varepsilon}{2}\|G_\theta(U^j)\|_\LHY^2\Big)\ds\\
			&\le C_{q,\varepsilon} \int_{t_j}^{t\land\tau_n} \big(1+ \|V_\theta^k(s)\|_Y^q  + \|U^k(s)\|_Y^q\big) \ds - \varepsilon\frac{q}{2} \int_{t_j}^{t\land\tau_n} \|V_\theta^k(s)\|_Y^{q-2} \|G_\theta(U^j)\|_\LHY^2\ds,
		\end{align*}
		where we have applied coercivity in the first inequality and Young's inequality with exponents $\frac{q}{q-2}$ and $\frac{q}{2}$ in the second inequality in case $q>2$. For $q=2$, the inequality is immediate.
        Young's inequality with the same parameters as well as the Cauchy--Schwarz inequality yield an upper bound for $P_j(t\land \tau_n)$ such that after rearranging a.s.\ on $\{t_j\le \tau_n\}$
\begin{align*}
			\|&V_\theta^k(t\land \tau_n)\|_Y^q + \varepsilon\frac{q}{2} \int_{t_j}^{t\land \tau_n} \|V_\theta^k(s)\|_Y^{q-2} \|G_\theta^k(s)\|_\LHY^2\ds
			\le \|U^j\|_Y^q +  C_{q,\varepsilon}\int_{t_j}^{t\land \tau_n} \|V_\theta^k(s)\|_Y^q\ds\\
			&\quad + C_{q,\varepsilon} \int_{t_j}^{t\land \tau_n} \big(1+ \|U^k(s)\|_Y^q\big) \ds
			+ 2\int_{t_j}^{t\land\tau_n}  \big(\|V_\theta^k(s)-U^k(s)\|_Y\|F_\theta^k(s)\|_Y\big)^{q/2}\ds + M_j(t\land \tau_n).
		\end{align*}
        Due to \eqref{eq:VthetakJumpNorm}, we can estimate the first term on the right-hand side by $\|V_\theta^k(t_j-)\|_Y^q$ a.s.\, allowing us to insert the analogue of the above estimate on $[t_{j-1},t_j)$ at $t=t_j-$. Iterating this procedure results in a.s.\
		\begin{align}
			\label{eq:iteratedVthetaEstimate}
			\|V_\theta^k(t\land\tau_n)\|_Y^q + \varepsilon D(t\land\tau_n)
			&\le \|u_0\|_Y^q + C_{q,\varepsilon} \int_0^{t\land\tau_n} \|V_\theta^k(s)\|_Y^q\ds\nonumber\\
            &\phantom{\le }+ C_{q,\varepsilon} \int_0^{t\land\tau_n} \big(1+ \|U^k(s)\|_Y^q\big) \ds
			+ R(t\land\tau_n) + M(t\land\tau_n)
		\end{align}
		with
		\begin{align}
			\label{eq:defDPathwiseStab}
			D(t)&\ce \frac{q}{2} \int_{0}^t \|V_\theta^k(s-)\|_Y^{q-2} \|G_\theta^k(s-)\|_\LHY^2\ds,\\
			\label{eq:defRPathwiseStab}
			R(t)&\ce 2\int_0^t  \big(\|V_\theta^k(s)-U^k(s)\|_Y\|F_\theta^k(s)\|_Y\big)^{q/2}\ds,\\
\label{eq:defMPathwiseStab}
			M(t)&\ce q\int_0^t\|V_\theta^k(s-)\|_Y^{q-2} (G_\theta^k(s-))^*V_\theta^k(s-)\dWHs.
		\end{align}
        Since $V_\theta^k(s)=V_\theta^k(s-)$ outside finitely many grid points, replacing $V_\theta^k(s)$ by $V_\theta^k(s-)$ in \eqref{eq:defDPathwiseStab} and \eqref{eq:defMPathwiseStab}, and likewise for $G_\theta^k(s-)$, did not affect the value of the integrals but ensures predictability. For fixed $k$, the stopped process $(M^{\tau_n}(t))_{t\in [0,T]}$ is a continuous square-integrable martingale with $M(0)=0$ due to
        \begin{align*}
            \E\la M\ra_{T\land \tau_n} \le q^2\E \int_0^{T\land\tau_n} \|V_\theta^k(s-)\|_Y^{2q-2} \|G_\theta^k(s-)\|_\LHY^2\ds \le q^2Tn^{2q-2}k^{-2\theta}<\infty
        \end{align*}
        Hence, $\E(M(t\land\tau_n))=0$.

    \textit{Step 2: Estimate for the expectation of \eqref{eq:iteratedVthetaEstimate}.}  
    Now, we take the expectation of \eqref{eq:iteratedVthetaEstimate}, so that the martingale term  vanishes. Repeating the arguments from the proof of Lemma \ref{lem:stabOneTimeStep} (see \eqref{eq:stabProofTaming} and below) yields
    \begin{align}
        \label{eq:estExpRt}
			\E R(t\land \tau_n) &\le \E R(t)\le C_{T,q}.
		\end{align}
    Moreover, $(q+\varepsilon)$-coercivity implies $q$-coercivity. Hence, the pointwise stability estimate from Proposition \ref{prop:stabY} implies
    \begin{align*}
        \int_0^t
		\E\big[
		\one_{\{s\le\tau_n\}}
		\|U^k(s)\|_Y^q
		\big]\ds 
        \le t 
	\bigg(
		\sup_{s\in [0,T]}\E
		\|U^k(s)\|_Y^q
	\bigg) 
    \le C_{T,q} \big(1+\E\|u_0\|_Y^q\big).
    \end{align*}
    Lastly, $\1_{\{s\le\tau_n\}}
	\|V_\theta^k(s)\|_Y^q
	\le
	\|V_\theta^k(s\wedge\tau_n)\|_Y^q$. Altogether, from \eqref{eq:iteratedVthetaEstimate} we deduce
		\begin{align*}
			\E\|&V_\theta^k(t\land\tau_n)\|_Y^q + \varepsilon \E D(t\land \tau_n)
			\le C_{T,q,\varepsilon}\big(1+\E\|u_0\|_Y^q\big) + C_{q,\varepsilon} \int_0^t \E\|V_\theta^k(s\land \tau_n)\|_Y^q\ds.
		\end{align*}
    Since $\varepsilon\E D(s\land \tau_n) \ge 0$ for all $0\le s \le t$, it can be added to the integrand on the right-hand side. A deterministic Gronwall argument thus implies uniformly in $n$
        \begin{align}
\label{eq:uniformStoppedEnergyEstimate}
	\sup_{n\in\N} \big(\E\|V_\theta^k(t\land\tau_n)\|_Y^q + \varepsilon \E D(t\land \tau_n)\big)
	&\le e^{C_{q,\varepsilon}t} C_{T,q,\varepsilon}\big(1 + \E\|u_0\|_Y^q \big).
\end{align}
    By definition of the stopping time and because $\tau_n\nearrow T$ a.s.\, also $V_\theta^k(t\wedge\tau_n)\to V_\theta^k(t)$ a.s.\ for $t\in [0,T]$ as $n\to\infty$. Since $D$ is nondecreasing and continuous, $D(t\wedge\tau_n)\nearrow D(t)$. By Fatou's lemma, we conclude that for all $t\in [0,T]$,
        \begin{align}
        \label{eq:EDtfinite}
			\varepsilon \E D(t)&\le \E\|V_\theta^k(t)\|_Y^q + \varepsilon \E D(t) \le \liminf_{n \to \infty}\big(\E\|V_\theta^k(t\land\tau_n)\|_Y^q + \varepsilon \E D(t\land\tau_n)\big)\nonumber\\
            &\le e^{C_{q,\varepsilon}t}C_{T,q,\varepsilon}\big(1+\E\|u_0\|_Y^q\big)<\infty.
		\end{align}

    \textit{Step 3: Estimate of the pathwise supremum.} Again starting from the almost sure estimate \eqref{eq:iteratedVthetaEstimate}, we drop the nonnegative term $\varepsilon D(t\land\tau_n)$ and take the expectation of the supremum over $t\in [0,T_0]$, where $T_0\in (0,T]$. This results in
        \begin{align}
        \label{eq:stoppedSupremumPreBDG}
        \E\sup_{t\in[0,T_0]}\|V_\theta^k(t\land \tau_n)\|_Y^q 
			&\le \E\|u_0\|_Y^q +  C_{q,\varepsilon} \int_0^{T_0}  \big(1+\E\|V_\theta^k(s\land \tau_n)\|_Y^q + \E\|\1_{\{s\le \tau_n\}}U^k(s)\|_Y^q\big)\ds\nonumber\\
            &\phantom{\le } + \E R(T_0\land\tau_n) + \E\sup_{t\in[0,T_0]} |M(t\land \tau_n)|.
        \end{align}
        We estimate the remainder $R$ via \eqref{eq:estExpRt} and note that by \eqref{eq:VthetakJumpNorm}, for the integrand, we can estimate
        \begin{align*}
            \|\1_{\{s\le \tau_n\}}U^k(s)\|_Y^q
            \le \sup_{r\in[0,s]} \|V_\theta^k(r\wedge\tau_n)\|_Y^q.
        \end{align*}
        a.s.\ for $s\in[t_j,t_{j+1})$ because for $s\le\tau_n$, we have $\|U^k(s)\|_Y^q =\|U^j\|_Y^q	
            =\|V_\theta^k(t_j)\|_Y^q$.
        For the stopped martingale term, the definition \eqref{eq:defMPathwiseStab}, the BDG inequality, and Young's inequality imply
        \begin{align*}
            &\E\sup_{t\in[0,T_0]} |M(t\land \tau_n)| \le C_q\E\Big(\int_0^{T_0\land \tau_n}\|V_\theta^k(s-)\|_Y^{2q-2}\|G_\theta^k(s)\|_\LHY^2 \ds\Big)^{1/2}\\
            &~\le C_q\E\bigg[\Big(\sup_{t\in[0,T_0]}\|V_\theta^k(t\land \tau_n)\|_Y^q\Big)^{1/2} D(T_0\land \tau_n)^{1/2}\bigg]\le \frac{1}{2}\E\sup_{t\in[0,T_0]}\|V_\theta^k(t\land \tau_n)\|_Y^q+ C_q \E D(T_0\land \tau_n)\\
            &~\le \frac{1}{2}\E\sup_{t\in[0,T_0]}\|V_\theta^k(t\land \tau_n)\|_Y^q+ C_{T,q,\varepsilon}\big(1+\E\|u_0\|_Y^q\big). 
        \end{align*}
        In the last inequality, the estimate \eqref{eq:uniformStoppedEnergyEstimate} for $\E D(T_0\land\tau_n)$ uniformly in $n$ was essential to control the martingale. Absorbing the first term into the left-hand side of \eqref{eq:stoppedSupremumPreBDG}, we conclude for all $T_0\in(0,T]$ that
        \begin{align*}
            \E\sup_{t\in[0,T_0]}&\|V_\theta^k(t\land \tau_n)\|_Y^q 
			\le 2\E\|u_0\|_Y^q  +C_{q,\varepsilon}\int_0^{T_0}  \E\sup_{r\in [0,s]}\|V_\theta^k(r\land \tau_n)\|_Y^q \ds +  C_{T,q,\varepsilon}\big(1+\E\|u_0\|_Y^q\big).
        \end{align*}
        Gronwall's inequality therefore gives, uniformly in $n\in\N$, 
        \begin{align*}
            \sup_{n\in\N}\bigg(\E\sup_{t\in[0,T]}\|V_\theta^k(t\land\tau_n)\|_Y^q \bigg)
			&\le e^{C_{q,\varepsilon}T} C_{T,q,\varepsilon}\big(1+\E\|u_0\|_Y^q\big).
        \end{align*}
        As $n\to\infty$, monotone convergence yields
        \begin{align*}
            \E\sup_{t\in[0,T]}\|V_\theta^k(t)\|_Y^q =\lim_{n\to\infty}\E\sup_{t\in[0,T]}\|V_\theta^k(t\land\tau_n)\|_Y^q \le C_{T,q,\varepsilon}\big(1+\E\|u_0\|_Y^q\big).
        \end{align*}
        By definition of $\hat U$, $\|\hat U(t)\|_Y =\|S(t-\lfloor t\rfloor)V_\theta^k(t)\|_Y \le \|V_\theta^k(t)\|_Y$ a.s.\ for all $t\in[0,T]$,
        from which the statement follows after taking the $q$-th root. 
\end{proof}

Analogously adapting the proof of Proposition \ref{prop:stabUbar}, pathwise uniform stability of the semilinear integrated counterparts $(\bar{U}(t))_{t\in[0,T]}$ can be shown. Since this result is not used below, we omit the details.

\section{Pathwise uniform convergence rates}
\label{sec:convergenceRate}

Our first main result, Theorem \ref{thm:introMain} establishes pathwise uniform convergence at rates up to $\frac12$ in time for the nonlinearities-stopped exponential Euler scheme. It is established under a monotonicity assumption in $X$ as well as other assumptions summarised in Subsection \ref{subsec:convAssumptions}. In the proof contained in Subsection \ref{subsec:convProof}, the error is split into the difference of the scheme with the semilinear integrated counterparts $\bar{U}$ and their difference from the solution. For the first, semigroup interpolation and, for the difference of the stopped with the original nonlinearities, Markov's inequality yield decay. The remaining difference is estimated via the Itô-type inequality, monotonicity, and local Lipschitz continuity. A stochastic Gronwall argument yields a bound in terms of the first error and increments over small intervals. Higher moment estimates for the scheme and the solution finally close the estimate. In Subsection \ref{subsec:convInterpolationFullTimeInterval}, we show that the convergence rate remains true for a piecewise constant interpolation of the scheme on the full time interval up to a square-root logarithmic correction factor, cf.\ Theorem \ref{thm:convergenceExtension}.

\subsection{Assumptions}
\label{subsec:convAssumptions}

The following assumptions depend on the indexed parameters and will subsequently be used for different sets of parameters. First, we summarise the assumptions on the underlying spaces and quantify the local Lipschitz continuity and polynomial growth assumptions, where the latter is as in Assumption \ref{ass:polGrRho}. 
\begin{assumption}[$\alpha,\nu,\rho$]
\label{ass:locLipPolGro}
    Let $H$ be a real separable Hilbert space and $X,Y$ be real or complex separable Hilbert spaces such that $Y \hra X$ continuously and, for some $\alpha \in (0,\frac{1}{2}]$, let $Y\hra D_A(\alpha,\infty)$. Suppose that $-A$ generates a $C_0$-contraction semigroup $(S(t))_{t\ge 0}$ on both $X$ and $Y$. Let $F:Y \to Y$ and $G:Y \to \LHY$ be Borel-measurable and satisfy the following.
    \begin{enumerate}[label=(\alph*)]
        \item\label{assItem:locLip} \emph{Local Lipschitz continuity:} For some $\nu\ge 1$, there is $C_\nu\ge 0$ such that for all $x,y \in Y$
    	\begin{align*}
    	    \|F(x)-F(y)\|_X+\|G(x)-G(y)\|_\LHX &\le C_\nu\|x-y\|_X(1+\|x\|_Y^\nu+\|y\|_Y^\nu).
    	\end{align*}
        \item\label{assItem:polGro} \emph{Polynomial growth on $Y$:}  For some $\rho \ge 1$, there is $C_\rho\ge 0$ such that for all $x\in Y$
        \begin{align*}
            \|F(x)\|_Y +\|G(x)\|_\LHY \le C_\rho(1+\|x\|_Y^\rho).
        \end{align*}
    \end{enumerate}
\end{assumption}
The restriction to $\nu,\rho\ge 1$ is for notational convenience: Assumption \ref{ass:locLipPolGro}($\alpha,\nu,\rho$) for some $\nu,\rho\in (0,1)$ directly implies it for $\nu=\rho=1$ after potentially increasing the constant. In particular, Assumption \ref{ass:locLipPolGro} implies the local well-posedness Assumption \ref{ass:localWP} with $C_n=C_\nu(1+2n^\nu)$ and $L_n=C_\rho(1+n^\rho)$. Thus, local well-posedness for suitable initial values follows from Theorem \ref{thm:localWellPosed}.

As in \cite[Ass.~6.1]{KliobaVeraar24Rate}, the convergence analysis can be extended to random, time-dependent nonlinearities $F:\Omega \times [0,T]\times Y \to Y$ and likewise for $G$ under progressive measurability and temporal Hölder continuity assumptions if the conditions above hold uniformly in $(\omega,t)$. Moreover, the results below extend to quasi-contractive semigroups (see Definition \ref{def:quasiContractive}) with parameter $\lambda \ge 0$. Indeed, then $(e^{-\lambda t}S(t))_{t\ge 0}$ is contractive with generator $-A-\lambda I$ and the additional linear term $\lambda I$ can be absorbed into $F$ without changing the equation.

For the sake of completeness, we repeat Assumption \ref{ass:coYp} here, which was used to show global well-posedness of the equation and stability of the scheme.

\begin{assumption}[$q$) ($q$-Coercivity in $Y$]
\label{ass:coYq}
    For some $q\in [2,\infty)$, there is a constant $C_q\ge 0$ such that for all $x\in Y$,
    \begin{align*}
        \Re \la x,F(x)\ra_Y + \frac{q-1}{2}\|G(x)\|_\LHY^2 \le C_q(1+\|x\|_Y^2).
    \end{align*}
\end{assumption}

Assumptions \ref{ass:locLipPolGro}($\alpha,\nu,\rho$) and Assumption \ref{ass:coYq}($q$)
 for some $q\in [2,\infty)$ together imply pointwise stability of the scheme and an a priori estimate for its $q$-th moments by Proposition \ref{prop:stabY}. If, in addition, Assumption \ref{ass:coYq}($\tilde{q}$) holds for some $\tilde{q}>q$, then Theorem \ref{thm:aprioriY} ensures global well-posedness of the stochastic evolution equation and a pathwise uniform a priori estimate for the solution. Moreover, Proposition \ref{prop:pathwiseUniformStabY} then gives pathwise uniform stability of the scheme.
 
 To show convergence rates, we further impose the following monotonicity condition, which is also referred to as a \textit{one-sided Lipschitz assumption}.
\begin{assumption}[$q,\eta$) (Monotonicity condition]
\label{ass:monqEta}
    For some $q \in [2,\infty)$ and $\eta>0$, there is a constant $C_{q,\eta}\ge 0$ such that for all $x,y \in Y$,
	\begin{equation*}
	     \Re \la x-y,F(x)-F(y)\ra_X + (1+\eta) \frac{q-1}{2} \|G(x)-G(y)\|_\LHX^2 \le C_{q,\eta} \|x-y\|_X^2.
	\end{equation*}
\end{assumption}

\begin{remark}
    Note that the coercivity and monotonicity assumptions become stronger as the parameter $q$ increases: Assumptions \ref{ass:coYq}($q$) and \ref{ass:monqEta}($q,\eta$) for some $q\ge 2$ and $\eta>0$ imply Assumptions \ref{ass:coYq}($\tilde{q}$) and \ref{ass:monqEta}($\tilde{q},\eta$) for all $\tilde{q}\in [2,q]$, respectively. 
\end{remark}

Recall the definitions of the piecewise constant and continuous extensions $U^k$ and $\hat{U}$ as well as the definition of the semilinear integrated counterparts $\bar{U}$ from Definitions \ref{def:UkExtension}, \ref{def:UhatExtension}, and \ref{def:UbarSemilinearIntegratedCounterparts} respectively. Abbreviate
\begin{equation*}
    \|u\|_{q,Z}\ce \|u\|_{L^q(\Omega;Z)},\quad \nn \phi \nn_Z\ce \|\phi\|_{\calL_2(H,Z)},\quad \nn v\nn_{q,Z}\ce \|v\|_{L^q(\Omega;\calL_2(H,Z))},\quad \|\zeta\|_q \ce \|\zeta\|_{L^q(\Omega)}
\end{equation*}
for $q \in [2,\infty)$, Hilbert spaces $Z$, and $u,\phi,v,\zeta$ in the respective spaces.

\subsection{Main error estimate}
\label{subsec:convProof}
\subsubsection{Error between the scheme and the semilinear integrated counterparts}

To prove an error estimate for the difference between the scheme $(U^j)_j$ and the semilinear integrated counterparts $(\bar{U}(t))_{t\in [0,T]}$, we actually show a pathwise uniform error estimate for the difference $\hat{U}(t)-\bar{U}(t)$ with the continuous extension $\hat{U}$ of the scheme. At the grid points, this yields the error estimate for the scheme. Since part of the error occurs due to comparing the stopped with the unstopped nonlinearities, the following consequence of Markov's inequality proves useful.

\begin{lemma}
	\label{lem:Markov2}
	Let $\Phi:\Omega \to Y$ be measurable, $\theta \in (0,\frac{1}{4}]$, and $q\in [2,\infty)$. Suppose that $F:Y \to Y$ and $G:Y \to \LHY$ are Borel-measurable. Then for all $\kappa \in [\theta/q,\infty)$ and $k>0$,
	\begin{align*}
		\big\|1-\1&_{\{\|F(\Phi)\|_Y+\|G(\Phi)\|_{\calL_2(H,Y)} \le k^{-\theta}\}}\big\|_{L^q(\Omega)}\le k^\kappa \big(\|F(\Phi)\|_{L^{\kappa q/\theta}(\Omega;Y)}+\|G(\Phi)\|_{L^{\kappa q/\theta}(\Omega;\LHY)}\big)^{\kappa/\theta}.
	\end{align*}
\end{lemma}
\begin{proof}
	An application of Markov's inequality and the triangle inequality yields
	\begin{align*}
		\big\|1-\1_{\{\|F(\Phi)\|_Y+\nn G(\Phi)\nn_Y \le k^{-\theta}\}}\big\|_{q} 
        &= \big|\P\big(\|F(\Phi)\|_Y+\nn G(\Phi)\nn_Y > k^{-\theta}\big)\big|^{1/q}\\
		&= \big|\P\big((\|F(\Phi)\|_Y+\nn G(\Phi)\nn_Y)^{\kappa q/\theta} > k^{-\kappa q}\big)\big|^{1/q}\\
		&\le \Big|k^{\kappa q} \E\big(\|F(\Phi)\|_Y+\nn G(\Phi)\nn_Y\big)^{\kappa q/\theta} \Big|^{1/q}\\
		&\le k^\kappa \big(\|F(\Phi)\|_{\kappa q/\theta,Y}+\nn G(\Phi)\nn_{\kappa q/\theta,Y}\big)^{\kappa/\theta}. \qedhere
	\end{align*}
\end{proof}
 
\begin{proposition}
	\label{prop:UhatUbarError}
	Let $q \in [2,\infty)$ and $\theta \in (0,\frac{1}{4}]$, and let Assumption \ref{ass:locLipPolGro}($\alpha,\nu,\rho$) hold for some $\alpha \in (0,\frac{1}{2}]$ and $\nu,\rho\ge 1$. 
    Suppose that $(1+\frac{\alpha}{\theta})\rho q$-coercivity as in Assumption \ref{ass:coYq}($(1+\frac{\alpha}{\theta})\rho q$) holds and that $u_0\in L_{\F_0}^{(1+\frac{\alpha}{\theta})\rho q}(\Omega;Y)$. Then there is a constant $C_{\alpha,T,q,\theta,\rho}\ge 0$ not depending on $k$, $N$, or $u_0$ such that
    \begin{align*}
		\bigg\|&\sup_{t\in [0,T]} \|\hat{U}(t)-\bar{U}(t)\|_X\bigg\|_{L^q(\Omega)} 
        \le  C_{\alpha,T,q,\theta,\rho}  \Big( 1+\|u_0\|_{(1+\frac{\alpha}{\theta})\rho q,Y}^{\rho(1+\frac{\alpha}{\theta})}\Big) k^\alpha.
	\end{align*}
	In particular, the upper bound holds for $\|\max_{0 \le j \le \Nk } \|U^j-\bar{U}(t_j)\|_X\|_{L^q(\Omega)} $ and the pathwise uniform error of the scheme $(U^j)_j$ and the semilinear integrated counterparts $\bar{U}$  converges at rate $\alpha$.
\end{proposition}
Local Lipschitz continuity as in Assumption \ref{ass:locLipPolGro}\ref{assItem:locLip} is not required in the following proof.
\begin{proof}
	By \ref{eq:UhatMildFormula} and Definition \ref{def:UbarSemilinearIntegratedCounterparts}, we can split the error into
	\begin{align*}
		\hat{U}(t)-\bar{U}(t) &= \int_0^t [S(t-\floors)-S(t-s)] F_\theta^k(s)\ds +\int_0^t S(t-s) [F_\theta^k(s) -F^k(s)]\ds\\
		&\phantom{= }+\int_0^t [S(t-\floors)-S(t-s)] G_\theta^k(s)\dWHs+\int_0^t S(t-s) [G_\theta^k(s) -G^k(s)]\dWHs\\
		&\ec E_{F,1}(t)+E_{F,2}(t)+E_{G,1}(t)+E_{G,2}(t).
	\end{align*}
	We start by estimating $E_{G,1}(t)$ via the BDG-type inequality for stochastic convolutions from Theorem \ref{thm:maximal-inequality} and obtain decay from the semigroup difference estimate of Lemma \ref{lem:sgInterpolation}. Polynomial growth of $G$ and $\1_{\theta,\ell}\le 1$ a.s.\ then give
	\begin{align}
    \label{eq:estimateProofMinusTheta}
		\bigg\|\sup_{t\in [0,T]} \|E_{G,1}(t)\|_X\bigg\|_q
        &\le C_q \Big(\int_0^T \bignn[S(s-\floors)-I] G_\theta^k(s)\bignn_{q,X}^2\ds\Big)^{1/2}\nonumber\\
		&\le C_{\alpha,q} \Big(\sum_{\ell=0}^{\Nk -1}\int_{t_\ell}^{t_{\ell+1}} (s-t_\ell)^{2\alpha} \nn\1_{\theta,\ell} G(U^\ell)\nn_{q,Y}^2 \ds \Big)^{1/2} \nonumber\\
		&\le C_{\alpha,q} \sqrt{T} k^\alpha \Big(\max_{0 \le \ell \le \Nk -1} \nn\1_{\theta,\ell} G(U^\ell)\nn_{q,Y}\Big)\\
		&\le C_{\alpha,q,\rho} \sqrt{T} k^\alpha \Big( 1+\max_{0 \le \ell \le \Nk -1}\|U^\ell\|_{\rho q,Y}^\rho\Big). \nonumber
	\end{align}
    Because $(1+\frac{\alpha}{\theta})\rho q\ge \rho q$ and Assumption \ref{ass:coYq}($(1+\frac{\alpha}{\theta})\rho q$) holds as well as $u_0\in L^{(1+\frac{\alpha}{\theta})\rho q}(\Omega;Y)$, Proposition \ref{prop:stabY} is applicable with $q \curve \rho q$ and provides pointwise stability of the $\rho q$-th moments of the scheme in $Y$. Hence,
    \begin{align*}
		\bigg\|\sup_{t\in [0,T]} \|E_{G,1}(t)\|_X\bigg\|_q
        &\le C_{\alpha,T,q,\rho} \big(1+\|u_0\|_{\rho q,Y}^\rho\big)  k^\alpha.
	\end{align*}
    Analogously, applying the triangle rather than the BDG-type inequality, the same upper bound is obtained for $E_{F,1}$, potentially after adjusting the constant.
    
	To estimate $E_{G,2}$, the BDG-type inequality and Hölder's inequality in $\Omega$ with parameters $(1+\frac{\theta}{\alpha})q$ and $(1+\frac{\alpha}{\theta})q$ result in 
	\begin{align*}
		\bigg\|\sup_{t\in [0,T]} \|E_{G,2}(t)\|_X\bigg\|_q
		&\le C_q \Big(\int_0^T \bignn[\1_\theta(U^k(s)) -1]G^k(s)\bignn_{q,X}^2\ds\Big)^{1/2}\\
		&\le C_q \sqrt{T} \Big(\max_{0 \le \ell \le \Nk -1}\nn G(U^\ell)\nn_{(1+\frac{\alpha}{\theta})q,X}\Big) \Big(\max_{0 \le \ell \le \Nk -1}\|\1_{\theta,\ell}-1\|_{(1+\frac{\theta}{\alpha})q,\R}\Big).
	\end{align*}
    We estimate both factors separately. First, the continuous embedding $Y\hra X$, polynomial growth of $G$, and pointwise stability of the $(1+\frac{\alpha}{\theta})\rho q$-th moments of the scheme in $Y$ from Proposition \ref{prop:stabY} with $q \curve (1+\frac{\alpha}{\theta})\rho q$ imply 
    \begin{align*}
        \max_{0 \le \ell \le \Nk -1}\nn G(U^\ell)\nn_{(1+\frac{\alpha}{\theta})q,X} 
        &\le C_\rho\Big(1+\max_{0 \le \ell \le \Nk -1}\|U^\ell\|_{(1+\frac{\alpha}{\theta})\rho q,Y}^\rho\Big)
        \le C_{T,q,\theta,\rho}\Big(1+\|u_0\|_{(1+\frac{\alpha}{\theta})\rho q,Y}^\rho\Big).
    \end{align*}
    Second, Lemma \ref{lem:Markov2} is applicable with $\Phi \curve U^\ell$, $q \curve (1+\frac{\theta}{\alpha})q$, and $\kappa=\alpha$ because of $\alpha>\frac{\alpha}{q} >\frac{\theta}{(\alpha+\theta)q}\alpha= \frac{\theta}{(1+\frac{\theta}{\alpha})q}$. It yields
    \begin{align*}
        \max_{0 \le \ell \le \Nk -1}\|\1_{\theta,\ell}-1\|_{(1+\frac{\theta}{\alpha})q,\R} &\le k^\alpha \max_{0 \le \ell \le \Nk -1}\big(\|F(U^\ell)\|_{(1+\frac{\alpha}{\theta})q,Y}+\nn G(U^\ell)\nn_{(1+\frac{\alpha}{\theta})q,Y}\big)^{\alpha/\theta}\\
        &\le C_{\alpha,T,q,\theta,\rho}  \Big(1+\|u_0\|_{(1+\frac{\alpha}{\theta})\rho q,Y}^{\rho \alpha/\theta}\Big)k^\alpha,
	\end{align*}
    where in the last inequality, we have used polynomial growth and Proposition \ref{prop:stabY} with $q \curve (1+\frac{\alpha}{\theta})\rho q$ as for the first term.
    In conclusion,
    \begin{align*}
		\bigg\|\sup_{t\in [0,T]} \|E_{G,2}(t)\|_X\bigg\|_q
		&\le C_{\alpha, T,q,\theta,\rho}\Big(1+ \|u_0\|_{(1+\frac{\alpha}{\theta})\rho q,Y}^{\rho(1+\frac{\alpha}{\theta})}\Big)k^\alpha.
	\end{align*}
	The same upper bound can be obtained for $E_{F,2}$ after changing the constant if necessary.
	Combining all four estimates, the claimed estimate follows.
\end{proof}

\begin{remark}
    In \eqref{eq:estimateProofMinusTheta}, it would also be possible to estimate by $C k^{\alpha-\theta}$ leveraging the stopping of the nonlinearities. However, this deteriorates the convergence rate unnecessarily. Indeed, polynomial growth, the coercivity assumption with sufficiently large parameter, and sufficiently high moments of the initial values suffice, as demonstrated in the proof.
\end{remark}

\subsubsection{Error between the mild solution and the semilinear integrated counterparts}

We now pass to the analysis of the error arising from $\bar{U}-U$.

\begin{lemma}
\label{lem:ItoFormulaBarU-U}
    Let $q\in (2,\infty)$, $\theta\in (0,\frac{1}{4}]$, and $u_0\in L_{\F_0}^q(\Omega;Y)$. Suppose that Assumption \ref{ass:locLipPolGro}($\alpha,\nu,\rho$) and Assumption \ref{ass:coYq}($q$) hold for some $\alpha \in (0,\frac12]$ and $\nu,\rho\ge 1$. Then for all $t\in[0,T]$ almost surely
    \begin{align}
        \label{eq:ItoFormulaBarU-U}
	   \|\bar{U}(t)-U(t)\|_X^q &\le q \int_0^t \|\bar{U}(s)-U(s)\|_X^{q-2} \Re\la\bar{U}(s)-U(s),F^k(s)-F(U(s))\ra_X \ds\nonumber\\
	   &\phantom{\le }+ \frac{q(q-1)}{2}\int_0^t \|\bar{U}(s)-U(s)\|_X^{q-2} \| G^k(s)-G(U(s))\|_\LHX^2\ds + M_t 
    \end{align}
    with the continuous local martingale
    \begin{equation*}
        M_t \ce q \int_0^t \|\bar{U}(s)-U(s)\|_X^{q-2} [G^k(s)-G(U(s))]^*[\bar{U}(s)-U(s)]\dWHs.
    \end{equation*}
\end{lemma}

\begin{proof}
    Our aim is to apply the Itô-type inequality \eqref{eq:lemmaItoApproxSimplified} to $\bar{U}-U$ on $Z=X$ with $p \curve q$. Since
    \begin{align*}
	    \bar{U}(t)-U(t) = \int_0^t S(t-s)[F^k(s)-F(U(s))]\ds+\int_0^t S(t-s)[G^k(s)-G(U(s))]\dWHs,
    \end{align*}
    we verify that it is applicable with $u_0=0$, $f=F^k-F\circ U$ and $g=G^k-G\circ U$, meaning that we have to show $F^k-F\circ U\in L^1(0,T;X)$ and $G^k-G\circ U \in L^2(0,T;\LHX)$ a.s. Progressive measurability follows from Borel measurability of $F$ and $G$, since $U$ and $U^k$ are progressively measurable as adapted and continuous or piecewise constant processes.
    By polynomial growth of $F$ on $Y$ and $Y\hra X$,
    \begin{align*}
    	\int_0^T &\|F(U^k(s))-F(U(s))\|_X\ds \le C \int_0^T \big(\|F(U^k(s))\|_Y+\|F(U(s))\|_Y\big)\ds\\
        &\le C_\rho \int_0^T \big(1+\|U^k(s)\|_Y^\rho+\|U(s)\|_Y^\rho\big)\ds = C_\rho \Big(T + \|U\|_{L^\rho(0,T;Y)}^\rho + k\sum_{j=0}^{\Nk -1} \|U^j\|_Y^\rho\Big).
    \end{align*}
    Likewise, for $G$
    \begin{align*}
    	\int_0^T \nn G(U^k(s))-G(U(s))\nn_X^2\ds 
         \le C_\rho \Big(T + \|U\|_{L^{2\rho}(0,T;Y)}^{2\rho} + k\sum_{j=0}^{\Nk -1} \|U^j\|_Y^{2\rho}\Big).
    \end{align*}
    Set $\lambda_0\ce \frac{2}{q}$. Then Assumption \ref{ass:coYq}($q$) is precisely the coercivity assumption required to apply Theorem \ref{thm:aprioriY} with $q \curve \lambda_0q=2$ and $\lambda \curve \lambda_0$. The required integrability of the initial values follows from $u_0\in L_{\F_0}^q(\Omega;Y)\seq L_{\F_0}^2(\Omega;Y)$. Therefore, Theorem \ref{thm:aprioriY} yields that $U\in C([0,T];Y)$ a.s.\ and thus, in particular, $U \in L^{2\rho}(0,T;Y)$ a.s. Furthermore, Proposition \ref{prop:stabY} implies that, in particular, $\|U^j\|_Y<\infty$ a.s.\ for every  $0 \le j \le \Nk $. Since there are only finitely many grid points, the sum is finite a.s.\ In conclusion, Lemma \ref{lem:ItoFormulaApprox} is applicable and yields the desired estimate.
    
    It remains to show that $(M_t)_{t\in [0,T]}$ is a continuous local martingale. Define the stopping times
    \begin{equation*}
        \tau_n \ce \inf\{t\in [0,T]:\|\bar{U}(t)-U(t)\|_X \ge n\},\quad \inf \emptyset \ce T
    \end{equation*}
    for $n\in \N$. Then almost surely the quadratic variation satisfies
    \begin{align*}
        [M]_{T\land \tau_n} &\le q^2 \int_0^{T\land \tau_n} \|\bar{U}(s)-U(s)\|_X^{2q-2} \nn G^k(s)-G(U(s))\nn_X^2 \ds\\
        &\le q^2n^{2q-2} C_\rho \int_0^{T\land \tau_n} \big(1+\|U^k(s)\|_Y^{2\rho}+\|U(s)\|_Y^{2\rho}\big)\ds <\infty,
    \end{align*}
    where the integral is finite almost surely by the reasoning above, regardless of the stopping time. This finishes the proof, since continuity of paths of $U-\bar{U}$ gives $\tau_n \nearrow T$ almost surely.
\end{proof}

The following lemma contains a central step of the error estimate, which is a suitable application of the stochastic Gronwall inequality \cite[Lemma~A.7]{surveyMarkAntonio25}.

\begin{lemma}
\label{lem:reduceUbarUtoUkUbar}
    Let $p\in [2,\infty)$, $\lambda \in (0,1)$,  $q=\frac{p}{\lambda}$, and $\theta \in (0,\frac{1}{4}]$. Suppose that Assumption \ref{ass:locLipPolGro}($\alpha,\nu,\rho$) holds for some $\alpha\in (0,\frac{1}{2}]$ and $\nu,\rho\ge 1$. Let $r_1=\rho q(1+\nu+\frac{\alpha}{\theta})$, $r_2=q(1+\frac{\theta\nu}{\alpha+\theta})$, and $u_0\in L_{\F_0}^{r_1}(\Omega;Y)$. Further suppose that Assumptions \ref{ass:coYq}($r_1$) and \ref{ass:monqEta}($q,\eta$) hold for some $\eta>0$. Then there is $C_{\lambda,p,\nu,\rho,T,\eta,\theta}\ge 0$ such that
    \begin{align*}
	   \bigg\|\sup_{t\in[0,T]} \|\bar{U}(t)-U(t)\|_X\bigg\|_{L^p(\Omega)}&\le C_{\alpha,T,p,\lambda,\theta,\nu,\rho,\eta} \big(1+\|u_0\|_{r_1,Y}^{\rho\nu }\big)\bigg(\sup_{t\in[0,T]} \|U^k(t)-\bar{U}(t)\|_{L^{r_2}(\Omega;X)}\bigg).
    \end{align*}
\end{lemma}

\begin{proof}
    Due to $q=\frac{p}{\lambda}>2$ and $r_1\ge q$, the Itô-type inequality for $\|\bar{U}(t)-U(t)\|_X^q$ from Lemma \ref{lem:ItoFormulaBarU-U} holds. After adding and subtracting $F(\bar{U}(s))$ and $G(\bar{U}(s))$ in \eqref{eq:ItoFormulaBarU-U}, an application of the Cauchy--Schwarz inequality and Young's inequality with parameter $\eta$ as in Assumption \ref{ass:monqEta} and $C_\eta = 1 + \eta^{-1}$ give
\begin{align*}
	\|&\bar{U}(t)-U(t)\|_X^q \le q \int_0^t \|\bar{U}-U\|_X^{q-2} \Re\big(\la\bar{U}-U,F(\bar{U})-F(U)\ra_X+\la\bar{U}-U,F^k- F(\bar{U})\ra_X\big) \ds\nonumber\\
	   &\phantom{\le }+ \frac{q(q-1)}{2}\int_0^t \|\bar{U}-U\|_X^{q-2} \big(\nn G^k- G(\bar{U})\nn_X+\nn G(\bar{U})-G(U)\nn_X\big)^2\ds + M_t \\
    &\le q \int_0^t \|\bar{U}-U\|_X^{q-2} \Big(\Re \big\la\bar{U}-U,F(\bar{U})-F(U)\big\ra_X+ (1+\eta)\frac{q-1}{2} \nn G(\bar{U})-G(U)\nn_X^2\Big) \ds\\
    &+q \int_0^t \|\bar{U}-U\|_X^{q-1} \|F^k-F(\bar{U})\|_X \ds+ \frac{q(q-1)}{2}C_\eta\int_0^t \|\bar{U}-U\|_X^{q-2} \nn G^k-G(\bar{U})\nn_X^2\ds+ M_t\\ 
    &\ec I_1+I_2+I_3+M_t 
\end{align*}
for all $t\in [0,T]$ almost surely. Here, $M_t$ is the continuous local martingale from Lemma \ref{lem:ItoFormulaBarU-U} and the integration variable $s$ is omitted.
For the first term, the monotonicity condition from Assumption \ref{ass:monqEta}($q,\eta$) yields
for all $t\in [0,T]$ a.s.\
\begin{equation*}
    I_1 \le C_{q,\eta}  \int_0^t\|\bar{U}(s)-U(s)\|_X^q\ds.
\end{equation*}
Young's inequality with parameters $\frac{q}{q-1}$ and $q$ as well as local Lipschitz continuity of $F$ as in Assumption \ref{ass:locLipPolGro}\ref{assItem:locLip} imply
\begin{align*}
    I_2 &\le C_{q}\int_0^t \|\bar{U}-U\|_X^q \ds 
    + \int_0^t \|F^k-F(\bar{U})\|_X^q \ds\\
	&\le C_q\int_0^t \|\bar{U}-U\|_X^q \ds 
    + C_{q,\nu} \int_0^t \|U^k-\bar{U}\|_X^q (1+\|U^k\|_Y^\nu+\|\bar{U}\|_Y^\nu)^q \ds.
\end{align*}
To estimate $I_3$, we repeat the argument with parameters $\frac{q}{q-2}$ and $\frac{q}{2}$ for Young's inequality and local Lipschitz continuity of $G$ to obtain 
\begin{align*}
    I_3 \le C_{q,\eta}\int_0^t \|\bar{U}-U\|_X^q \ds 
    + C_{q,\nu,\eta} \int_0^t \|U^k-\bar{U}\|_X^q \big(1+\|U^k\|_Y^\nu+\|\bar{U}\|_Y^\nu\big)^q \ds.
\end{align*}
Having estimated $I_1$ to $I_3$, we conclude that for all $t \in [0,T]$ a.s.
\begin{align*}
	\varphi(t)\ce\|\bar{U}(t)-U(t)\|_X^q \le C_{q,\eta}\int_0^t \varphi(s) \ds +H_t+M_t
\end{align*}
with the continuous local martingale $(M_t)_{t\ge 0}$ satisfying $M_0=0$ and the adapted, non-negative, continuous, and increasing process $(H_t)_{t\ge 0}$ defined by
\begin{align*}
	H_t\ce C_{q,\nu,\eta} \int_0^t \|U^k(s)-\bar{U}(s)\|_X^q \big(1+\|U^k(s)\|_Y^\nu+\|\bar{U}(s)\|_Y^\nu\big)^q \ds.
\end{align*}
Thus, the stochastic Gronwall lemma \cite[Corollary~5.4]{Geiss24} is applicable with $A_s= C_{q,\eta}s$ and gives for all $\lambda\in(0,1)$ and $T>0$,
\begin{equation*}
\|e^{-A_T} \varphi_T^*\|_{L^\lambda(\Omega)} \le (1-\lambda)^{-1/\lambda}\lambda^{-1}\|H_T\|_{L^\lambda(\Omega)}.
\end{equation*}
Taking $\lambda$-th powers and using that $A_T$ is deterministic, the estimate simplifies to
\begin{equation*}
    \|\varphi_T^*\|_{L^\lambda(\Omega)}^\lambda 
    \le e^{\lambda A_T}  (1-\lambda)^{-1}\lambda^{-\lambda}\|H_T\|_{L^\lambda(\Omega)}^\lambda 
    = C_{T,q,\lambda,\eta}\E H_T^\lambda 
    \le C_{T,q,\lambda,\eta} (\E H_T)^\lambda,
\end{equation*}
where Jensen's inequality was used in the last step to pull the expectation inside the $\lambda$-th power. Since $\lambda q= p$, we deduce
\begin{align*}
	\E\sup_{t\in[0,T]} \|\bar{U}(t)-U(t)\|_X^p&=\E\bigg(\sup_{t\in[0,T]} \|\bar{U}(t)-U(t)\|_X^q\bigg)^\lambda=\|\varphi_T^*\|_{L^\lambda(\Omega)}^\lambda \\
	&\le C_{T,q,\lambda,\nu,\eta} \Big(\E \int_0^T \|U^k(s)-\bar{U}(s)\|_X^q\big(1+\|U^k(s)\|_Y^\nu+\|\bar{U}(s)\|_Y^\nu\big)^q\ds\Big)^\lambda. 
\end{align*}
 Let $\beta_2=1+\frac{\theta\nu}{\alpha+\theta}\in(1,\infty)$ and $\beta_1=1+\frac{\alpha+\theta}{\theta\nu}\in(1,\infty)$, which satisfy $\frac{1}{\beta_1}+\frac{1}{\beta_2}=1$. Fubini's theorem and Hölder's inequality in $\Omega$ with parameters $\beta_1 q$ and $\beta_2 q$ then imply
    \begin{align}
    \label{eq:differentBetaLater}
	   \E&\sup_{t\in[0,T]} \|\bar{U}(t)-U(t)\|_X^{p}
	   \le C  \Big(\int_0^T \E\big[\|U^k(s)-\bar{U}(s)\|_X^q(1+\|U^k(s)\|_Y^\nu+\|\bar{U}(s)\|_Y^\nu)^q\big]\ds\Big)^\lambda\nonumber\\
       &\le C \Big(\int_0^T \|U^k(s)-\bar{U}(s)\|_{\beta_2q,X}^q\big\|1+\|U^k(s)\|_Y^\nu+\|\bar{U}(s)\|_Y^\nu\big\|_{\beta_1q}^q\ds\Big)^\lambda\nonumber\\
       &\le C T^\lambda \bigg(\sup_{t\in[0,T]} \|U^k(t)-\bar{U}(t)\|_{\beta_2q,X}^{\lambda q}\bigg) \bigg(1+\sup_{t\in[0,T]}\|U^k(t)\|_{\beta_1\nu q,Y}^{\nu\lambda q}+\sup_{t\in[0,T]}\|\bar{U}(t)\|_{\beta_1\nu q,Y}^{\nu\lambda q}\bigg),
    \end{align}
    where $C$ depends on $ (\alpha,T,q,\lambda,\nu,\eta)$.
    Pointwise stability of the scheme and the semilinear integrated counterparts yield the statement of the lemma. Indeed, Proposition \ref{prop:stabY} is applicable with $q\curve \beta_1\nu q=\frac{r_1}{\rho}$ due to Assumption \ref{ass:coYq}($\frac{r_1}{\rho}$), which is satisfied since $\frac{r_1}{\rho}\le r_1$, and $u_0\in L_{\F_0}^{r_1/\rho}(\Omega;Y)$. It gives
    \begin{equation*}
        \sup_{t\in[0,T]}\|U^k(t)\|_{\beta_1\nu q,Y}\le \sup_{t\in [0,T]} \|\hat{U}(t)\|_{\beta_1\nu q,Y} 
        \le C_{\alpha,T,q,\theta,\nu} (1+\|u_0\|_{\beta_1\nu q,Y})=C_{\alpha,T,q,\theta,\nu} (1+\|u_0\|_{r_1/\rho,Y}).
    \end{equation*}
    Likewise, an application of Proposition \ref{prop:stabUbar} with $q\curve \beta_1\nu q=\frac{r_1}{\rho}$ requires Assumption \ref{ass:coYq}($\rho \beta_1\nu q$) and $u_0\in L_{\F_0}^{\rho \beta_1\nu q}(\Omega;Y)$, which are satisfied due to $\rho \beta_1\nu q=r_1$, and results in
\begin{align*}
    \sup_{t \in [0,T]} \|\bar{U}(t)\|_{\beta_1\nu q,Y} \le C_{\alpha,T,q,\theta,\nu,\rho}(1+\|u_0\|_{\rho \beta_1\nu q,Y}^\rho)=C_{\alpha,T,q,\theta,\nu,\rho}(1+\|u_0\|_{r_1,Y}^\rho). 
\end{align*}
    Via $\beta_2q=r_2$ and $\|u_0\|_{r_1/\rho,Y}\le 1+\|u_0\|_{r_1,Y}^\rho$, which follows from Hölder's inequality and $\rho\ge 1$, we deduce by inserting the two stability estimates that
    \begin{align*}
	   \E\sup_{t\in[0,T]} \|\bar{U}(t)-U(t)\|_X^p&\le C_{\alpha,T,p,\lambda,\theta,\nu,\rho,\eta} \big(1+\|u_0\|_{L^{r_1}(\Omega;Y)}^{\rho\nu p}\big)\bigg(\sup_{t\in[0,T]} \|U^k(t)-\bar{U}(t)\|_{L^{r_2}(\Omega;X)}^p\bigg).
    \end{align*}
    Taking the $p$-th root finishes the proof.
\end{proof}

Lastly, an estimate for the moments of $U^k-\bar{U}$ arising in Lemma \ref{lem:reduceUbarUtoUkUbar} has to be established. The difference between the piecewise constant and the continuous extensions can be decomposed into the difference $\hat{U}-\bar{U}$ between the continuous extension and $\bar{U}$, which we already dealt with in Proposition \ref{prop:UhatUbarError}, and the difference $U^k-\hat{U}$ between the two different extensions. For the latter, we can leverage that their values at all grid points agree, meaning that it suffices to consider the differences over small time intervals. 

\begin{proposition}
\label{prop:diffUkUbar}
   Let $p\in [2,\infty)$, $\lambda \in (0,1)$,  $q=\frac{p}{\lambda}$, and $\theta \in (0,\frac{1}{4}]$. Suppose that Assumption \ref{ass:locLipPolGro}($\alpha,\nu,\rho$) holds for some $\alpha \in (0,\frac{1}{2}]$ and $\nu,\rho\ge 1$. Let $r_1=\rho q(1+\nu+\frac{\alpha}{\theta})$, $r_2=q(1+\frac{\theta\nu}{\alpha+\theta})$, and $u_0\in L_{\F_0}^{r_1}(\Omega;Y)$. Further suppose that Assumption \ref{ass:coYq}($r_1$) holds. Then there is $C_{\alpha,T,p,\theta,\lambda,\nu,\rho}\ge 0$ such that
    \begin{align}
    \label{eq:diffUkUbarProp}
        \sup_{t\in[0,T]} \|U^k(t)-\bar{U}(t)\|_{L^{r_2}(\Omega;X)} 
        &\le C_{\alpha,T,p,\theta,\lambda,\nu,\rho}\Big( 1+\|u_0\|_{r_1,Y}^{\rho(1+\frac{\alpha}{\theta})}\Big)k^\alpha.
    \end{align}
\end{proposition}
\begin{proof}
    Before passing to the estimate, note that
    \begin{equation}
    \label{eq:r1r2coefficientsProof}
        r_2 \le \rho r_2 \le \Big(1+\frac{\alpha}{\theta}\Big)\rho r_2 = \Big(1+\frac{\alpha}{\theta}\Big)\rho q\Big(1+\frac{\theta\nu}{\alpha+\theta}\Big) = \rho q\Big(1+\nu+\frac{\alpha}{\theta}\Big) = r_1.
    \end{equation}
    An application of the triangle inequality and pulling the supremum inside the first norm result in
    \begin{align}
    \label{eq:splitUhatUbarProof}
        \sup_{t\in[0,T]} \|U^k(t)-\bar{U}(t)\|_{r_2,X} 
        &\le  \bigg\|\sup_{t\in[0,T]}\|\hat{U}(t)-\bar{U}(t)\|_X\bigg\|_{r_2}+\sup_{t\in[0,T]}\|\hat{U}(t)-U^k(t)\|_{r_2,X}.
    \end{align}
    To estimate the first term, we apply Proposition \ref{prop:UhatUbarError} with $q \curve r_2$, which requires Assumption \ref{ass:coYq}($(1+\frac{\alpha}{\theta})\rho r_2$) and the respective moments of $u_0$. This is satisfied due to \eqref{eq:r1r2coefficientsProof},
    so that
    \begin{align*}
		\bigg\|\sup_{t\in [0,T]} \|\hat{U}(t)-\bar{U}(t)\|_X\bigg\|_{r_2} 
        &\le  C_{\alpha,T,q,\theta,\rho} \Big( 1+\|u_0\|_{r_1,Y}^{\rho(1+\frac{\alpha}{\theta})}\Big) k^\alpha.
	\end{align*}
    We pass to the second term in \eqref{eq:splitUhatUbarProof}, the difference of the piecewise constant extension $U^k$ and the continuous extension $\hat{U}$ of the scheme. 
    Recalling \eqref{eq:UhatfromUk}, we have $\hat{U}(t)-U^k(t)=0$ if $t=t_j$ for some $0\le j\le N$ and else
	\begin{align*}
		\hat{U}(t)-U^k(t) &= [S(t-\floort)-I]U^k(t) + \int_\floort^t S(t-\floort)  F_\theta^k(\floort)\ds +  \int_{\floort}^t S(t-\floort) G_\theta^k(\floort) \dWHs
	\end{align*}
    because $\floors=\floort$ on $(\floort,t)$. We estimate the supremum in time of the $r_2$-th moments of these terms individually. For the first term, decay is obtained via the semigroup difference estimate from Lemma \ref{lem:sgInterpolation}. Pointwise  stability from Proposition \ref{prop:stabY} with $q \curve r_2$, which is applicable due to \eqref{eq:r1r2coefficientsProof}, then implies
    \begin{align*}
        \sup_{t\in[0,T]} \big\|\big[S(t-\floort)-I\big]U^k(t)\big\|_{r_2,X} 
        &\le C \sup_{t\in[0,T]}\Big((t-\floort)^\alpha \|U^k(t)\|_{r_2,Y} \Big)
        \le C_{T,r_2} \big(1+\|u_0\|_{r_2,Y}\big) k^\alpha.
    \end{align*}
    Via contractivity of $(S(t))_{t\ge 0}$, the embedding $Y\hra X$, the stopped nonlinearities, $\theta\le \frac{1}{4}$, and $k\le T$, the $F$-term can be estimated by
    \begin{align}
    \label{eq:FtermLeverageStoppedNLproof}
        \sup_{t\in[0,T]} \Big\|\int_\floort^t S(t-\floort) F_\theta^k(\floort)\ds\Big\|_{r_2,X} 
        &\le \sup_{t\in[0,T]} \Big((t-\floort) \|F_\theta^k(\floort)\|_{r_2,X}\Big)\nonumber\\
        &\le C k\max_{0\le j \le \Nk -1} \|\1_{\theta,j} F(U^j)\|_{r_2,Y}
        \le C k^{1-\theta} \le C T^{1-\theta-\alpha}k^\alpha.
    \end{align}
     In contrast, for the $G$-term, we do not make use of the stopped nonlinearities but leverage polynomial growth and \eqref{eq:r1r2coefficientsProof} to obtain pointwise stability from Proposition \ref{prop:stabY} with $q \curve \rho r_2$. This results in
    \begin{align}
    \label{eq:GtermLeverageStoppedNLproof}
        \sup_{t\in[0,T]} &\Big\|\int_\floort^t S(t-\floort) G_\theta^k(\floort)\dWHs\Big\|_{ r_2,X}\le C_{r_2}\sup_{t\in[0,T]} \Big(\int_\floort^t \bignn S(t-\floort) G_\theta^k(\floort)\bignn_{ r_2,X}^2\ds\Big)^{1/2}.\nonumber\\
        &\le C_{r_2} k^{1/2}\max_{0\le j\le \Nk -1} \nn G(U^j)\nn_{ r_2,Y} 
        \le C_{r_2,\rho} k^{1/2} \Big(1+\max_{0\le j\le \Nk -1}\|U^j\|_{\rho r_2,Y}^\rho\Big)\nonumber\\
        &\le C_{T,r_2,\rho} \big(1+\|u_0\|_{\rho r_2,Y}^\rho\big) k^\alpha,
    \end{align}
    where we have used that $k^{1/2}\le T^{1/2-\alpha}k^\alpha$.
    Altogether, for the second term in \eqref{eq:splitUhatUbarProof}, 
    \begin{align*}
        \sup_{t\in[0,T]} \|\hat{U}(t)-U^k(t)\|_{r_2,X}
        &\le C_{\alpha,T,r_2,\theta,\rho} \big(1+\|u_0\|_{\rho r_2,Y}^\rho\big) k^\alpha.
    \end{align*}
    Tracing back the constant dependencies finishes the proof.
\end{proof}

\begin{remark}
    \label{rem:useStoppedNLforF}
    We would like to point out that instead of the stopped nonlinearities, we could use polynomial growth and stability of $u_0$ to achieve decay at rate $1$ in \eqref{eq:FtermLeverageStoppedNLproof}. Since the argument above yields a bound independent of the initial values with decay of order $1-\theta$, which is larger than the optimal rate $\alpha \le \frac{1}{2}$ for the full error, this is not necessary. In contrast, for the $G$-term, leveraging the stopped nonlinearities in \eqref{eq:GtermLeverageStoppedNLproof} would result in suboptimal convergence at rate $\frac{1}{2}-\theta$ instead of $\alpha$.
\end{remark}

\subsubsection{Error between the scheme and the solution}

\begin{theorem}[Convergence rate]
\label{thm:convergenceMain}
    Let $p\in [2,\infty)$, $\lambda \in (0,1)$, $q=\frac{p}{\lambda}$, and $\theta \in (0,\frac{1}{4}]$. Suppose that Assumption \ref{ass:locLipPolGro}($\alpha,\nu,\rho$) holds for some $\alpha\in (0,\frac{1}{2}]$ and $\nu,\rho\ge 1$. Let $r=\rho q(1+\nu+\frac{\alpha}{\theta})$ and $u_0\in L_{\F_0}^{r}(\Omega;Y)$. Further suppose that Assumptions \ref{ass:coYq}($r$) and \ref{ass:monqEta}($q,\eta$) hold for some $\eta>0$.
    Then the nonlinearities-stopped exponential Euler scheme $(U^j)_{0\le j\le N}$ converges at rate $\alpha\in (0,\frac{1}{2}]$ to the solution $(U(t))_{t\in [0,T]}$ of \eqref{eq:SEEsemilinear} in a pathwise uniform sense. More precisely, there is a constant $C\ge 0$ depending on $(\alpha,T,p,\lambda,\theta,\nu,\rho,\eta)$ but independent of $k$, $N$, and $u_0$ such that the extension $(\hat{U}(t))_{t\in [0,T]}$ of the scheme from Definition \ref{def:UhatExtension} satisfies
     \begin{align*}
        \bigg\|\sup_{t\in[0,T]} \|\hat{U}(t)-U(t)\|_X\bigg\|_{L^p(\Omega)}
        &\le C\big(1+\|u_0\|_{L^r(\Omega;Y)}^{r/q} \big)k^\alpha.
    \end{align*}
    In particular, 
    \begin{align*}
        \Big\|\max_{0\le j \le N} \|U^j-U(t_j)\|_X\Big\|_{L^p(\Omega)}
        &\le C\big(1+\|u_0\|_{L^r(\Omega;Y)}^{r/q} \big)k^\alpha.
    \end{align*}
\end{theorem}
\begin{proof}
    Via the triangle inequality, we split the error according to
    \begin{align}
    \label{eq:firstTriangle}
	   \bigg\|\sup_{t\in[0,T]} \|\hat{U}(t)-U(t)\|_X\bigg\|_p&\le \bigg\|\sup_{t\in[0,T]} \|\hat{U}(t)-\bar{U}(t)\|_X\bigg\|_p+\bigg\|\sup_{t\in[0,T]} \|\bar{U}(t)-U(t)\|_X\bigg\|_p.
    \end{align}
    By Proposition \ref{prop:UhatUbarError} with $q \curve p$, the first term is bounded by
      \begin{align*}
		\bigg\|&\sup_{t\in [0,T]} \|\hat{U}(t)-\bar{U}(t)\|_X\bigg\|_p 
        \le  C_{\alpha,T,p,\theta,\rho}  \Big( 1+\|u_0\|_{(1+\frac{\alpha}{\theta})\rho p,Y}^{\rho(1+\frac{\alpha}{\theta})}\Big) k^{\alpha},
	\end{align*}
    where we have used that $r$-coercivity implies $(1+\frac{\alpha}{\theta})\rho p$-coercivity due to $(1+\frac{\alpha}{\theta})\rho p \le r$ and $r$ equals $r_1$ from Proposition \ref{prop:UhatUbarError}. In order to obtain an estimate for the second error term, we use the inequality \eqref{eq:diffUkUbarProp} from Proposition \ref{prop:diffUkUbar} to estimate the last factor in the inequality from Lemma \ref{lem:reduceUbarUtoUkUbar}, resulting in
    \begin{align*}
	   \bigg\|\sup_{t\in [0,T]} \|\bar{U}(t)-U(t)\|_X\bigg\|_p &\le C_{\alpha,T,p,\lambda,\theta,\nu,\rho,\eta} \big(1+\|u_0\|_{r,Y}^{\rho\nu} \big)\Big(1+\|u_0\|_{r,Y}^{\rho(1+\frac{\alpha}{\theta})} \Big)k^\alpha\\
       &\le C_{\alpha,T,p,\lambda,\theta,\nu,\rho,\eta} \Big(1+\|u_0\|_{r,Y}^{\rho(1+\nu+\frac{\alpha}{\theta})} \Big)k^\alpha.
    \end{align*}
    Here, we have used in particular that Assumption \ref{ass:coYq}($r$) holds and that $u_0\in L^r(\Omega;Y)$. Inserting these two error estimates into \eqref{eq:firstTriangle} gives the final error estimate for $\hat{U}(t)-U(t)$ after absorbing lower powers and, via Hölder's inequality, lower moments of the initial values. Since $\hat{U}(t_j)=U^j$ for $0\le j \le N$, the error estimate for $U^j-U(t_j)$ is obtained by taking the supremum only over the grid points.
\end{proof}

\begin{remark}[Parameters]
\begin{enumerate}[label=(\alph*)]
    \item The parameter $r=\rho q(1+\nu+\frac{\alpha}{\theta})$ for the moments of the initial values and the coercivity assumption is contained in $(4p,\infty)$ in the optimal rate case $\alpha=\frac{1}{2}$. Choosing $\theta\in (0,\frac{1}{4}]$ larger results in lower $r$, allowing for larger noise coefficients and initial values with less moments. The higher the Lipschitz and growth parameters $\nu$ and $\rho$ are, the more restrictive Assumption \ref{ass:coYq}($r$). However, it is only affecting the coefficient in front of the noise, not the admissible order of growth, as illustrated for the stochastic transport equation with quadratic noise in Subsection \ref{subsec:transportExample}.
    \item The Hölder parameters $\beta_1$ and $\beta_2$ are chosen optimally in the sense that other parameters would result in higher $r$ and thus a stronger coercivity condition and higher integrability requirements on the initial values in the main theorem.
\end{enumerate}
\end{remark}

\begin{remark}[Extension to $2$-smooth Banach spaces]
    It is natural to ask whether the well-posedness and convergence results can be generalised from Hilbert to $2$-smooth Banach spaces, a class of Banach spaces including $L^p$ for $p\in [2,\infty)$, as done in 
    \cite[Sec.~3]{KliobaVeraar24Irregular} for globally Lipschitz nonlinearities. Regarding the pointwise stability as well as the convergence results, such an extension would require replacing Hilbert--Schmidt operators by $\gamma$-radonifying operators, which seems feasible, and suitably formulating coercivity and monotonicity conditions with dual pairings instead of scalar products. The main challenge already arises earlier in the local well-posedness proof. The reduction to the globally Lipschitz case relied on the uniformly $X$-Lipschitz projections $P_n:X\to B_n$ onto balls in $Y$. This construction is not available in general Banach spaces. Furthermore, the radial projections onto balls in $Y$ are $Y$-Lipschitz but need not be $X$-Lipschitz in general. Hence, a different approach to local and global well-posedness would be required for an extension beyond Hilbert spaces.
\end{remark}

\subsection{Error estimate for an interpolation to the full time interval}
\label{subsec:convInterpolationFullTimeInterval}

While Theorem \ref{thm:convergenceMain} gives a pathwise uniform convergence rate for $\hat{U}-U$ on the full time interval $[0,T]$, computing the continuous extension $\hat{U}$ requires knowledge of the underlying Brownian motion and the semigroup at all times $t\in [0,T]$ rather than merely at the grid points. In practice, it is easier to compute actual interpolations of the scheme from the grid to $[0,T]$. We show that, up to a square-root logarithmic correction factor, the piecewise constant right-continuous interpolation $(U^k(t))_{t\in[0,T]}$ of the scheme $(U^j)_{0\le j \le N}$ achieves the same pathwise uniform convergence rate on $[0,T]$ as the scheme does at the grid points. The proof relies on a logarithmic square function estimate for stochastic integrals.

\begin{theorem}[Uniform error on the full interval]
\label{thm:convergenceExtension}
    Let $p\in [2,\infty)$, $\lambda \in (0,1)$, $q=\frac{p}{\lambda}$, and $\theta \in (0,\frac{1}{4}]$. Suppose that Assumption \ref{ass:locLipPolGro}($\alpha,\nu,\rho$) holds for some $\alpha\in (0,\frac{1}{2}]$ and $\nu,\rho\ge 1$. Let $r=\rho q(1+\nu+\frac{\alpha}{\theta})$ and $u_0\in L_{\F_0}^{r}(\Omega;Y)$. Further suppose that Assumptions \ref{ass:coYq}($r$) and \ref{ass:monqEta}($q,\eta$) hold for some $\eta>0$.
    Then the right-continuous piecewise constant extension $(U^k(t))_{t\in [0,T]}$ of the nonlinearities-stopped exponential Euler scheme from Definition \ref{def:UkExtension} converges at rate $\alpha\in (0,\frac{1}{2}]$, with a logarithmic correction if $\alpha=\frac12$, to the solution $(U(t))_{t\in [0,T]}$ of \eqref{eq:SEEsemilinear} in a pathwise uniform sense. More precisely, there is a constant $C\ge 0$ depending on $(\alpha,T,p,\lambda,\theta,\nu,\rho,\eta)$ as well as $\|u_0\|_{L^r(\Omega;Y)}$ but independent of $k$ and $N$ such that
     \begin{align*}
        \bigg\|\sup_{t\in[0,T]} \|U^k(t)-U(t)\|_X\bigg\|_{L^p(\Omega)}
        &\le \begin{cases}
            C\sqrt{\max\{\log(T/k),p\}}k^{1/2},& \alpha=\frac12,\\
            Ck^\alpha,&\alpha\in (0,\frac12).
        \end{cases}
    \end{align*}
\end{theorem}

\begin{proof}
    First, we prove the claim for $\alpha=\frac12$. Define the increasing function $\Psi:[0,T]\to[0,\infty)$ by $\Psi(0)\ce 0$ and $\Psi(h)\ce h^{1/2}\sqrt{1+\log(T/h)}$ for $h\in (0,T]$. 
We say that $u\in C^\Psi([0,T];X)$ if $u:[0,T]\to X$ is continuous and
\begin{equation*}
    [u]_{C^\Psi([0,T];X)} = \sup_{0\leq s<t\leq T} \frac{\|u(t)-u(s)\|}{\Psi(t-s)}<\infty.
\end{equation*}
Moreover, we set $\|u\|_{C^\Psi([0,T];X)} \ce \|u\|_{\infty}+ [u]_{C^\Psi([0,T];X)}$. 
    By  \cite[Lemma~6.9]{KliobaVeraar24Rate}, using this generalised Hölder space,
    \begin{equation*}
        \sup_{t\in [0,T]} \|U^k(t)-U(t)\|_X \le \Psi(k) \|U\|_{C^\Psi([0,T];X)} + \max_{0 \le j \le N} \|U^j-U(t_j)\|_X\quad \text{a.s.}
    \end{equation*}
    After taking $p$-th moments, the error estimate in the grid points of Theorem \ref{thm:convergenceMain} bounds the second term by $Ck^{1/2}$ and thus also by $C \Psi(k)$ as required.    
    We abbreviate the norm in $L^p(\Omega;C^\Psi([0,T];X))$ by $\|\cdot\|_{L^pC^\Psi X}$. To show the necessary path regularity of the mild solution, we analyse the initial value term and the deterministic and stochastic convolutions separately. Contractivity of the semigroup and Lemma \ref{lem:sgInterpolation} give
    \begin{align*}
        \|S(\cdot)u_0\|_{L^pC^\Psi X}&\le\bigg\|\sup_{0\le s<t\le T}\frac{\|[S(t)-S(s)]u_0\|_X}{\Psi(t-s)}\bigg\|_p + \bigg\|\sup_{0\le t\le T}\|S(t)u_0\|_X\bigg\|_p\\
        &\le C \bigg(\sup_{h\in (0,T]}\frac{1}{\sqrt{1+\log(T/h)}}\bigg) \|u_0\|_{p,Y} + \|u_0\|_{p,X} \le C_T\|u_0\|_{p,Y}.
    \end{align*}
    The difference of the deterministic convolutions at times $s<t$ can be split according to
    \begin{align*}
        \int_0^s [S(t-r)-S(s-r)]F(U(r))\,\rmd r + \int_s^t S(t-r)F(U(r))\,\rmd r.
    \end{align*}
    Using Lemma \ref{lem:sgInterpolation} and contractivity of the semigroup as before, this yields
    \begin{align*}
        \Big\|&\int_0^\cdot S(\cdot-r)F(U(r))\,\rmd r\Big\|_{L^pC^\Psi X}\le C_T\bigg\|\sup_{0\le s\le T}\int_0^s\|F(U(r))\|_Y\,\rmd r\bigg\|_p\\
        &\phantom{\le }+ \bigg\|\sup_{0\le s<t\le T}\frac{t-s}{\Psi(t-s)}\sup_{r\in [0,T]}\|F(U(r))\|_X\bigg\|_p + T \bigg\|\sup_{r\in [0,T]}\|F(U(r))\|_X\bigg\|_p\\
        &\le C_{T,p,q,\rho} \big(1+\|u_0\|_{\rho p,Y}^\rho\big),
    \end{align*}
    where we have used $Y\hra X$, polynomial growth of $F$, and the a priori estimate from Theorem \ref{thm:aprioriY} in the last step. For the path regularity of the stochastic convolution, we employ \cite[Lemma~6.10]{KliobaVeraar24Rate} with $p_0=q$, resulting in
    \begin{align*}
        \|J_{G\circ U}\|_{L^pC^\Psi X} &\le C_{T,p,q}\big(\nn G(U) \nn_{p,2,Y} + \nn G(U)\nn_{q,\infty,X} \big)\le C_{T,p,q,\rho}\big(1+\| U \|_{\rho p,\infty,Y}^\rho + \| U\|_{\rho q,\infty,Y}^\rho \big)\\
        &\le C_{T,p,q,\rho}\big(1+\|u_0\|_{\rho q,Y}^\rho\big).
    \end{align*}
    In the last two inequalities, we have used polynomial growth of $G$ and the a priori estimate from Theorem \ref{thm:aprioriY}, which is applicable because Assumption \ref{ass:coYq}($r$) holds and $r>\rho q$. In summary,
    \begin{equation*}
        \Psi(k)\|U\|_{L^pC^\Psi X} \le C_{T,p,q,\rho}\big(1+\|u_0\|_{\rho q,Y}^\rho\big)\sqrt{1+\log(T/k)} k^{1/2}.
    \end{equation*}
    
    Now consider the case $\alpha<\frac{1}{2}$. Define $\Phi_\alpha(h)\ce h^\alpha$ for $h\in [0,T]$ such that $C^{\Phi_\alpha}([0,T];X)=C^\alpha([0,T];X)$ is the usual $\alpha$-Hölder space. The splitting of the error, the estimate on the grid as well as the path regularity estimates for all terms except the stochastic convolution carry through analogously, since the logarithmic factor was not used in the estimates. Split the difference of the stochastic convolutions at times $0\le s<t\le T$ into 
    \begin{equation*}
        J_{G\circ U}(t)-J_{G\circ U}(s) = [S(t-s)-I]J_{G\circ U}(s) + \int_s^t S(t-r)G(U(r))\dWH(r) \ec T_1(t,s) + T_2(t,s).
    \end{equation*}
    Since $\Phi_\alpha \ge C_{\alpha,T}\Psi$ on $(0,T]$, we can thus estimate
    \begin{align*}
        \|&J_{G\circ U}\|_{L^pC^{\Phi_\alpha}X} \le \bigg\|\sup_{0\le s<t\le T} \frac{\|T_1(t,s)\|_X}{(t-s)^\alpha}\bigg\|_p + \bigg\|\sup_{0\le s<t\le T} \frac{\|T_2(t,s)\|_X}{C_{\alpha,T}\Psi(t-s)}\bigg\|_p+\|J_{G\circ U}\|_{p,\infty,X}.
    \end{align*}
    The last two terms are bounded by $C_{T,p,q,\rho}$ as above for $\alpha=\frac{1}{2}$ also for $\alpha <\frac{1}{2}$, as an inspection of the proof of \cite[Lemma~6.10]{KliobaVeraar24Rate} shows, noting that $r>\rho q$ holds for all $\alpha\in (0,\frac12]$. Via Lemma \ref{lem:sgInterpolation} and the maximal inequality from Theorem \ref{thm:maximal-inequality}, we bound the first term by
    \begin{align*}
        C_\alpha \bigg\|\sup_{0\le s<t\le T} \frac{(t-s)^\alpha\|J_{G\circ U}(s)\|_Y}{(t-s)^\alpha}\bigg\|_p \le C_{\alpha,p}\nn G(U)\nn_{p,2,Y} \le C_{\alpha,T,p,q,\rho}\big(1+\|u_0\|_{\rho q,Y}^\rho\big).
    \end{align*}
    The error estimate then follows by combining the estimate for $\|U\|_{L^pC^{\Phi_\alpha}X}$ with Theorem \ref{thm:convergenceMain}.
\end{proof}

\section{Tamed exponential Euler schemes}
\label{sec:generalTamings}

For the nonlinearities-stopped exponential Euler scheme, we have made use of the stopped nonlinearities to obtain stability, which would not have been possible for the standard exponential Euler scheme. In the convergence proof, this stopping led to an additional error term in Proposition \ref{prop:UhatUbarError}, which was dealt with via Markov's inequality. Completely stopping the update of the nonlinearities in the scheme when they exceed a prescribed threshold, however, is just one option to address the stability problem. Various different tamings have been studied for parabolic
SPDEs with nonglobally Lipschitz nonlinearities, for example in
\cite{GyongySabanisSiska16,Wang20taming, Brehier22tamedExpEuler}. In this section, we revisit the stability and convergence arguments from Sections \ref{sec:stability} and \ref{sec:convergenceRate} in order to generalise these results to a broader class of tamings. The resulting convergence statement can be found in Theorem \ref{thm:convergenceMainTamed}, our second main result. 

\begin{definition}[Scheme]
\label{def:TamedScheme}
    The \emph{tamed exponential Euler scheme} $(U^j)_{j=0,\ldots,\Nk}$ with $U^j:\Omega\to Y$ is defined by
    \begin{equation}
    \label{eq:defTamedSchemeUj}
        U^{j+1} \ce S(k) U^j + kS(k) F_k(U^j) + S(k)G_k(U^j)\Delta W_{j+1},~0\le j \le \Nk -1,\quad U^0 \ce u_0,
    \end{equation}
    where $\Delta W_{j+1} \ce W_H(t_{j+1})-W_H(t_j)$, the last term is understood in the sense of \eqref{eq:convradonW} and $F_k:Y\to Y$ and $G_k:Y\to\LHY$ are Borel-measurable.
\end{definition}

By a slight abuse of notation, in this section we also denote the piecewise constant and the continuous extensions of the tamed exponential Euler scheme as well as its semilinear integrated counterparts by $(U^k(t))_{t\in [0,T]}$, $(\hat{U}(t))_{t\in [0,T]}$, and $(\bar{U}(t))_{t\in [0,T]}$ as in Definitions \ref{def:UkExtension}, \ref{def:UhatExtension}, and \ref{def:UbarSemilinearIntegratedCounterparts}, respectively. Borel measurability of $F_k,G_k$ ensures adaptedness of the scheme. Note the difference between the tamed nonlinearities $F_k,G_k$ introduced here and the functions $F^k,G^k$ evaluating the nonlinearities at the last grid point via $F^k(s)\ce F(U^k(s))$ from Definition \ref{def:UkExtension}. 

A common situation is given by scalar tamings
\begin{equation}
\label{eq:scalarTaming}
         F_k:Y\to Y,~F_k(x)\ce a_k(x)F(x),\qquad G_k:Y \to \LHY,~~ G_k(x)\ce b_k(x)G(x)
\end{equation}
with Borel-measurable $a_k,b_k: Y \to [0,1]$. Under Assumption \ref{ass:locLipPolGro}, $F_k$ and $G_k$ are then Borel measurable because $F,G,a_k,b_k$ are. The nonlinearities-stopped exponential Euler scheme from Definition \ref{def:Scheme} is recovered via
\begin{align}
\label{eq:akbkNLstoppedExpEuler}
    a_k(x)=b_k(x)\ce
    \1_{\{\|F(x)\|_Y+\|G(x)\|_{\LHY}\le k^{-\theta}\}}.
\end{align}

In order to ensure stability and convergence of the tamed exponential Euler method, we impose the following three parameter-dependent assumptions. In contrast to previous sections, we now allow for different growth parameters of the tamed nonlinearities $F_k$ and $G_k$ and include $0$ in the admissible parameter range. This provides more flexibility e.g.\ in the case of noise of linear growth, cf.\ Corollary \ref{cor:optimalSeparateTaming}. Stability is a straightforward consequence of the first two assumptions and the stability proof from Proposition \ref{prop:stabY}.

\begin{assumption}[$q$) (Tamed coercivity]
\label{ass:tamedCoercivity}
    For given $q\in [2,\infty)$, there is a constant $C_q\ge 0$ such that for all $k\in (0,T]$ and $x\in Y$,
    \begin{equation*}
        \Re \la x,F_k(x) \ra_Y+\frac{q-1}{2}\|G_k(x)\|_\LHY^2 \le C_q(1+\|x\|_Y^2).
    \end{equation*}
\end{assumption}

\begin{assumption}[$\theta_F,\theta_G$) (Generalized growth bound for the tamed nonlinearities]
\label{ass:growthBoundTaming}
    For given $\theta_F,\theta_G\in [0,\frac12]$, there is a constant $C_{T,\theta}\ge 0$ such that for all $k\in (0,T]$ and $x\in Y$,
    \begin{equation}
     \label{eq:tamedGrowthBounds}
        \|F_k(x)\|_Y \le C_{T,\theta}k^{-\theta_F}(1+\|x\|_Y),\quad \|G_k(x)\|_\LHY\le C_{T,\theta}k^{-\theta_G}(1+\|x\|_Y)
    \end{equation}
    and $\theta_F+\theta_G\le \frac{1}{2}$ as well as
    \begin{equation}
    \label{eq:tamedNLsmaller}
        \|F_k(x)\|_Y \le \|F(x)\|_Y\quad\text{and}\quad \|G_k(x)\|_\LHY \le \|G(x)\|_\LHY.
    \end{equation}
\end{assumption}

We denote dependence of constants on both $\theta_F,\theta_G$ by the index $\theta$.

\begin{assumption}[$\alpha,\beta,q$) (Consistency order of the taming]
\label{ass:consistencyTaming}
    For given $\alpha\in(0,\frac12]$, $\beta\ge 1$, and $q\in [2,\infty)$, there is a constant $C_{\alpha,\beta,q}\ge 0$ such that for all $k\in (0,T]$ and $\Phi\in L^{\beta q}(\Omega;Y)$,
    \begin{equation}
    \label{eq:tamingConsistency}
        \|F_k(\Phi)-F(\Phi)\|_{L^q(\Omega;Y)}+\|G_k(\Phi)-G(\Phi)\|_{L^q(\Omega;\LHY)} \le C_{\alpha,\beta,q}k^{\alpha}\big(1+\|\Phi\|_{L^{\beta q}(\Omega;Y)}^\beta\big).
    \end{equation}
\end{assumption}

\begin{proposition}[Stability of tamed exponential Euler]
\label{prop:stabYtamed}
    Suppose that $-A$ generates a $C_0$-contraction semigroup $(S(t))_{t\ge 0}$ on $Y$ and that tamed $q$-coercivity as in Assumption \ref{ass:tamedCoercivity}($q$) holds for some $q\in [2,\infty)$. Let $u_0\in L_{\F_0}^q(\Omega;Y)$, let Assumption \ref{ass:growthBoundTaming}($\theta_F,\theta_G$) hold for some $\theta_F,\theta_G\in [0,\frac{1}{2}]$. Then the tamed exponential Euler scheme is pointwise stable in $Y$. That is, for all $T>0$ there exists a constant $C_{T,q,\theta}\ge 0$ independent of $k$ such that 
    \begin{equation*}
        \max_{0 \le j \le N} \|U^j\|_{L^q(\Omega;Y)}= \sup_{t\in [0,T]} \|U^k(t)\|_{L^q(\Omega;Y)}\le \sup_{t\in [0,T]} \|\hat{U}(t)\|_{L^q(\Omega;Y)}\le  C_{T,q,\theta} \big(1+\|u_0\|_{L^q(\Omega;Y)} 
        \big)<\infty.
    \end{equation*}
\end{proposition}
\begin{proof}
    We only indicate changes needed in the proofs of Proposition \ref{prop:stabY} and Lemma \ref{lem:stabOneTimeStep}. In the definition \eqref{eq:defVlxr} of $V_{t_\ell}^x$, replace $F_\theta$ and $G_\theta$ by $F_k$ and $G_k$. 
    As discussed in Remark \ref{rem:coercivityFGtheta}, the coercivity argument already used that $F_\theta$ and $G_\theta$ satisfy Assumption \ref{ass:tamedCoercivity}($q$) and thus remains valid for the tamed scheme with $F_k$ and $G_k$. The stopped nonlinearities were also leveraged to estimate the last term in \eqref{eq:stabProofExp}. Repeating the arguments in \eqref{eq:stabProofTaming}, we can then use Assumption \ref{ass:growthBoundTaming}, $\theta_F\le \frac12$, and $\theta_F+\theta_G\le \frac{1}{2}$ to deduce
    \begin{align*}
        \int_{t_\ell}^{t_\ell+r}&\big((s-t_\ell)\|F_k(x)\|_{q,Y}^2+C_q\sqrt{s-t_\ell}\|F_k(x)\|_{q,Y}\nn G_k(x)\nn_{q,Y}\big)^{q/2}\ds\\
        &\le C_{T,q,\theta}r\big(k^{1-2\theta_F}+k^{\frac12-\theta_F-\theta_G}\big)^{q/2}\big(1+\|x\|_{q,Y}\big)^q \le C_{T,q,\theta}r\big(1+\E\|x\|_Y^q\big).
    \end{align*}
    The remaining proof of Lemma \ref{lem:stabOneTimeStep} carries through analogously and the one-step stability estimate is extended to $[0,T]$ as in Proposition \ref{prop:stabY}.
\end{proof}

Via the arguments from the proof of Proposition \ref{prop:stabUbar} using Proposition \ref{prop:stabYtamed} rather than \ref{prop:stabY}, stability of $(\bar{U}(t))_{t\in [0,T]}$ follows.

\begin{proposition}
\label{prop:stabUbarTamed}
    Suppose that $-A$ generates a $C_0$-contraction semigroup on $Y$ and Assumption \ref{ass:polGrRho}($\rho$) holds for some $\rho\ge1$. Further, suppose that Assumptions \ref{ass:tamedCoercivity}($\rho q$) and \ref{ass:growthBoundTaming}($\theta_F,\theta_G$) hold for some $q\in[2,\infty)$ and $\theta_F,\theta_G\in [0,\frac{1}{2}]$. Let $u_0\in L_{\F_0}^{\rho q}(\Omega;Y)$. Then, for all $T>0$, there is a constant $C_{T,q,\theta,\rho} \ge 0$, independent of $k$, such that 
    \begin{align*}
        \sup_{t\in[0,T]}\|\bar U(t)\|_{L^q(\Omega;Y)}
        \le C_{T,q,\theta,\rho}
        \big(1+\|u_0\|_{L^{\rho q}(\Omega;Y)}^\rho\big).
    \end{align*}
\end{proposition}

Regarding the convergence proof, mainly Proposition \ref{prop:UhatUbarError} needs to be adapted to the new tamings. Then convergence follows via the same strategy as for the stopped nonlinearities, which yields the second main result in Theorem \ref{thm:convergenceMainTamed}: pathwise uniform convergence rates for a class of tamed exponential Euler schemes.

\begin{proposition}
	\label{prop:UhatUbarErrorTamed}
	Let $q \in [2,\infty)$ and let Assumption \ref{ass:locLipPolGro}($\alpha,\nu,\rho$) hold for some $\alpha \in (0,\frac{1}{2}]$ and $\nu,\rho\ge 1$. 
    Suppose that Assumptions \ref{ass:growthBoundTaming}($\theta_F,\theta_G$) and  \ref{ass:consistencyTaming}($\alpha,\beta,q$) hold for some $\theta_F,\theta_G\in [0,\frac12]$ and $\beta\ge \rho$. Further, let $u_0\in L_{\F_0}^{\beta q}(\Omega;Y)$ and suppose that Assumption \ref{ass:tamedCoercivity}($\beta q$) holds. Then there is a constant $C_{\alpha,\beta,T,q,\theta,\rho}\ge 0$ not depending on $k$, $N$, or $u_0$ such that
    \begin{align*}
		\Big\|\max_{0 \le j \le \Nk } \|U^j-\bar{U}(t_j)\|_X\Big\|_{L^q(\Omega)}
        \le \bigg\|&\sup_{t\in [0,T]} \|\hat{U}(t)-\bar{U}(t)\|_X\bigg\|_{L^q(\Omega)} 
        \le  C_{\alpha,\beta,T,q,\theta,\rho}  \big( 1+\|u_0\|_{\beta q,Y}^{\beta}\big) k^\alpha.
	\end{align*}
\end{proposition}
\begin{proof}
	We comment on the differences to the proof of Proposition \ref{prop:UhatUbarError}. Using the same splitting
	\begin{align*}
		\hat{U}(t)-\bar{U}(t) &=  E_{F,1}(t)+E_{F,2}(t)+E_{G,1}(t)+E_{G,2}(t)
	\end{align*}
	and simply replacing $\1_{\theta,\ell}\le 1$ by \eqref{eq:tamedNLsmaller} from Assumption \ref{ass:growthBoundTaming} yields  
	\begin{align}
    \label{eq:estimateProofMinusTheta2}
		\bigg\|\sup_{t\in [0,T]} \|E_{G,1}(t)\|_X\bigg\|_q
        &\le C_{\alpha,q,\rho} \sqrt{T} k^\alpha \Big( 1+\max_{0 \le \ell \le \Nk -1}\|U^\ell\|_{\rho q,Y}^\rho\Big). 
	\end{align}
    Since $\beta q \ge \rho q$, Assumption \ref{ass:tamedCoercivity}($\rho q$) holds, so that Proposition \ref{prop:stabYtamed} with $q \curve \rho q$ implies
    \begin{align*}
		\bigg\|\sup_{t\in [0,T]} \|E_{G,1}(t)\|_X\bigg\|_q
        &\le C_{\alpha,T,q,\theta,\rho} \big(1+\|u_0\|_{\rho q,Y}^\rho\big)  k^\alpha
	\end{align*} 
    and an analogous estimate for $E_{F,1}(t)$.
    
	To estimate $E_{G,2}(t)$, the Markov inequality argument from Lemma \ref{lem:Markov2} is no longer required. Instead, consistency of the taming as in Assumption \ref{ass:consistencyTaming}($\alpha,\beta,q$) provides the bound
	\begin{align*}
		\bigg\|\sup_{t\in [0,T]} \|E_{G,2}(t)\|_X\bigg\|_q
		&\le C_q \Big(\int_0^T \bignn G_k(U^k(s)) - G(U^k(s))\bignn_{q,X}^2\ds\Big)^{1/2}\\
        &\le C_{\alpha,\beta,q}k^\alpha \Big(\int_0^T \big(1+\|U^k(s)\|_{\beta q,Y}^{2\beta}\big)\ds\Big)^{1/2}\le C_{\alpha,\beta,T,q,\theta} k^\alpha \big(1+\|u_0\|_{\beta q,Y}^{\beta}\big),
	\end{align*}
    where we have also used $Y\hra X$ and Proposition \ref{prop:stabYtamed} with $q \curve \beta q$, which is applicable by Assumption \ref{ass:tamedCoercivity}($\beta q$). Proceeding likewise for $E_{F,2}(t)$ finishes the proof.
\end{proof}

\begin{theorem}[Convergence rate for tamed exponential Euler schemes]
\label{thm:convergenceMainTamed}
    Let $p\in [2,\infty)$, $\lambda \in (0,1)$ and $q=\frac{p}{\lambda}$. Suppose that Assumption \ref{ass:locLipPolGro}($\alpha,\nu,\rho$) holds for some $\alpha\in (0,\frac{1}{2}]$ and $\nu,\rho\ge 1$. Moreover, let the tamed nonlinearities satisfy Assumptions \ref{ass:growthBoundTaming}($\theta_F,\theta_G$) and \ref{ass:consistencyTaming}($\alpha,\beta,q(1+\frac{\rho \nu}{\beta})$) for some $\theta_F,\theta_G\in [0,\frac{1}{2}]$ and $\beta \ge \rho$. 
    Let $r=q(\beta+\rho\nu)$ and $u_0\in L_{\F_0}^{r}(\Omega;Y)$. Further, suppose that Assumptions \ref{ass:coYq}($r$), \ref{ass:tamedCoercivity}($r$), and \ref{ass:monqEta}($q,\eta$) hold for some $\eta>0$. 
    
    Then the tamed exponential Euler scheme $(U^j)_{0\le j\le N}$ from Definition \ref{def:TamedScheme} converges at rate $\alpha$ to the solution $(U(t))_{t\in [0,T]}$ of \eqref{eq:SEEsemilinear} in a pathwise uniform sense. More precisely, there is a constant $C\ge 0$ depending on $(\alpha,\beta,T,p,\lambda,\theta_F,\theta_G,\nu,\rho,\eta)$ but independent of $k$, $N$, and $u_0$ such that 
     \begin{align*}
        \Big\|\max_{0\le j \le N} \|U^j-U(t_j)\|_X\Big\|_{L^p(\Omega)} \le \bigg\|\sup_{t\in[0,T]} \|\hat{U}(t)-U(t)\|_X\bigg\|_{L^p(\Omega)}
        &\le C\big(1+\|u_0\|_{L^r(\Omega;Y)}^{r/q} \big)k^\alpha.
    \end{align*}
\end{theorem}
\begin{proof}
    Following the same proof strategy as in Theorem \ref{thm:convergenceMain}, we comment on the changes required. We replace the definition of $r_1,r_2,\beta_1,\beta_2$ by
    \begin{equation*}
        r_1 \ce r=q(\beta+\rho\nu),\quad r_2\ce q\Big(1+\frac{\rho\nu}{\beta}\Big),\quad\beta_1\ce 1+\frac{\beta}{\rho\nu},\quad \beta_2 \ce 1+\frac{\rho\nu}{\beta}.
    \end{equation*}
    Note that $\beta_1$ and $\beta_2$ are Hölder conjugates. Lemma \ref{lem:reduceUbarUtoUkUbar} with the new Hölder parameters $\beta_1$ and $\beta_2$ in \eqref{eq:differentBetaLater} then results in estimating the moments of order $\beta_1\nu q= r_1/\rho$ and $\beta_2q=r_2$. Applying Propositions \ref{prop:stabYtamed} and \ref{prop:stabUbarTamed} with $q \curve r_1/\rho$ rather than Propositions \ref{prop:stabY} and \ref{prop:stabUbar} in the proof of Lemma \ref{lem:reduceUbarUtoUkUbar} gives a bound in terms of $\sup_{t\in [0,T]} \|U^k(t)-\bar{U}(t)\|_{r_2,X}$ with prefactor $(1+\|u_0\|_{r_1,Y}^{\rho \nu})$.
    
    The latter is estimated as in Proposition \ref{prop:diffUkUbar} but using Proposition \ref{prop:UhatUbarErrorTamed} with $q \curve r_2$ in lieu of Proposition \ref{prop:UhatUbarError}. It is applicable because $\beta r_2=r_1$ and Assumptions \ref{ass:consistencyTaming}($\alpha,\beta,r_2$), \ref{ass:growthBoundTaming}($\theta_F,\theta_G$), and \ref{ass:tamedCoercivity}($r_1$) hold, and $u_0\in L_{\F_0}^{r_1}(\Omega;Y)$. Since $r_2\le \rho r_2\le \beta r_2$, Assumption \ref{ass:tamedCoercivity}($\rho r_2$) and $u_0 \in L_{\F_0}^{\rho r_2}(\Omega;Y)$, which follow from the respective statements for $r\ge \rho r_2$, imply that the applications of Proposition \ref{prop:stabY} with $q \curve r_2$ and $q \curve \rho r_2$ can both be replaced by Proposition \ref{prop:stabYtamed} with the same parameters.

    In the proof of Theorem \ref{thm:convergenceMain}, rather than applying Proposition \ref{prop:UhatUbarError} with $q \curve p$, we use the embedding $L^{r_2}(\Omega)\hra L^p(\Omega)$ and then apply \ref{prop:UhatUbarErrorTamed} with $q \curve r_2$, which uses $\beta r_2=r$, Assumptions \ref{ass:consistencyTaming}($\alpha, \beta,r_2$) and \ref{ass:tamedCoercivity}($r$) as well as $u_0\in L_{\F_0}^r(\Omega;Y)$. Finally, Proposition \ref{prop:diffUkUbar} is applied with the changes discussed in the last paragraph. Combining the factors $(1+\|u_0\|_{r_1,Y}^{\rho \nu})$ resulting from Lemma \ref{lem:reduceUbarUtoUkUbar} and $(1+\|u_0\|_{r_1,Y}^{\beta})$ from Proposition \ref{prop:UhatUbarErrorTamed} yields the bound with $1+\|u_0\|_{r,Y}^{r/q}$ due to $r/q=r_1/q=\beta+\rho\nu$.
\end{proof}

\begin{remark}
\label{rem:tamingInTermsOfakbk}
    We comment on scalar tamings, which include the nonlinearities-stopped exponential Euler scheme as a special case. For scalar tamings given by $a_k,b_k:Y\to [0,1]$ as in \eqref{eq:scalarTaming}, Assumptions \ref{ass:tamedCoercivity}($q$), \ref{ass:growthBoundTaming}($\theta_F,\theta_G$), and \ref{ass:consistencyTaming}($\alpha,\beta,q$) simplify to $\theta_F+\theta_G \le \frac{1}{2}$ and
    \begin{align}
    \label{eq:tamedCoercivityakbk}
        a_k(x) \Re \la x,F(x) \ra_Y + \frac{q-1}{2}b_k(x)^2 \|G(x)\|_\LHY^2 &\le C_q(1+\|x\|_Y^2),\\
        a_k(x)\|F(x)\|_Y &\le C_{T,\theta} k^{-\theta_F}(1+\|x\|_Y),\nonumber\\ 
        b_k(x)\|G(x)\|_\LHY &\le C_{T,\theta} k^{-\theta_G}(1+\|x\|_Y),\nonumber\\
        \label{eq:tamedConsistencyakbk}
        \|(1-a_k(\Phi))F(\Phi)\|_{q,Y}+\nn (1-b_k(\Phi))G(\Phi)\nn_{q,Y} &\le C_{\alpha,\beta,q}k^\alpha \big(1+\|\Phi\|_{\beta q,Y}^\beta\big)
    \end{align}
    for all $x\in Y$, $k\in (0,T]$, and $\Phi \in L^{\beta q}(\Omega;Y)$. 
    The growth bound \eqref{eq:tamedNLsmaller} is automatically satisfied due to $a_k(x),b_k(x)\in[0,1]$. A useful sufficient condition for the tamed coercivity \eqref{eq:tamedCoercivityakbk} to hold is
    \begin{equation}
    \label{eq:tamedCoercivityakbk01}
        0\le b_k(x)^2 \le a_k(x)\le 1,\quad x\in Y
    \end{equation}
    together with coercivity of the untamed nonlinearities, Assumption \ref{ass:coYq}($q$). In particular, every common scalar taming with $[0,1]$-valued $a_k=b_k$ satisfies \eqref{eq:tamedCoercivityakbk01} and thus preserves coercivity. This also applies to the nonlinearities-stopped exponential Euler scheme as noted in Remark \ref{rem:coercivityFGtheta}. A sufficient condition for \eqref{eq:tamedConsistencyakbk}, the consistency of the taming, is given by the pointwise estimate
    \begin{equation}
    \label{eq:sufficientConsistencyakbk}
        (1-a_k(x))\|F(x)\|_Y+(1-b_k(x))\nn G(x)\nn_Y \le C_{\alpha,\beta} k^\alpha\big(1+\|x\|_Y^\beta\big),\quad x\in Y.
    \end{equation}
    For the nonlinearities-stopped exponential Euler scheme, cf.\ Definition \ref{def:Scheme}, Theorem \ref{thm:convergenceMainTamed} is applicable due to \eqref{eq:akbkNLstoppedExpEuler} with $\theta_F=\theta_G=\theta\in (0,\frac{1}{4}]$ and $\beta=\rho(1+\frac{\alpha}{\theta})$ and yields precisely the same bound as Theorem \ref{thm:convergenceMain} under the same regularity assumptions.
\end{remark}

With the convergence theorem for general tamings at hand, we can reconsider the nonlinearities-stopped exponential Euler and ask what the optimal way to stop the nonlinearities is for given $F$ and $G$ of different polynomial growth orders $\rho_F,\rho_G\ge 1$. Here, optimality concerns the required integrability of the initial values as well as the coercivity parameter, which e.g.\ influences the required smallness of the quadratic noise for the transport equation considered in Subsection \ref{subsec:transportExample}. One is tempted to stop each nonlinearity separately whenever its growth exceeds the respective growth bound $\theta_F$ or $\theta_G$, i.e.\ in case $\theta_F,\theta_G>0$ taking
\begin{equation*}
    F_k(x) = \1_{\{\|F(x)\|_Y\le k^{-\theta_F}\}}F(x),\quad G_k(x) = \1_{\{\|G(x)\|_\LHY\le k^{-\theta_G}\}}G(x)
\end{equation*}
and untamed $F_k=F$ or $G_k=G$ if the respective parameter vanishes. Subsequently, the choice of $\theta_F,\theta_G\in [0,\frac12]$ could be optimised. However, this taming may violate the tamed coercivity from Assumption \ref{ass:tamedCoercivity}($r$) since $b_k^2\le a_k$ may fail and thus \eqref{eq:tamedCoercivityakbk01} need not be satisfied. Hence, we consider a coercivity-preserving modification of this stopping, where $b_k$ also depends on $\theta_F$, and optimise the stopping parameters for it.

\begin{corollary}[Optimal stopping parameters]
\label{cor:optimalSeparateTaming}
    Let $p\in [2,\infty)$, $\lambda \in (0,1)$ and $q=\frac{p}{\lambda}$. Suppose that Assumption \ref{ass:locLipPolGro}($\alpha,\nu,\rho$) holds for some $\alpha\in (0,\frac{1}{2}]$ and $\nu,\rho\ge 1$ and that 
    \begin{align*}
        \|F(x)\|_Y &\le C_{\rho_F}\big(1+\|x\|_Y^{\rho_F}\big),\qquad
        \|G(x)\|_\LHY \le C_{\rho_G}\big(1+\|x\|_Y^{\rho_G}\big)
    \end{align*}
    for $\rho_F=\rho$ and some $\rho_G\in [1,\rho_F]$ and $C_{\rho_F},C_{\rho_G}\ge 0$ for all $x\in Y$. Let $(U^j)_{0\le j\le N}$ be the tamed exponential Euler scheme from Definition \ref{def:TamedScheme} with tamed nonlinearities
    \begin{align}
\label{eq:FkGkoptimalStopping}
        F_k(x) \ce \1_{\{\|F(x)\|_Y\le k^{-\theta_F}\}}F(x),\quad
        G_k(x)\ce \1_{\{\|F(x)\|_Y\le k^{-\theta_F},\ \|G(x)\|_\LHY\le k^{-\theta_G}\}}G(x)
    \end{align}
    with
    \begin{align*}
        \theta_F &\ce \frac{\alpha\rho_F}{\beta_*-\rho_F},\qquad
        \theta_G \ce \frac{\alpha\rho_G}{\beta_*-\rho_G},\\
        \beta_* &\ce \frac{1}{2}\Big((1+2\alpha)(\rho_F+\rho_G)+\sqrt{(1+2\alpha)^2(\rho_F+\rho_G)^2-4(1+4\alpha)\rho_F\rho_G}\Big).
    \end{align*} 
    Let $r\ce q(\beta_*+\rho\nu)$ and $u_0\in L_{\F_0}^r(\Omega;Y)$. Further, suppose that Assumptions \ref{ass:coYq}($r$) and \ref{ass:monqEta}($q,\eta$) hold for some $\eta>0$. 
    
    Then the tamed exponential Euler scheme $(U^j)_{0\le j\le N}$ converges at rate $\alpha$ to the solution $(U(t))_{t\in [0,T]}$ of \eqref{eq:SEEsemilinear} and there is a constant $C\ge 0$ depending on $(\alpha,T,p,\lambda,\nu,\rho_F,\rho_G,\eta)$ but independent of $k$, $N$, and $u_0$ such that 
     \begin{align*}
        \Big\|\max_{0\le j \le N} \|U^j-U(t_j)\|_X\Big\|_{L^p(\Omega)} \le \bigg\|\sup_{t\in[0,T]} \|\hat{U}(t)-U(t)\|_X\bigg\|_{L^p(\Omega)}
        &\le C\big(1+\|u_0\|_{L^r(\Omega;Y)}^{r/q} \big)k^\alpha.
    \end{align*}
    If $G$ is of linear growth on $Y$, the same claim holds for $G_k=G$, $F_k$ as in \eqref{eq:FkGkoptimalStopping}, $\theta_G=0$, and $\theta_F=\frac{1}{2}$ with $\beta_*=\rho_F(1+2\alpha)$. 
\end{corollary}
\begin{proof}
The claim follows from Theorem \ref{thm:convergenceMainTamed} provided that Assumptions \ref{ass:tamedCoercivity}($r$),  \ref{ass:growthBoundTaming}($\theta_F,\theta_G$) and \ref{ass:consistencyTaming}($\alpha,\beta_*,q(1+\frac{\rho \nu}{\beta_*})$) hold and $\beta_* \ge \rho=\rho_F$. By definition of $F_k$ and $G_k$, the indicators in \eqref{eq:FkGkoptimalStopping} satisfy $0\le b_k(x)\le a_k(x)\le 1$ and thus also \eqref{eq:tamedCoercivityakbk01}, which together with coercivity of $F$ and $G$ implies Assumption \ref{ass:tamedCoercivity} (see Remark \ref{rem:tamingInTermsOfakbk}).

Next, note that $\beta_*$ is the larger root of the quadratic equation
\begin{align}
\label{eq:optimalBetaQuadratic}
    \beta^2 -(1+2\alpha)(\rho_F+\rho_G)\beta +(1+4\alpha)\rho_F\rho_G=0.
\end{align}
Rearranging this directly gives $\theta_F+\theta_G=\frac{1}{2}$. The growth bounds \eqref{eq:tamedGrowthBounds} and \eqref{eq:tamedNLsmaller} are immediate from the definition of $a_k$ and $b_k$, which verifies Assumption \ref{ass:growthBoundTaming}.

Define
\begin{align*}
    \beta_F &\ce \rho_F\Big(1+\frac{\alpha}{\theta_F}\Big),\quad
    \beta_{G,1}\ce \rho_G\Big(1+\frac{\alpha}{\theta_G}\Big),\quad
    \beta_{G,2}\ce \rho_G+\alpha\frac{\rho_F}{\theta_F}.
\end{align*}
We verify Assumption \ref{ass:consistencyTaming}($\alpha,\beta_*,q(1+\frac{\rho\nu}{\beta_*})$) by checking the sufficient pointwise condition in \eqref{eq:sufficientConsistencyakbk}. If $\theta_F=0$, the difference vanishes. Using that $\1_{\{R>k^{-\theta}\}} \le k^\alpha R^{\alpha/\theta}$ for all $R\ge 0$ and $\theta>0$ as well as polynomial growth of $F$ gives 
\begin{align*}
    \|F_k(x)-F(x)\|_Y = \1_{\{\|F(x)\|_Y>k^{-\theta_F}\}}\|F(x)\|_Y \le k^\alpha\|F(x)\|_Y^{1+\alpha/\theta_F} \le C_{\alpha,\theta_F,\rho_F} k^\alpha\big(1+\|x\|_Y^{\beta_F}\big).
\end{align*}
For $G_k$, since $b_k$ is the indicator of the intersection of two events, we can bound $1-b_k$ by the sum of the respective indicators. Repeating the arguments thus results in
\begin{align*}
    \|G_k(x)-G(x)\|_\LHY
    &\le k^\alpha\big(\|G(x)\|_\LHY^{1+\alpha/\theta_G}+\|F(x)\|_Y^{\alpha/\theta_F}\|G(x)\|_\LHY\big)\\
    &\le C_{\alpha,\theta_F,\theta_G,\rho_F,\rho_G} k^\alpha\big(1+\|x\|_Y^{\beta_{G,1}}+\|x\|_Y^{\beta_{G,2}}\big).
\end{align*}
Since $\rho_G\le \rho_F$, also $\beta_{G,2}\le \rho_F+\alpha\frac{\rho_F}{\theta_F}=\beta_F$. Moreover, the choice of $\beta_*$ is such that $\beta_*=\beta_F=\beta_{G,1}$. Hence, pointwise the differences are bounded by $Ck^\alpha(1+\|x\|_Y^{\beta_*})$. 

Linear growth of $G$ means that \eqref{eq:tamedGrowthBounds} for $G_k=G$ is satisfied with $\theta_G=0$ and trivially \eqref{eq:tamedNLsmaller} holds. Moreover, the difference $G_k(\Phi)-G(\Phi)$ vanishes. Consequently, Assumptions \ref{ass:growthBoundTaming}($\frac12,0$) and \ref{ass:consistencyTaming}($\alpha,\beta,q(1+\frac{\rho\nu}{\beta_*})$) are satisfied by the above considerations on $F_k$. Due to linear growth of $G_k$, the tamed coercivity Assumption \ref{ass:tamedCoercivity}($r$) reduces to an assumption on $F_k$, which holds because of $a_k(x)\le 1$ and untamed coercivity.
Theorem \ref{thm:convergenceMainTamed} then yields both claims.
\end{proof}

Optimality of $\beta_*$ stems from $\beta_*=\beta_F=\beta_{G,1}$ resulting in minimal $r$. If $\rho_G>\rho_F$, the coupling term in $b_k$ dominates, i.e.\ $\beta_{G,2}>\beta_F$. The optimal parameters are then given by
\begin{equation*}
    \theta_F=\frac{\rho_F}{2(\rho_F+\rho_G)},\quad \theta_G=\frac{\rho_G}{2(\rho_F+\rho_G)},\quad \beta_*=\beta_{G,2}=\rho_G+2\alpha(\rho_F+\rho_G).
\end{equation*}
As a consequence of the optimal parameters for $\rho_F\ge \rho_G$, $\theta_F=\theta_G=\frac{1}{4}$ is optimal in terms of regularity for $\theta_F=\theta_G$.

Lastly, we analyse a scheme that tames the nonlinearities in a continuous manner rather than stopping discontinuously and comment on the limitation to rate $\frac12$.

\begin{example}[Fractional taming]
    \label{ex:fractionalTaming}
    In the context of weak convergence rates for the stochastic heat equation with polynomial $F$ and additive noise, Bréhier \cite{Brehier22tamedExpEuler} considered the taming corresponding to $a_k(x)\ce (1+k\|F(x)\|_X)^{-1}$. To adapt this to the hyperbolic setting, strong rates, and multiplicative superlinearly growing noise, we consider the \emph{fractionally tamed exponential Euler scheme} given by $F_k= a_k F$ and $G_k=b_kG$ with
    \begin{align}
    \label{eq:fractionalTaming}
        a_k(x)\ce b_k(x)\ce \frac{1}{1+k^\alpha R(x)^{4\alpha}},\quad R(x)\ce\|F(x)\|_Y+\|G(x)\|_\LHY.
    \end{align}
    
    Let $\alpha\in [\frac14,\frac12]$. Under the assumptions of Theorem \ref{thm:convergenceMainTamed} on the untamed nonlinearities, the semigroup, and $X,Y$ with $r=\rho q(1+\nu+4\alpha)$, the scheme converges at rate $\alpha$ in a pathwise uniform sense for all $u_0\in L^r(\Omega;Y)$. If $G$ is globally Lipschitz on $X$ and of linear growth on $Y$, the noise does not need to be tamed. Pathwise uniform convergence at rate $\alpha\in (0,\frac12]$ then also holds for the tamed scheme with $G_k\ce G$ and $a_k(x)\ce (1+k^\alpha\|F(x)\|_Y)^{-1}$ under the same conditions with $r=\rho q(2+\nu)$.
    
    Indeed, $a_k=b_k$ are $[0,1]$-valued, which ensures tamed coercivity as in Assumption \ref{ass:tamedCoercivity}($r$) via Remark \ref{rem:tamingInTermsOfakbk} in the case of \eqref{eq:fractionalTaming} or linear growth of $G$ on $Y$ else, and untamed coercivity. For $G$ and thus $G_k$ of linear growth on $Y$, $\|F_k(x)\|_Y\le k^{-\alpha}$ is immediate from the definition, so Assumption \ref{ass:growthBoundTaming}($\alpha,0$) holds. In the case of \eqref{eq:fractionalTaming}, Assumption \ref{ass:growthBoundTaming}($\frac14,\frac14$) holds due to $4\alpha\ge 1$ and thus
    \begin{align*}
        a_k(x)R(x)=b_k(x)R(x)=k^{-1/4}\frac{k^{1/4} R(x)}{1+(k^{1/4}R(x))^{4\alpha}} \le C_\alpha k^{-1/4}.
    \end{align*}
    Since for \eqref{eq:fractionalTaming},
    \begin{equation*}
        (1-a_k(x))\|F(x)\|_Y\le \frac{k^\alpha R(x)^{4\alpha}}{1+k^\alpha R(x)^{4\alpha}}R(x) \le k^\alpha R(x)^{1+4\alpha}\le C_{\alpha,\rho} k^\alpha \big(1+\|x\|_Y^{\rho(1+4\alpha)}\big)
    \end{equation*}
    and likewise for $(1-b_k(x))\|G(x)\|_\LHY$, 
    Assumption \ref{ass:consistencyTaming}($\alpha,\beta,q(1+\frac{\rho\nu}{\beta})$) holds with $\beta\ce \rho(1+4\alpha)$. This $\beta$ also satisfies $r=q(\beta+\rho\nu)=\rho q(1+\nu+4\alpha)$, so that Theorem \ref{thm:convergenceMainTamed} is applicable. In the global Lipschitz case, $G_k-G\equiv 0$ and
    \begin{align*}
        (1-a_k(x))\|F(x)\|_Y &= \frac{k^\alpha\|F(x)\|_Y}{1+k^\alpha\|F(x)\|_Y}\|F(x)\|_Y
        \le k^\alpha\|F(x)\|_Y^2
        \le C_\rho k^\alpha\big(1+\|x\|_Y^{2\rho}\big),
\end{align*}
    whence Assumption \ref{ass:consistencyTaming} holds with $\beta=2\rho$ and the claim follows from Theorem \ref{thm:convergenceMainTamed}.

    The restriction $\alpha\ge \frac14$ to higher rates in the case of \eqref{eq:fractionalTaming} is necessary to verify the growth bounds for the taming. For lower $\alpha\in (0,\frac14)$, the choice $a_k(x)=b_k(x)=(1+k^\alpha R(x))^{-1}$ results in convergence at rate $\alpha$ for $\theta_F=\theta_G=\alpha$, $\beta=2\rho$, and thus $r=\rho q(2+\nu)$. 
\end{example}

\begin{remark}[Higher-order convergence]
    For multiplicative noise, the rate $\frac12$ is the natural upper limit for time discretisations using only first-order Wiener increments due to Lévy's modulus of continuity (see e.g.\ \cite[Theorem 3]{Muller-Gronbach}). To obtain higher convergence rates, more information from the Brownian motion needs to be incorporated, such as iterated stochastic integrals in the exponential Milstein scheme \cite{JentzenRoeckner15_Milstein}. For hyperbolic SPDEs with globally Lipschitz nonlinearities and multiplicative noise, pathwise uniform convergence at rate $1$ was obtained in \cite{KastnerKlioba26Milstein} by Kastner and the author. It seems plausible that a combination of the additional Milstein terms with the taming mechanisms developed here could yield convergence rates up to $1$ also in the superlinear case under additional assumptions, but goes beyond the scope of this work.
\end{remark}

\section{Applications to hyperbolic SPDEs with polynomial nonlinearities}
\label{sec:applications}

In this section, we apply the convergence results from Sections \ref{sec:convergenceRate} and \ref{sec:generalTamings} to obtain convergence rates for different tamed exponential Euler schemes for four semilinear non-parabolic SPDEs with polynomial nonlinearities. First, the nonlinear stochastic transport equation is considered with an Allen--Cahn nonlinearity in Subsection \ref{subsec:transportExample}. The negative sign of the nonlinearity is leveraged to compensate for sufficiently small quadratic noise. Lower moment requirements on the initial values and lower admissible regularity of the noise covariance are achieved by leveraging fractional kernel estimates in order to take $Y=H^{1/2+\varepsilon}$ for small $\varepsilon>0$. Passing to second-order equations, a nonlinear Schrödinger equation with dissipative damping is analysed in Subsection \ref{subsec:dissipativeNLSExample} and the Klein--Gordon equation with damping in the velocity variable, a wave-type equation, in Subsection \ref{subsec:kleinGordon}. For both, optimal convergence at rate $\frac12$ is shown for nonlinearities-stopped and fractionally tamed exponential Euler schemes, complementing the convergence rate $1$ obtained for nonlinearly damped wave-type equations with additive noise in \cite{CaiCohenWang25,CuiHongSun25KleinGordon}. Lastly, Subsection \ref{subsec:AiryExample} presents the Airy equation as a third-order example to showcase the influence of higher-order polynomial nonlinearities as well as limitations to suboptimal convergence rates for higher-order equations in the current framework.

\subsection{The nonlinear stochastic transport equation}
\label{subsec:transportExample}

Consider the stochastic transport equation with an Allen--Cahn nonlinearity and small quadratic noise on the one-dimensional torus $\T\ce \R/(2\pi\Z)$ 
\begin{align}\tag{NLT}\label{eq:transport}
    \rmd U + c\grad U \dt = (U-U^3)\dt + \sigma |U|U \dWQ\text{ on }[0,T],\quad U(0)=u_0 \in L^2(\T),
\end{align}
where $c\in\R\setminus\{0\}$, $\sigma\in\R$ has sufficiently small modulus to be specified later, and $W_Q$ is a $Q$-Wiener process on $L^2(\T)$ with non-negative, self-adjoint covariance operator $Q$ of trace class. For Lévy noise, globally Lipschitz $F,G$, and a full discretisation via discontinuous Galerkin and a backwards Euler scheme, mean-square convergence at rate $\frac12$ was shown in \cite[Thm.~5.1]{BarthStein25}. Also for globally Lipschitz $F$ and $G$, pathwise uniform convergence at rate $1$ was obtained for the higher-order Milstein scheme in \cite{KastnerKlioba26Milstein}.

Let $\varepsilon\in (0,\frac12]$ and set $s\ce \frac12+\varepsilon$. We verify the conditions of \ref{thm:convergenceMain} for $X=H=L^2(\T)$, 
\begin{equation}
\label{eq:transportChoiceYFG}
    Y=H^s(\T),\quad F(u)=u-u^3,\quad\text{ and }G(u)=\sigma M_{u|u|}Q^{1/2},
\end{equation}
where $M_{u|u|}$ denotes the multiplication operator associated with $u|u|$, as well as $A=c\grad$ on $D(A)=H^1(\T)$. In this subsection only, denote by $L^p$, $p\in [2,\infty]$ and $H^s$ the corresponding Lebesgue and Bessel potential spaces over $\T$, respectively. We equip $Y=H^s$ with the norm
\begin{align}
\label{eq:defDsTransport}
	\|u\|_Y^2 \ce \|u\|_{L^2(\T)}^2+\|D^su\|_{L^2(\T)}^2,\quad D^s\ce (-\Delta)^{s/2}.
\end{align}
Then $-A$ generates the translation (semi-)group, which is unitary and thus contractive on $L^2$ and $H^s$~\cite[I.4.18]{EngelNagel}. Moreover, since $D(A)=H^1$, we have $Y=H^s=D(A^s) \hra D(A^{1/2})\hra D_A(\frac12,\infty)$, so that we can take $\alpha=\frac12$ in Assumption \ref{ass:locLipPolGro}.

Note that $s>\frac12$ implies that $Y=H^s$ is a Banach algebra as a consequence of the fractional Leibniz rule on the torus~\cite[Proposition~1]{BenyiOhZhao25} and $H^s\hra L^\infty$. The Banach algebra property directly gives polynomial growth of $F$ on $H^s$ with $\rho_F=3$. Let $(e_n)_{n\in\N}$ be an orthonormal basis of $L^2(\T)$ and denote by $C_s>0$ the Banach algebra constant. Then
\begin{align*}
	\nn G(u)\nn_Y^2 &= \sigma^2 \sum_{n\in\N}\|u|u|Q^{1/2}e_n\|_{H^s}^2 \le C_s^2 \sigma^2 \sum_{n\in\N}\|u|u|\|_{H^s}^2\|Q^{1/2}e_n\|_{H^s}^2
    \le C_s^4 \sigma^2 \nn Q^{1/2}\nn_{H^s}^2 \|u\|_{H^s}^4,
\end{align*}
that is, polynomial growth of $G$ on $H^s$ with $\rho_G=2$ provided that $Q^{1/2}\in \calL_2(L^2,H^s)$. The factorization $u^3-v^3=(u^2+uv+v^2)(u-v)$, Hölder's inequality, and the embedding $H^s\hra L^\infty$ imply local Lipschitz continuity of $F$ with $\nu=2$ via
\begin{align*}
	\|F(u)-F(v)\|_X\le \big(1+\|u^2+uv+v^2\|_{L^\infty}\big)\|u-v\|_{L^2} \le C_{H^s\hra L^\infty}^2\big(1+\|u\|_{H^s}^2+\|v\|_{H^s}^2\big)\|u-v\|_{L^2}
\end{align*}
for all $u,v\in H^s$.
Likewise, $G$ is locally Lipschitz due to
\begin{align*}
	\nn G(u)-G(v)\nn_X^2&\le \sigma^2 \sum_{n\in\N} \||u|+|v|\|_{L^\infty}^2\|u-v\|_{L^2}^2\|Q^{1/2}e_n\|_{L^\infty}^2\\
	& \le \sigma^2 C_{H^s\hra L^\infty}^2\nn Q^{1/2}\nn_{H^s}^2(\|u\|_{H^s}+\|v\|_{H^s})^2\|u-v\|_{L^2}^2.
\end{align*}
Consequently, Assumption \ref{ass:locLipPolGro}($\frac12,2,3$) is satisfied.

Next, we verify the monotonicity condition from Assumption \ref{ass:monqEta}($q,\eta$) for some $q>2$, $\eta>0$. To this end, we first observe that by the mean value theorem and the Cauchy--Schwarz inequality, one has
\begin{equation}
\label{eq:transport43}
    (a|a|-b|b|)^2=\Big(\int_a^b 2|t|\dt\Big)^2\le 4(b-a)\int_a^b t^2\dt= \frac{4}{3}(a-b)(a^3-b^3),\quad a,b \in \R.
\end{equation}
This allows us to bound
\begin{align*}
    \Re &\la u-v,F(u)-F(v)\ra_{L^2}+(1+\eta)\frac{q-1}{2}\nn G(u)-G(v)\nn_{L^2}^2\\
    &= \int_\T (u-v)^2-(u-v)(u^3-v^3)\dx + \sigma^2 (1+\eta)\frac{q-1}{2}\sum_{n\in\N}\|(u|u-v|v|)Q^{1/2}e_n\|_{L^2}^2\\
    &\le \|u-v\|_{L^2}^2-\Big(1- \frac{4}{3}\sigma^2Q_\infty^2 (1+\eta)\frac{q-1}{2}\Big)\int_\T(u-v)(u^3-v^3)\dx\le \|u-v\|_{L^2}^2
\end{align*}
with $Q_\infty^2\ce \sum_{n\in\N} \|Q^{1/2}e_n\|_{L^\infty}^2 \le C_s^2\nn Q^{1/2}\nn_{H^s}^2$ for all $s>\frac12$ provided that
\begin{equation}
\label{eq:transportSigmaMon}
    |\sigma|\le \sqrt{\frac{3}{2(1+\eta)(q-1)}}\frac{1}{Q_\infty},\quad Q_\infty = \Big(\sum_{n\in\N} \|Q^{1/2}e_n\|_{L^\infty}^2\Big)^{1/2}.
\end{equation}

Showing coercivity as in Assumption \ref{ass:coYq}($r$) for some $r>2$ is significantly more straightforward in $Y=H^1$ than for the fractional Bessel potential spaces $Y=H^s$ with $s\in (\frac12,1)$. However, taking $s$ closer to $\frac12$ allows us both to relax the required regularity of the initial values and the covariance operator. Hence, we include both cases, starting with $s=1$. Here, we can work with the usual Sobolev norm in $H^1$ and explicitly compute the scalar product
\begin{align}
\label{eq:transportFcoH1}
	\Re \la u,F(u)\ra_{H^1} &= \la u,u \ra_{H^1} - \la u,u^3\ra_{L^2} - \la \partial_x u, \partial_x(u^3) \ra_{L^2} = \|u\|_{H^1}^2 - \int u^4 \dx - 3\int (\partial_x u)^2u^2 \dx\nonumber\\
	& = \|u\|_{H^1}^2-\|u\|_{L^4}^4 - 3\|u\partial_x u\|_{L^2}^2.
\end{align}
Clearly, the sign of the nonlinearity is crucial here in order to obtain a bound in terms of $\|u\|_{H^1}^2$. Repeating the arguments for polynomial growth of $G$ and using that $\partial_x(u|u|)=2|u|\partial_xu$, we deduce
\begin{align*}
	\nn G(u)\nn_{H^1}^2 &\le C_1^2 \sigma^2 \nn Q^{1/2}\nn_{H^1}^2 \|u|u|\|_{H^1}^2 = C_1^2 \sigma^2 \nn Q^{1/2}\nn_{H^1}^2\big(4\|u\partial_x u\|_{L^2}^2+\|u\|_{L^4}^4\big).
\end{align*}
Taking this upper bound times $\frac{r-1}{2}$, adding it to \eqref{eq:transportFcoH1} and rearranging, we obtain that $r$-coercivity as in Assumption \ref{ass:coYq}($r$) for $Y=H^1$ is satisfied under the smallness condition
\begin{align*}
    |\sigma|\le \sqrt{\frac{3}{2C_1^2(r-1)}}\nn Q^{1/2}\nn_{H^1}^{-1}.
\end{align*}

Now we pass to verifying coercivity in $Y=H^s$ with $s\in (\frac12,1)$. This relies on pointwise representations of the fractional Laplacian on the torus with a positive kernel from \cite[Thm.~1.5]{RoncalStinga16}. Namely, for $0<\vartheta<2$ there is a positive, periodic, even
kernel $K_{\vartheta/2}$ with $K_{\vartheta/2}(x) \eqsim_\vartheta |x|^{-1-\vartheta}$ near $x=0$ such that
\begin{align}
\label{eq:transportKernelRepresentation}
    D^\vartheta f(x) &=\PV\int_{\T}\big(f(x)-f(y)\big)K_{\vartheta/2}(x-y)\dy
\end{align}
for smooth periodic $f$. With $D^s$ as in \eqref{eq:defDsTransport},
\begin{align*}
    \Re\la u,F(u)\ra_{H^s} = \|u\|_{H^s}^2 - \|u\|_{L^4}^4-\la D^s u, D^s (u^3)\ra_{L^2}.
\end{align*}
For the last term, self-adjointness of the fractional Laplacian, \eqref
{eq:transportKernelRepresentation}, and symmetrising the integral via symmetry of the kernel and its behaviour around $0$ give 
\begin{align*}
    \la D^s u, D^s (u^3)\ra_{L^2} &= \la D^{2s}u,u^3\ra_{L^2} = \int_\T \PV\int_\T (u(x)-u(y))K_s(x-y)\dy \,u^3(x)\dx\\
    &= \frac{1}{2} \int_\T \int_\T (u(x)-u(y)) \big(u^3(x)-u^3(y)\big) K_s(x-y)\dy \dx.
\end{align*}
The noise term is bounded by
\begin{align*}
    \nn G(u)\nn_{H^s}^2&\le C_s^2 \sigma^2 \nn Q^{1/2}\nn_{H^s}^2 \|u|u|\|_{H^s}^2 
    \le C_s^2 \sigma^2 \nn Q^{1/2}\nn_{H^s}^2 \big(\|u\|_{L^4}^4+\|D^s(u|u|)\|_{L^2}^2\big),
\end{align*}
where the kernel representation and \eqref{eq:transport43} for the last term give
\begin{align*}
    \|D^s(u|u|)\|_{L^2}^2&=\la D^{2s}(u|u|),u|u|\ra_{L^2}=\frac12 \int_\T\int_\T \big((u|u|)(x)-(u|u|)(y)\big)^2 K_s(x-y)\dy\dx\\
    &\le \frac{2}{3} \int_\T\int_\T (u(x)-u(y))(u(x)^3-u(y)^3) K_s(x-y)\dy\dx = \frac43 \la D^s u, D^s(u^3)\ra_{L^2}.
\end{align*}
Altogether,
\begin{align*}
    \Re \la u,F(u)\ra_{H^s}+\frac{r-1}{2}\nn G(u)\nn_{H^s}^2 &\le \|u\|_{H^s}^2 -\Big(1-\frac{r-1}{2}C_s^2\sigma^2\nn Q^{1/2}\nn_{H^s}^2\Big)\|u\|_{L^4}^4\\
    &\phantom{\le }-\Big(1-\frac{2(r-1)}{3}C_s^2\sigma^2\nn Q^{1/2}\nn_{H^s}^2\Big)\la D^su,D^s(u^3)\ra_{L^2}\le \|u\|_{H^s}^2
\end{align*}
under the condition $\sigma^2 \le \frac32 (r-1)^{-1}C_s^{-2}\nn Q^{1/2}\nn_{H^s}^{-2}$.

The identities and inequalities extend from smooth functions to all
$u\in H^s$ by density, since $H^s$ is a Banach
algebra and $u\mapsto u^3$ as well as $u\mapsto u|u|$ are continuous on $H^s$. As a consequence of Theorems \ref{thm:aprioriY} and \ref{thm:convergenceMain}, we thus obtain global well-posedness and pathwise uniform convergence of the nonlinearities-stopped exponential Euler scheme at rate $\frac12$. 

\begin{theorem}
\label{thm:transportNLstop}
    Let $\varepsilon\in(0,\frac{1}{2}]$, set $s=\frac{1}{2}+\varepsilon$, and suppose that $Q^{1/2}       \in\calL_2(L^2(\T),H^s(\T))$. 
    Further, let $c\in \R\setminus\{0\}$, $p\in[2,\infty)$, $\lambda\in(0,1)$, and $q=\frac{p}{\lambda}$, and set $r\ce 15q$. If
    \begin{align*}
        |\sigma|<\min\bigg\{ \sqrt{\frac{3}{2C_s^2(r-1)}}\| Q^{1/2}\|_{\calL_2(L^2,H^s)}^{-1}, \sqrt{\frac{3}{2(q-1)}}\frac{1}{Q_\infty}\bigg\}
    \end{align*}
    with $Q_\infty$ as in \eqref{eq:transportSigmaMon}.
    Then for all $u_0\in L_{\F_0}^r(\Omega;H^s(\T))$, \eqref{eq:transport} admits a unique global mild solution $U\in L^r(\Omega;C([0,T];H^s(\T)))$ and the nonlinearities-stopped exponential Euler scheme $(U^j)_{0\le j\le N}$ with $Y,F,G$ as in \eqref{eq:transportChoiceYFG} converges pathwise uniformly at rate $\frac{1}{2}$. Moreover, there is a constant $C\ge 0$ independent of $k$, $N$, and $u_0$ such that
    \begin{align*}
        \Big\|\max_{0\le j\le N} \|U^j-U(t_j)\|_{L^2(\T)}\Big\|_{L^p(\Omega)} \le C\big(1+\|u_0\|_{L^{15q}(\Omega;H^s(\T))}^{15}\big)k^{1/2}.
    \end{align*}
\end{theorem}

We observe that the parameter controlling the integrability of $u_0$ and the coercivity constant satisfies $r\ge 15q$ for any admissible $\theta\in (0,\frac{1}{4}]$ with $\theta=\frac14$ being optimal. In order to lower this parameter, we use different stopping indicators for $F$ and $G$ that leverage the different growth parameters $\rho_F=3$ and $\rho_G=2$ as in Corollary \ref{cor:optimalSeparateTaming}.

\begin{theorem}
    Suppose that the assumptions of Theorem \ref{thm:transportNLstop} hold for $r\ce (11+\sqrt{7})q$ and let $u_0\in L_{\F_0}^r(\Omega;H^s(\T))$. Consider the tamed exponential Euler scheme $(U^j)_{0\le j\le N}$ with
    \begin{align*}
        F_k(u)\ce \1_{\{\|F(u)\|_{H^s}\le k^{-\theta_F}\}}F(u),\quad G_k(u)\ce \1_{\{\|F(u)\|_{H^s}\le k^{-\theta_F}, \nn G(u)\nn_{H^s}\le k^{-\theta_G}\}}G(u)
    \end{align*}
    for $u\in H^s(\T)$, where $\theta_F\ce \frac12(\sqrt{7}-2)$ and $\theta_G\ce \frac12(3-\sqrt{7})$. Then for some constant $C\ge 0$ independent of $k$, $N$, and $u_0$,
    \begin{align*}
        \Big\|\max_{0\le j\le N}\|U^j-U(t_j)\|_{L^2(\T)}\Big\|_{L^p(\Omega)}
        &\le C\big(1+\|u_0\|_{L^{(11+\sqrt{7})q}(\Omega;H^s(\T))}^{11+\sqrt{7}}\big)k^{1/2}.
    \end{align*}
\end{theorem}

Compared with the common stopping parameter $\theta=\theta_F=\theta_G=\frac14$ in Theorem \ref{thm:transportNLstop}, this reduces $r=15q$ to $r=(11+\sqrt{7})q\approx 13.6 q$. Moreover, the smallness condition on the noise originating from the coercivity condition is relaxed, while the one coming from the monotonicity condition remains unchanged. Convergence is also obtained for the fractionally tamed scheme from Example \ref{ex:fractionalTaming}, where the nonlinearities are tamed smoothly rather than stopped discontinuously.

\begin{theorem}
\label{thm:transportFractional}
    Let the assumptions of Theorem \ref{thm:transportNLstop} hold and consider the fractionally tamed exponential Euler scheme $(U^j)_{0\le j\le N}$ with $F_k\ce a_k F$ and $G_k\ce a_k G$ for
    \begin{align*}
        a_k(u)&\ce \frac{1}{1+k^{1/2}R(u)^2},\quad 
        &R(u)&\ce \|F(u)\|_{H^s}+\|G(u)\|_{\calL_2(L^2,H^s)}\quad (u\in H^s(\T)).
    \end{align*}
    Then the fractionally tamed exponential Euler scheme
    converges pathwise uniformly with rate $\frac12$. Moreover, for all $u_0\in L_{\F_0}^{15q}(\Omega;H^s(\T))$ there is a constant $C\ge 0$, independent of $k$, $N$, and
    $u_0$, such that
    \begin{align*}
        \Big\|\max_{0\le j\le N} \|U^j-U(t_j)\|_{L^2(\T)} \Big\|_{L^p(\Omega)}
        &\le C\big(1+\|u_0\|_{L^{15q}(\Omega;H^s(\T))}^{15}\big)k^{1/2}.
    \end{align*}
\end{theorem}

\subsection{The nonlinear stochastic Schrödinger equation with dissipative damping}
\label{subsec:dissipativeNLSExample}

Consider the dissipatively damped stochastic nonlinear Schrödinger equation
\begin{align}\tag{DNLS}\label{eq:dissipativeNLS}
    \rmd U - i\Delta U \dt = -(\gamma+i\kappa)|U|^2U \dt + i\sigma U \dWQ\text{ on }[0,T],\quad U(0)=u_0 \in L^2(\T;\C)
\end{align}
on the one-dimensional torus $\T= \R/(2\pi\Z)$ for some $\gamma>0$, $\kappa,\sigma\in\R$, and a $Q$-Wiener process $(W_Q(t))_{t\ge 0}$ under further assumptions to be specified later. The dissipative damping is needed to ensure coercivity, which does not hold for the classical nonlinear Schrödinger equation with $F(u)=-i|u|^2u$. Hence, our analysis in this subsection is complimentary to the numerical analysis of the NLS equation in, for instance, \cite{deBouardDebussche06NLSconvOrder,ChenDangHong24,Cui25JDE}. The damped equation \eqref{eq:dissipativeNLS} can be recast in our framework via
\begin{align}
\label{eq:defYFGdissNLS}
    X &\ce L^2(\T;\C),\quad H\ce L^2(\T;\R),\quad Au\ce -i\Delta u,\quad u\in D(A)\ce H^2(\T;\C),\nonumber\\
    Y&\ce H^1(\T;\C),\quad F(v)\ce -(\gamma+i\kappa)|v|^2v,\quad G(v)\ce i\sigma M_v Q^{1/2}\quad (v\in Y).
\end{align}
The complex-valued spaces are regarded as real Hilbert spaces via the usual identification wherever necessary. In this subsection only, we abbreviate $L^p\ce L^p(\T;\C)$, $p\in [2,\infty]$, and $H^s \ce H^s(\T;\C)$, $s>0$. Fix an orthonormal basis $(e_n)_{n\in\N}$ of $L^2(\T;\R)$.

Since $-A$ generates the unitary Schrödinger group on both $X$ and $Y$ and $Y=D(A^{1/2})$, we verify Assumption \ref{ass:locLipPolGro} with $\alpha=\frac12$. To show local Lipschitz continuity, we observe that $G$ is linear and that $|N(w)-N(z)|\le C (|w|^2+|z|^2)|w-z|$ for $w,z\in \C$. Thus, repeating some of the arguments from the previous subsection such as $H^1\hra L^\infty$, we have
\begin{align*}
    \|F(u)-F(v)\|_X &\le C_{\gamma,\kappa}\big(\|u\|_{L^\infty}^2+\|v\|_{L^\infty}^2\big)\|u-v\|_{L^2} \le C_{\gamma,\kappa}\big(\|u\|_{H^1}^2+\|v\|_{H^1}^2\big)\|u-v\|_{L^2}\\
    \nn G(u)-G(v)\nn_X &\le |\sigma| Q_\infty \|u-v\|_{L^2},\qquad Q_\infty \ce \Big(\sum_{n\in\N} \|Q^{1/2}e_n\|_{L^\infty(\T;\R)}^2\Big)^{1/2}.
\end{align*}
The assumption $Q^{1/2}\in\calL_2(L^2(\T;\R),H^1(\T;\R))$ in particular ensures that $Q_\infty<\infty$. Hence, local Lipschitz continuity holds with $\nu=2$ for any $\gamma,\kappa,\sigma\in\R$ and $G$ is even globally Lipschitz continuous. Due to $\partial_x N(u) = 2|u|^2\partial_x u + u^2\overline{\partial_x u}$ and thus $\|\partial_x N(u)\|_{L^2} \le 3\|u\|_{L^\infty}^2\|\partial_x u\|_{L^2}$, $F$ is of cubic growth on $Y$. Moreover, since $H^1$ is a Banach algebra, 
\begin{equation*}
    \nn G(u)\nn_Y = |\sigma| \Big(\sum_{n\in\N} \|uQ^{1/2}e_n\|_{H^1}^2\Big)^{1/2} \le C |\sigma| \nn Q^{1/2}\nn_{H^1(\T;\R)} \|u\|_{H^1}.
\end{equation*}
Hence, Assumption \ref{ass:locLipPolGro}($\frac12,2,3$) is satisfied.

Since $G$ is globally Lipschitz on $X$, it suffices to verify the monotonicity condition for $F$. The real derivative of $N$ is given by $DN(\xi)\chi = 2|\xi|^2\chi+\xi^2\overline{\chi}$ for $\xi,\chi\in\C$. Hence, writing $\xi \overline{\chi}=re^{i\varphi}$ in polar coordinates,
\begin{equation*}
    \overline{\chi}DN(\xi)\chi = 2 |\xi|^2|\chi|^2+\overline{\chi}\xi^2\overline{\chi} = 2|\xi\overline{\chi}|^2 + (\xi\overline{\chi})^2 = 2r^2 + r^2 e^{2i\varphi} = r^2(2+\cos(2\varphi))+ir^2 \sin(2\varphi),
\end{equation*}
from which we conclude via the trigonometric inequality $\sqrt{3}|\sin \vartheta|\le 2+\cos\vartheta$ for all $\vartheta\in\R$ that
\begin{equation}
\label{eq:ImBoundNLS}
    \Re (\overline{\chi} DN(\xi)\chi)\ge 0,\quad |\Im (\overline{\chi} DN(\xi)\chi)| = r^2|\sin(2\varphi)|\le \frac{1}{\sqrt{3}} \Re (\overline{\chi} DN(\xi)\chi).
\end{equation}
Let $w,z\in\C$ and take $\chi=w-z$. The fundamental theorem of calculus then gives
\begin{equation*}
    \Re \big(\overline{w-z}(N(w)-N(z))\big) \ge 0,\quad \big|\Im \big(\overline{w-z}(N(w)-N(z))\big)\big| \le \frac{1}{\sqrt{3}} \Re \big(\overline{w-z}(N(w)-N(z))\big).
\end{equation*}
Hence, under the condition $\gamma\ge\frac{1}{\sqrt{3}}|\kappa|$, monotonicity as in Assumption \ref{ass:monqEta}($q,\eta$) for arbitrary $q\ge 2$, $\eta>0$ holds because of
\begin{equation*}
    \Re \la u-v,F(u)-F(v)\ra_X = \int_\T \big(-\gamma \Re \big(\overline{u-v}(N(u)-N(v))\big) +\kappa \Im \big(\overline{u-v}(N(u)-N(v))\big)\dx \le 0.
\end{equation*}
Since $G$ is of linear growth on $Y$, it suffices to verify the coercivity condition from Assumption \ref{ass:coYq} for $F$. By definition of $F$ and \eqref{eq:ImBoundNLS}, the leading term of the scalar product can be estimated by
\begin{align*}
    \Re \la \partial_x u, \partial_x F(u)\ra_{L^2} &= - \Re \int_\T (\gamma+i\kappa)\overline{\partial_x u} DN(u)\partial_x u \dx\\
    &= -\int_\T \big(\gamma \Re(\overline{\partial_x u} DN(u)\partial_x u) -\kappa \Im(\overline{\partial_x u} DN(u)\partial_x u)\big)\dx\\
    &\le  -\Big(\gamma-\frac{|\kappa|}{\sqrt{3}}\Big)\int_\T \Re\big(\overline{\partial_x u} DN(u)\partial_x u\big) \dx \le 0.
\end{align*}
Combined with
\begin{equation*}
    \Re \la u, F(u)\ra_{L^2} = - \Re \int_\T (\gamma+i\kappa)\overline{u} N(u) \dx
    = -\gamma \int_\T |u|^4\dx = -\gamma\|u\|_{L^4}^4\le 0,
\end{equation*}
this shows Assumption \ref{ass:coYq}($r$) for any $r\ge 2$, $\sigma\in\R$, and $Q^{1/2}\in\calL_2(L^2(\T;\R),H^1(\T;\R))$ under the condition that $\gamma\ge\frac{1}{\sqrt{3}}|\kappa|$.

\begin{theorem}
\label{thm:dissipativeNLSConvergenceStopNL}
    Let $p\in[2,\infty)$, $\lambda\in(0,1)$, $q=p/\lambda$, and $Y,F,G$ as in \eqref{eq:defYFGdissNLS}. Suppose that $\gamma>0$, $\kappa\in\R$ satisfy $\gamma\ge \frac{|\kappa|}{\sqrt{3}}$, $\sigma\in\R$, and $Q^{1/2}\in\calL_2(L^2(\T;\R),H^1(\T;\R))$. Then for all $u_0\in L_{\F_0}^{15q}(\Omega;H^1(\T;\C))$, \eqref{eq:dissipativeNLS} admits a unique global mild solution $U$ and the
    nonlinearities-stopped exponential Euler scheme $(U^j)_{0\le j\le N}$ converges pathwise uniformly at rate $\frac12$. That is, for some constant $C$ independent of $k$, $N$, and $u_0$,
    \begin{align*}
        \Big\|\max_{0\le j\le N} \|U^j-U(t_j)\|_{L^2(\T;\C)}\Big\|_{L^p(\Omega)}
        &\le C\big(1+\|u_0\|_{L^{15q}(\Omega;H^1(\T;\C))}^{15}\big)k^{1/2}.
    \end{align*}
    The same estimate holds if $(U^j)_{0\le j \le N}$ is the fractionally tamed exponential Euler scheme with $F_k\ce a_k F$, $G_k\ce a_kG$, and
    \begin{equation*}
        a_k(u)\ce \big(1+k^{1/2}(\|F(u)\|_{H^1(\T;\C)}+\|G(u)\|_{\calL_2(L^2(\T;\R),H^1(\T;\C))})^2\big)^{-1},\quad u\in H^1(\T;\C).
        \end{equation*}
    Moreover, for every $u_0\in L_{\F_0}^{12q}(\Omega;H^1(\T;\C))$, the tamed exponential Euler scheme $(V^j)_{0\le j \le N}$ given by \eqref{eq:defTamedSchemeUj} with $G_k\ce G$ and either $F_k(u)\ce \1_{\{\|F(u)\|_{H^1}\le k^{-1/2}\}}F(u)$ or $F_k(u)\ce (1+k^{1/2}\|F(u)\|_{H^1})^{-1}F(u)$ for $u\in H^1(\T;\C)$ converges pathwise uniformly at rate $\frac12$ and satisfy
      \begin{align*}
        \Big\|\max_{0\le j\le N} \|V^j-U(t_j)\|_{L^2(\T;\C)}\Big\|_{L^p(\Omega)}
        &\le C\big(1+\|u_0\|_{L^{12q}(\Omega;H^1(\T;\C))}^{12}\big)k^{1/2}.
    \end{align*}
\end{theorem}
\begin{proof}
    Theorem \ref{thm:convergenceMain} applies with $\alpha=\frac{1}{2}$, $\nu=2$, $\rho=3$, and $\theta=\frac14$, which proves the first assertion. Linear growth of $G$ on $Y$ allows us to apply Corollary \ref{cor:optimalSeparateTaming} with $\theta_F=\frac12$ and $\theta_G=0$, resulting in the third claim with $r=q(\rho_F(1+2\alpha)+\rho\nu)=12q$. The second and the last assertion follow from Theorem \ref{thm:convergenceMainTamed} as outlined in Example \ref{ex:fractionalTaming}.
\end{proof}

If $\gamma >\frac{1}{\sqrt{3}}|\kappa|$, the remaining negative terms in the monotonicity and coercivity estimates can be used to compensate for small quadratic noise following the ideas from the previous subsection.

\subsection{The stochastic Klein--Gordon equation with nonlinear velocity damping}
\label{subsec:kleinGordon}

As a further second-order example, we consider a stochastic Klein--Gordon equation with nonlinear damping in the velocity variable. Let $\calO\seq \R$ be a bounded interval and set $\Lambda \ce I-\Delta$ on $D(\Lambda)=H^2(\calO)\cap H_0^1(\calO)$, where $\Delta$ denotes the Dirichlet Laplacian on $L^2(\calO)$. Consider
\begin{align}\label{eq:KleinGordon} \tag{dKG}
    \Bigg\{\begin{split}
        \rmd u&= v\dt,\quad u(0)=u_0\\
        \rmd v + \Lambda u \dt &= (v-v^3)\dt + \sigma v \dWQ,\quad v(0)=v_0,
    \end{split}
\end{align}
where $\sigma\in \R$ and $(W_Q(t))_{t\ge 0}$ is a $Q$-Wiener process. We work with the phase spaces $X\ce D(\Lambda^{1/2})\times L^2(\calO)=H_0^1(\calO)\times L^2(\calO)$ and $Y\ce D(\Lambda) \times D(\Lambda^{1/2})$ equipped with the norms defined by
\begin{equation*}
    \|(u,v)\|_X^2=\|\Lambda^{1/2}u\|_{L^2}^2+\|v\|_{L^2}^2,\quad \|(u,v)\|_Y^2=\|\Lambda u\|_{L^2}^2+\|\Lambda^{1/2}v\|_{L^2}^2
\end{equation*}
where here and in the following we abbreviate $L^p\ce L^p(\calO)$, $p\in [2,\infty]$, and $H^s\ce H^s(\calO)$, $s>0$. The norms are equivalent to the ones given by $(\|u\|_{H^1}^2+\|v\|_{L^2}^2)^{1/2}$ and $(\|u\|_{H^2}^2+\|v\|_{H^1}^2)^{1/2}$, respectively. The phase spaces are constructed precisely such that the operator $A$ given by $A(u,v)\ce (-v,\Lambda u)$ for $(u,v)\in Y$ has domain $D(A)=Y$. Since $A$ is skew-adjoint on $X$, Stone's theorem~\cite[II.3.24]{EngelNagel} yields that $-A$ generates a unitary group and thus, in particular, a contractive semigroup $(S(t))_{t\ge 0}$ on $X$. The semigroup leaves $Y=D(A)$ invariant and commutes with $A$, so that by definition of the $Y$-norm
\begin{align*}
    \|S(t)U\|_Y=\|AS(t)U\|_X=\|S(t)AU\|_X=\|AU\|_X=\|U\|_Y,
\end{align*}
i.e.\ $S$ is also contractive on $Y$. Due to $Y=D(A)\hra D_A(\frac12,\infty)$, we may take $\alpha=\frac12$.

Let $F(U)\ce (0,f(v))\ce (0,v-v^3)$ for $U=(u,v)$ and $G(U)\ce (0,g(v)) \ce (0,\sigma M_v Q^{1/2})$, and set $H\ce L^2(\calO)$. The estimates for the nonlinearities are obtained as in the preceding subsections, observing that boundedness of $\calO$ yields that $H^1\hra L^\infty$ and $H^1$ is a Banach algebra. Using the equivalent norms and reducing from the phase space to the velocity coordinate, we then deduce
\begin{align*}
    \|F(U)\|_Y=\|f(v)\|_{H^1}=\|v-v^3\|_{H^1}\le C(1+\|v\|_{H^1}^3)\le C(1+\|U\|_Y^3).
\end{align*}
The noise $G$ is linear and thus globally Lipschitz on $X$ for sufficiently smooth $Q^{1/2}$. Since $U=(u,v)\in Y$ implies $v\in H_0^1(\calO)$, the noise term $g(v)$ also satisfies Dirichlet boundary conditions. By the Banach algebra property, $G$ thus maps $Y$ into $\LHY$ and is of linear growth between these spaces.
In summary, Assumption \ref{ass:locLipPolGro}($\frac12,2,3$) holds provided that $Q^{1/2}\in \calL_2(L^2,H^1)$. Coercivity follows from the arguments in the previous subsections via
\begin{align*}
    \la U,F(U)\ra_Y &= \la \Lambda^{1/2} v, \Lambda^{1/2}f(v)\ra_{L^2} = \la v, (v-v^3)\ra_{L^2} + \la \partial_x v, \partial_x(v-v^3)\ra_{L^2}\\
    &\le \|v\|_{L^2}^2+\|\partial_x v\|_{L^2}^2 = \|\Lambda^{1/2}v\|_{L^2}^2 \le \|U\|_Y^2.
\end{align*}
Likewise, the monotonicity argument can be reduced to previous $L^2$-estimates.

\begin{theorem}
\label{thm:KleinGordonConvergence}
    Let $\calO\seq\R$ be a bounded interval, let $p\in[2,\infty)$, $\lambda\in(0,1)$, $q=p/\lambda$, $\sigma\in\R$, and $Y\ce (H^2(\calO)\cap H_0^1(\calO))\times H_0^1(\calO)$. 
    Suppose that $Q^{1/2}\in\calL_2(L^2(\calO),H^1(\calO))$ and set $F_k(U)\ce (0,(1+k^{1/2}\|v-v^3\|_{H^1})^{-1}(v-v^3))$ for $U=(u,v)\in Y$ and $G_k\ce G$. Then for all $U_0=(u_0,v_0)$ with $u_0\in L_{\F_0}^{12q}(\Omega;H^2(\calO)\cap H_0^1(\calO))$ and $v_0\in L_{\F_0}^{12q}(\Omega; H_0^1(\calO))$, \eqref{eq:KleinGordon} admits a unique global mild solution and the
    tamed exponential Euler scheme $(U^j)_{0\le j\le N}$ converges pathwise uniformly at rate $\frac12$ with 
    \begin{align*}
        \Big\|\max_{0\leq j\le N} \|U^j-U(t_j)\|_{H^1(\calO)\times L^2(\calO)}\Big\|_{L^p(\Omega)}
        &\le C\big(1+\|U_0\|_{L^{12q}(\Omega;Y)}^{12}\big)k^{1/2}
    \end{align*}
    for some $C\ge 0$ independent of $k$, $N$, and $U_0$. The same estimate holds if $(U^j)_j$ is the tamed exponential Euler scheme with $F_k(U)\ce \1_{\{\|v-v^3\|_{H^1}\le k^{-1/2}\}}(0,v-v^3)$ and $G_k\ce G$.
\end{theorem}
\begin{proof}
    This is a consequence of Theorems \ref{thm:aprioriY} and \ref{thm:convergenceMainTamed} as well as Example \ref{ex:fractionalTaming} with parameters $\alpha=\frac12$, $\nu=2$, $\rho=3$, $\theta_F=\frac12$, $\theta_G=0$, $\beta=6$, and $r=(\beta+\rho\nu)q=12q$.
\end{proof}

As for the dissipative NLS equation, analogous results hold for the nonlinearities-stopped exponential Euler scheme, i.e.\ stopping both $F$ and $G$ rather than just the drift, under the stronger moment condition $U_0\in L_{\F_0}^{15q}(\Omega;Y)$. Via the arguments from Subsection \ref{subsec:transportExample}, also quadratic noise of the form $G(U)=g(v)=\sigma M_{v|v|}Q^{1/2}$ could be dealt with under a smallness condition on $|\sigma|$.

In the additive noise case, convergence rates for full discretisations of a closely related nonlinearly damped stochastic wave equation in dimensions $d\in \{1,2\}$ were studied in \cite{CaiCohenWang25}. For a modified implicit exponential Euler method combined with a spectral Galerkin approximation in space, they obtain mean-square convergence rate $1$ in time and $\frac12$ in space w.r.t.\ the $X$-norm in dimension $1$, cf.\ \cite[Thms.~4,~5]{CaiCohenWang25}. Theorem \ref{thm:KleinGordonConvergence} is the counterpart of the temporal result in $d=1$ for multiplicative noise, where the natural limitation to rate $\frac12$ occurs. Moreover, we obtain error estimates for the stronger pathwise uniform error for a class of explicit schemes. Structure-preserving semi-implicit SAV schemes for stochastic Klein--Gordon equations with nonlinearities depending only on the displacement variable and additive noise were shown to converge at rate $1$ in a mean-square sense in \cite[Thm.~5.1]{CuiHongSun25KleinGordon}. 

\subsection{The nonlinear stochastic Airy equation}
\label{subsec:AiryExample}

Lastly, we investigate a higher-order example with a polynomial nonlinearity of arbitrary odd order. The Airy equation arises from the Korteweg--de Vries equation by linearisation and the leading operator is of third order. As illustrated below, the latter leads to suboptimal rates in our current setting. Coercivity as in Assumption \ref{ass:coYq} can only be verified on $H^1$, since in higher-order Sobolev or Bessel potential spaces, superlinear terms without a negative sign arise from the product rule. As a third-order equation, $D(A)=H^3$ and thus $Y=H^1=D(A^{1/3})$, which limits the rate to $\alpha\le \frac{1}{3}$. Achieving the optimal rate $\alpha=\frac12$ for multiplicative noise would require a generalised coercivity assumption, which is left for future work.

Consider the real-valued Airy equation with a generalised Allen--Cahn nonlinearity
\begin{align}\tag{NAIR}
\label{eq:airy}
    \rmd u+\kappa\partial_x^3u \dt
    &=(u-u^{2m+1})\dt +\sigma u\dWQ\text{ on }[0,T], \quad u(0)=u_0,
\end{align}
where $\kappa\in\R\setminus\{0\}$, $m\in\N$, and $(W_Q(t))_{t\ge 0}$ is a $Q$-Wiener process on $L^2(\T)$. Abbreviate $L^p\ce L^p(\T)$, $p\in [2,\infty]$, and $H^1\ce H^1(\T)$ in this subsection.
We consider $X=L^2$,
\begin{equation}
\label{eq:YFGAiry}
    Y=H^1,\quad F(u)=u-u^{2m+1},\quad G(u)=\sigma M_u Q^{1/2}\quad (u\in Y).
\end{equation}

\begin{theorem}
\label{thm:airyConvergence}
    Let $p\in[2,\infty)$, $\lambda\in(0,1)$, $q=p/\lambda$, $\kappa\in\R\setminus\{0\}$, $m\in \N$, and set $r\ce(2m+1)(2m+\frac{7}{3})q$. 
    Suppose that $Q^{1/2}\in\calL_2(L^2,H^1)$. Then for all $u_0\in L_{\F_0}^r(\Omega;H^1(\T))$, \eqref{eq:airy} admits a unique global mild solution. Moreover, the
    nonlinearities-stopped exponential Euler scheme $(U^j)_{0\le j\le N}$ with $Y,F,G$ as in \eqref{eq:YFGAiry} has pathwise uniform convergence rate $\frac13$ and satisfies
    \begin{align*}
        \Big\|\max_{0\leq j\le N} \|U^j-U(t_j)\|_{L^2(\T)}\Big\|_{L^p(\Omega)}
        &\le C\big(1+\|u_0\|_{L^r(\Omega;H^1(\T))}^{r/q}\big)k^{1/3}
    \end{align*}
    for some $C\ge 0$ independent of $k$, $N$, and $u_0$. If $(U^j)_j$ is the tamed exponential Euler scheme with $F_k(u)\ce \1_{\{\|F(u)\|_{H^1}\le k^{-1/2}\}}F(u)$ and $G_k\ce G$, the estimate holds under the assumptions with $r\ce(2m+1)(2m+\frac{5}{3})q$.
\end{theorem}
\begin{proof}
    The well-posedness follows from Theorem \ref{thm:aprioriY} and the convergence claims from Theorems \ref{thm:convergenceMain} and \ref{thm:convergenceMainTamed} together with Example \ref{ex:fractionalTaming} with $\alpha=\frac{1}{3}$, $\nu=2m$, $\rho=2m+1$, and $\theta=\frac14$ or $\theta_F=\frac12$ and $\theta_G=0$, respectively. Regarding their applicability, we only comment on the parts which are novel compared to the previous subsections.

    The operator $-A$ given by $Au=\kappa \partial_x^3 u$ for $u\in D(A)=H^3(\T)$ generates a unitary group on both $L^2$ and $H^1$ by Stone's theorem~\cite[II.3.24]{EngelNagel}. Hence, $Y\hra D(A^{1/3})\hra D_A(\frac13,\infty)$, meaning that we can take $\alpha=\frac13$. For local Lipschitz continuity, we observe that the factorisation
    \begin{equation*}
        v^{2m+1}-u^{2m+1}=-(u-v)(u^{2m}+u^{2m-1}v+\ldots+v^{2m})
    \end{equation*}
    combined with Hölder's inequality and the embedding $H^1\hra L^\infty$ as before results in $\nu=2m$. The same factorisation implies monotonicity via $\la u-v,v^{2m+1}-u^{2m+1}\ra_{L^2}\le 0$ and linearity of the remaining terms. Polynomial growth of $F$ with $\rho=2m+1$ is a consequence of the Banach algebra property of $Y=H^1$. Lastly, coercivity holds for any $r\ge 2$ by linear growth of $G$ on $Y$ and, for $u\in H^1$,
    \begin{align*}
        \la u,F(u)\ra_{H^1}&=\la u,u\ra_{H^1}-\la u,u^{2m+1}\ra_{L^2}-\la \partial_xu,(2m+1)u^{2m}\partial_x u\ra_{L^2}\\
        &=\|u\|_{H^1}^2-\|u\|_{L^{2m+2}}^{2m+2}-(2m+1)\|u^{m}\partial_x u\|_{L^2}^2 \le \|u\|_{H^1}^2. \qedhere
    \end{align*}
\end{proof}

Similar results can be obtained for fractional tamings for the Airy equation as in Theorem \ref{thm:transportFractional}.
For the cubic Allen--Cahn nonlinearity, Theorem \ref{thm:airyConvergence} gives $r=13q$ and $r=11q$ for the nonlinearities-stopped exponential Euler scheme and the tamed version stopping only $F$, respectively. In the case of a quintic nonlinearity, one has $r=\frac{95}{3}q$ and $r=\frac{85}{3}q$, respectively. Higher-order nonlinearities thus translate to higher moment assumptions on the initial values and, in the case of quadratic noise as in Subsection \ref{subsec:transportExample}, a more restrictive smallness condition. 

\begin{remark}
    Motivated by the rate restriction $\alpha\le \frac13$ arising from the choice $Y=H^1$, one is tempted to choose a more regular space like $Y=H^2$ to reach the optimal rate $\frac12$. However, already for the cubic Allen--Cahn nonlinearity, integration by parts for the leading term of the $H^2$-scalar product gives
    \begin{align*}
        \la \partial_x^2 u,\partial_x^2 F(u)\ra_{L^2} &= \la \partial_x^2 u,\partial_x (\partial_x u-3u^2\partial_xu)\ra_{L^2} = \la \partial_x^2 u,\partial_x^2 u-6u(\partial_x u)^2-3u^2\partial_x^2u\ra_{L^2}\\
        &= \|\partial_x^2 u\|_{L^2}^2+2\|\partial_x u\|_{L^4}^4-3\|u\partial_x^2u\|_{L^2}^2.
    \end{align*}
    The quartic growth and the positive sign of the second term do not yield a bound in terms of $C(1+\|u\|_{H^2}^2)$, so coercivity in $H^2$ seems out of reach. This may be circumvented via a weaker coercivity condition, which is left for future work. 
\end{remark}

\def\polhk#1{\setbox0=\hbox{#1}{\ooalign{\hidewidth
			\lower1.5ex\hbox{`}\hidewidth\crcr\unhbox0}}} \def\cprime{$'$}

\end{document}